\documentclass[a4paper,reqno]{amsart}

\usepackage[T1]{fontenc}
\usepackage[utf8]{inputenc}
\usepackage{lmodern}
\usepackage[english]{babel}
\usepackage[margin=2.0cm]{geometry}

\usepackage{amsmath,amsthm,amssymb,mathrsfs,yhmath,stmaryrd}

\usepackage{bbold,yfonts}
\usepackage[bbgreekl]{mathbbol}

\usepackage{mathtools}
\usepackage{amsfonts}
\usepackage{enumerate}
\usepackage{enumitem}
\usepackage{hyperref}
\usepackage{colonequals,bbm}
\DeclareRobustCommand\longtwoheadrightarrow
	{\relbar\joinrel\twoheadrightarrow}
\DeclareRobustCommand\longhookrightarrow
	{\lhook\joinrel\longrightarrow}

\usepackage{tikz}
\usetikzlibrary{cd}

\newcommand*\encircle[1]{\tikz[baseline=(char.base)]{
		\node[shape=circle,draw,inner sep=1pt] (char) {#1};}}

\def\barroman#1{\sbox0{#1}\dimen0=\dimexpr\wd0+1pt\relax
	\makebox[\dimen0]{\rlap{\vrule width\dimen0 height 0.06ex depth 0.06ex}%
	\rlap{\vrule width\dimen0 height\dimexpr\ht0+0.03ex\relax 
		depth\dimexpr-\ht0+0.09ex\relax}%
	\kern.5pt#1\kern.5pt}}

\newcommand{\N}{\mathbb{N}}

\newcommand{\R}{\mathbb{R}}

\newcommand{\Hq}{\mathbb{H}}
\newcommand{\bk}{\mathbbm{k}}
\newcommand{\Spenc}{\mathcal{S}}

\newcommand{\Mod}{\operatorname{Mod}}
\newcommand{\AMod}{{}_A\!\Mod}

\newcommand{\ModA}{\Mod_A}
\newcommand{\ModB}{\Mod_B}
\newcommand{\AModA}{{}_A\!\Mod_A}
\newcommand{\AModB}{{}_A\!\Mod_B}

\newcommand{\Diff}{\operatorname{Diff}}
\newcommand{\EDiff}{\operatorname{\check{Diff}}}

\newcommand{\NatDiff}{\mathcal{D}\!i\!f\!f}
\newcommand{\Symb}{\operatorname{Symb}}
\newcommand{\symb}{\varsigma}

\newcommand{\phat}{\widehat{p}}

\newcommand{\EJ}{\check{J}}
\newcommand{\Ej}{\check{j}}
\newcommand{\Ephat}{\check{p}}
\newcommand{\Epi}{\check{\pi}}

\newcommand{\El}{\check{l}}

\newcommand{\ESpenc}{\check{\Spenc}}

\newcommand{\PJ}{\breve{J}}

\newcommand{\Pphat}{\breve{p}}
\newcommand{\Ppi}{\breve{\pi}}

\newcommand{\Pl}{\breve{l}}

\newcommand{\PSpenc}{\breve{\Spenc}}

\newcommand{\itJ}{\dot{J}}

\newcommand{\itpi}{\dot{\pi}}

\newcommand{\itl}{\dot{l}}
\newcommand{\itSpenc}{\dot{\Spenc}}

\newcommand{\sqhiota}[1]{t^{\{{#1}\}}_d}

\numberwithin{equation}{section}

\allowdisplaybreaks

\makeatletter
\renewcommand{\tocsection}[3]{%
	\indentlabel{\@ifnotempty{#2}{\bfseries\ignorespaces#1 #2\quad}}\bfseries#3}
\renewcommand{\tocsubsection}[3]{%
	\indentlabel{\@ifnotempty{#2}{\ignorespaces#1 #2\quad}}#3}

\newcommand\@dotsep{4.5}
\def\@tocline#1#2#3#4#5#6#7{\relax
	\ifnum #1>\c@tocdepth 
	\else
	\par \addpenalty\@secpenalty\addvspace{#2}%
	\begingroup \hyphenpenalty\@M
	\@ifempty{#4}{%
		\@tempdima\csname r@tocindent\number#1\endcsname\relax
	}{%
		\@tempdima#4\relax
	}%
	\parindent\z@ \leftskip#3\relax \advance\leftskip\@tempdima\relax
	\rightskip\@pnumwidth plus1em \parfillskip-\@pnumwidth
	#5\leavevmode\hskip-\@tempdima{#6}\nobreak
	\leaders\hbox{$\m@th\mkern \@dotsep mu\hbox{.}\mkern \@dotsep mu$}\hfill
	\nobreak
	\hbox to\@pnumwidth{\@tocpagenum{\ifnum#1=1\bfseries\fi#7}}\par
	\nobreak
	\endgroup
	\fi}
\AtBeginDocument{%
\expandafter\renewcommand\csname r@tocindent0\endcsname{0pt}
}
\def\l@subsection{\@tocline{2}{0pt}{2.5pc}{5pc}{}}
\makeatother

\makeatletter
\@namedef{subjclassname@2020}{\textup{2020} Mathematics Subject Classification}
\makeatother

\newcommand{\Hom}{\operatorname{Hom}}
\newcommand{\AHom}{{}_A\!\Hom}

\newcommand{\DO}{\mathcal{D}}

\DeclareMathOperator*{\coker}{co{\ker}}

\newcommand{\End}{\operatorname{End}}

\newcommand{\Tor}{\operatorname{Tor}}

\newcommand{\Sym}{s}

\newcommand{\smooth}[1]{\mathcal{C}^{\infty}\!{({#1})}}

\newcommand{\id}{\mathrm{id}}

\newcommand{\im}{\mathrm{Im}}

\newcommand{\ev}{\operatorname{ev}}

\newcommand{\trunc}[1]{\Lbag\!{#1}\!\Rbag}

\newcommand{\starhat}{\mathbin{\hat{\star}}} 
\newcommand{\circhat}{\mathbin{\hat{\circ}}} 

\renewcommand{\Re}{\operatorname{Re}}

\theoremstyle{definition}

\newtheorem{defi}{Definition}[section]
\newtheorem{eg}[defi]{Example}

\theoremstyle{plain}

\newtheorem{theo}[defi]{Theorem}
\newtheorem{prop}[defi]{Proposition}
\newtheorem{cor}[defi]{Corollary}
\newtheorem{lemma}[defi]{Lemma}

\newtheorem*{theo*}{Theorem}
\newtheorem*{prop*}{Proposition}
\newtheorem*{cor*}{Corollary}
\newtheorem*{lemma*}{Lemma}

\theoremstyle{remark}
\newtheorem{rmk}[defi]{Remark}

\newcommand{\jetsstorderwrtd}{Proposition~4.6, p.~15}
\newcommand{\jetscorspencerdeltajes}{Corollary~8.31, p.~53}
\newcommand{\jetsrmkrhodo}{Remark~2.10, p.~8}
\newcommand{\jetsdefnjet}{Definition~8.1, p.~41}
\newcommand{\jetsdefpiiotan}{Definition~8.8, p.~44}
\newcommand{\jetsdefrho}{(2.13), p.~8}
\newcommand{\jetsdefsesqui}{Definition~8.21, p.~48}
\newcommand{\jetsdefshj}{Definition~5.22, p.~22}
\newcommand{\jetsdefisesquisequence}{Definition~8.23, p.~48}
\newcommand{\jetseqDHI}{(5.47), p.~24}
\newcommand{\jetseqdecsjn}{(5.53), p.~25}
\newcommand{\jetseqdefdeltahk}{(6.30), p.~34}
\newcommand{\jetseqdefinitionNdE}{(2.19), p.~9}
\newcommand{\jetseqiotajn}{(8.2), p.~42}
\newcommand{\jetseqiotaprimestuff}{(5.46), p.~24}
\newcommand{\jetseqDHnaturalDO}{(7.31), p.~40}
\newcommand{\jetseqnhpink}{(5.2), p.~18}
\newcommand{\jetseqrhon}{(2.30), p.~10}
\newcommand{\jetslemmaecdcurvaturelinear}{Lemma~6.9, p.~30}
\newcommand{\jetslemmaextwedge}{Lemma~8.27, p.~51}
\newcommand{\jetslemmaholprol}{Lemma~8.13, p.~45}
\newcommand{\jetslemmaimageD}{Lemma~8.26, p.~50}
\newcommand{\jetsproponejses}{Proposition~2.19, p.~10}
\newcommand{\jetspropconnexionsplits}{Proposition~4.10, p.~17}
\newcommand{\jetspropdifferentialoperatorcomposition}{Proposition~10.3, p.~58}
\newcommand{\jetspropfunctorialtwojetseq}{Proposition~7.15, p.~41}
\newcommand{\jetsprophjc}{Proposition~8.11, p.~44}
\newcommand{\jetspropholinsemi}{Proposition~8.18, p.~46}
\newcommand{\jetsproplowspencercohom}{Proposition~6.24, p.~35}
\newcommand{\jetspropoperatorsalsohigherorder}{Proposition~10.2, p.~58}
\newcommand{\jetspropsemiholcpxfact}{Proposition~8.19, p.~47}
\newcommand{\jetsproptensorcomparison}{Proposition~8.7, p.~43}
\newcommand{\jetsrmkholprolpi}{Remark~8.15, p.~45}
\newcommand{\jetsrmkiotaholinj}{Remark~8.3, p.~42}
\newcommand{\jetsrmklowsemihol}{Remark~5.23, p.~22}
\newcommand{\jetsrmktensorNddiffop}{Remark~4.7, p.~16}
\newcommand{\jetsrmkweakOthreeflat}{Remark~8.29, p.~52}
\newcommand{\jetstheosjchar}{Theorem~5.36.(ii), p.~24}
\newcommand{\jetssdifferentialoperators}{§10}
\newcommand{\jetsssSpencer}{§6.3}
\newcommand{\jetssssSplitting}{§2.2.1, p.~8}
\newcommand{\jetssonejetfunctor}{§2}
\newcommand{\jetsssfunctorialityoftwojet}{§7.3}
\newcommand{\jetsssquaternions}{§10.2}
\newcommand{\jetsdiagdefiiotand}{(8.16), p.~44}
\newcommand{\jetsdefnonholjetfunctor}{Definition~5.1, p.~18}
\newcommand{\jetslemmaedhinkerspencer}{Lemma~8.28, p.~51}
\newcommand{\jetstheohigherwolves}{Theorem~8.30, p.~52}
\newcommand{\jetsdefiSterminal}{Definition~4.14, p.~17}
\newcommand{\jetsrmkproldop}{Remark~2.20, p.~10}
\newcommand{\jetsrmkLksecondorderDO}{Remark~10.16, p.~60}
\newcommand{\jetseqquaternionpartialderivatives}{(10.17), p.~60}
\newcommand{\jetseqquaternionLaplacian}{(10.21), p.~61}
\newcommand{\symbolsssjetmodulesrepresentingobjects}{§2.1}
\newcommand{\symbolssselementaljets}{§3.1}
\newcommand{\symbolsssrepresentabilitysymbols}{§4.1}
\newcommand{\symbolscorsymmformsNN}{Corollary~2.10, p.~6}
\newcommand{\symbolsdefelementalnjetfunctor}{Definition~3.1, p.~10}
\newcommand{\symbolsdefphatn}{(2.1), p.~5}
\newcommand{\symbolsdefsymgenprolrelations}{Definition~2.8, p.~6}
\newcommand{\symbolseqskelprimitivecompare}{(3.43), p.~18}
\newcommand{\symbolslemmaelementalsmm}{Lemma~3.22, p.~15}
\newcommand{\symbolslemmaelementalsmmtwo}{Lemma~3.22.(ii), p.~15}
\newcommand{\symbolslemmaskelprimitivecompare}{Lemma~3.29, p.~18}
\newcommand{\symbolsdefsymbolquotientdef}{Definition~4.1, p.~21}
\newcommand{\symbolslemmaprimitivesmm}{Lemma~3.34, p.~21}
\newcommand{\symbolsdefrestrictionsymbol}{Definition~4.6, p.~22}
\newcommand{\symbolsremidentifyrestrictionsymbol}{Remark~4.7, p.~22}
\newcommand{\symbolseqphatnepimonocomponents}{(3.1), p.~10}
\newcommand{\symbolspropcharacterizationsymmgeneratedbyprol}{Proposition~4.4, p.~22}
\newcommand{\symbolspropcriterionWDO}{Proposition~3.18, p.~14}
\newcommand{\symbolspropcriterionWDOone}{Proposition~3.18.(i), p.~14}
\newcommand{\symbolspropcriterionWDOthree}{Proposition~3.18.(iii), p.~14}
\newcommand{\symbolspropelementaljetproperties}{Proposition~3.2, p.~10}
\newcommand{\symbolspropelementaljetpropertiesone}{Proposition~3.2.(i), p.~10}
\newcommand{\symbolspropelementaljetpropertiestwo}{Proposition~3.2.(ii), p.~10}

\newcommand{\symbolspropelementaljetpropertiesfour}{Proposition~3.2.(iv), p.~10}
\newcommand{\symbolspropsymbolsrepresentationwelldefined}{Proposition~4.5, p.~22}
\newcommand{\symbolspropsymbolsofconnections}{Proposition~4.11, p.~25}
\newcommand{\symbolspropquantisation}{Proposition~4.22, p.~28}
\newcommand{\symbolspropsymbolcomposition}{Proposition~4.18, p.~27}
\newcommand{\symbolsdefisymbolalgebra}{Definition~4.21, p.~28}

\begin{document}

\title{Higher Order Connections in Noncommutative Geometry}
\author{Keegan J.~Flood, Mauro Mantegazza, Henrik Winther}
\address{Faculty of Mathematics and Computer Science\\
	UniDistance Suisse\\
	Schinerstrasse 18\\
	3900 Brig\\
	Switzerland}
	\email{keegan.flood@unidistance.ch}
\address{Department of Mathematics and Physics\\
	Charles University\\
	Sokolovsk\'{a} 49/83\\
	186 75 Prague 8\\
	Czech Republic}
	\email{mauro.mantegazza.uni@gmail.com}
\address{Department of Mathematics and Statistics\\
	UiT - The Arctic University of Norway\\
	Hansine Hansens veg 18\\
	N-9019 Tromsø\\
	Norway}
	\email{henrik.winther@uit.no}



\begin{abstract}
We prove that, in the setting of noncommutative differential geometry, a system of higher order connections is equivalent to a suitable generalization of the notion of phase space quantization (in the sense of Moyal star products on the symbol algebra).
Moreover, we show that higher order connections are equivalent to (ordinary) connections on jet modules.
This involves introducing the notion of natural linear differential operator, as well as an important family of examples of such operators, namely the Spencer operators, generalizing their corresponding classical analogues.
Spencer operators form the building blocks of this theory by providing a method of converting between the different manifestations of higher order connections.
A system of such higher order connections then gives a quantization, by which we mean a splitting of the quotient projection that defines symbols as classes of differential operators up to differential operators of lower order.
This yields a notion of total symbol and of star product, the latter of which corresponds, when restricted to the classical setting, to phase space quantization in the context of quantum mechanics.
In this interpretation, we allow the analogues of the position coordinates to form a possibly noncommutative algebra.
\end{abstract}

\maketitle

\tableofcontents

\section{Introduction}
\label{s:introduction}
Naturality in differential geometry can be seen from two distinct perspectives.
The first is more traditional: invariance of quantities under change of coordinates or diffeomorphisms.
The second point of view comes from category theory, utilizing functors and natural transformations.
The relationship between these two points of view is treated in the book \cite{NaturalOperations}.
In particular, they take the categorical picture to be the more fundamental one, and it is shown that classically, the diffeomorphism-equivariant picture emerges from it.
The present situation in noncommutative differential geometry is such that we have a much better grasp on the categorical picture.
Hence, that shall also be our starting point.
We develop natural differential operators from that perspective in §\ref{s:Natural_differential_operators}.
This subject is of independent interest, but it also provides essential constraints on any prospective theory of noncommutative diffeomorphisms.

Having laid the groundwork for natural operators, we then generalize a family of such operators, called Spencer operators, to our noncommutative setting in §\ref{s:Spencer_operators_and_the_Spencer_complex}.
This family comprises some of the most central examples of natural differential operators in the classical context.

The Spencer operators emerge from the work of D.~Spencer (and his collaborators) on the geometric approach to partial differential operators (cf.\ \cite{Spencer,goldschmidt1967existence}).
These operators were originally developed to give a canonical method of reducing the formal study of systems of partial differential equations to the study of first-order systems.
The Spencer operators are thus commonly used as a geometric tool to determine consistency and involutivity of systems of partial differential equations.

Partial differential equations and jets are two sides of the same coin, and thus it comes as no surprise that the Spencer operators can be used to give an elegant characterization of the holonomic jets amongst semiholonomic and nonholonomic jets (on manifolds or sections of vector bundles).
In the present paper, this property is carried through into the noncommutative setting, cf.\ Lemma \ref{lemma:sesquiholonomic_Spencer}.

On the other hand, the Spencer operator approach also found influential applications in the theory of Lie groups, Lie pseudogroups, $G$-structures, and finite-type geometries.
For example, one can give an account of the space of $G$-adapted connections, and a separation of their curvature and torsion into those components that are intrinsic to the structure, and those that are incidental or arbitrary.
This means that Spencer operators have several other important applications to mathematical physics, which are not immediately linked to each other conceptually, for example, variational calculus in continuum mechanics involving both forces and couplings; and field-matter couplings in gauge theory, to mention a few (cf.\ \cite{Pommaret}).

In the formalism of variational calculus, the variational functionals themselves are typically (integrals of) functions on a jet bundle over some fibered manifold \cite{variationsjets}.
To formulate variational theory in this way, however, it is most convenient to be working in the setting where we have a splitting (in the category of vector bundles) of the jet projection $\pi^{n,n-1}\colon J^n \mathcal{E} \rightarrow J^{n-1}\mathcal{E}$, which is to say an identification of the $(n-1)$-jet bundle $J^{n-1}\mathcal{E}$ with a subbundle of $J^n \mathcal{E}$ (rather than just a quotient) (cf.\ \cite{anderson1992introduction}).
Geometrically, a splitting $C^n\colon J^{n-1} \mathcal{E}\hookrightarrow J^n \mathcal{E}$ encodes a \textit{higher order connection} (\emph{$n$-connection}), which is also an interesting geometric object, and has been subject of study for its own sake (cf.\ \cite{eastwood2009higher, HolomorphicHigherCon, libermann1964, Ehresmannconnectionsdordresup, virsik1967non, yuen1971higher}).
For example, an $n$-connection can be seen as inducing, and being induced by, a vector bundle connection (satisfying some properties) on the $(n-1)$-jet bundle.
The tool for switching between these two points of view on higher order connections turns out to be precisely the aforementioned Spencer operators, cf.\ Theorem \ref{theo:higher_connections_are_connections}.

There is a natural notion of curvature for a given higher order connection $C^n$, generalizing that of curvature for a connection.
This curvature can be seen as an obstruction to the integrability of $C^n$, in that the map $J^1 (C^n)\circ C^n$ takes values in holonomic jets precisely when the curvature vanishes (cf.\ \cite{libermann1997}).
In this spirit, we generalize this notion to our setting, cf.\ §\ref{ss:Curvature_of_higher_connections} and §\ref{ss:Relation_with_connections_on_jet_modules}.

In differential geometry, splittings of the jet sequence are also closely related to splittings of the corresponding algebras of linear differential operators of finite order into homogeneous components.
This is equivalent to finding a way to associate, to a given (principal) symbol of degree $n$, a differential operator of order $n$ with that symbol.
In other words, it amounts to finding a section of the symbol projection mapping differential operators to their symbol (cf.\ \cite[Theorem~6, p.~89]{PalaisVectorBundles} and \cite[Section~3, p.~235]{Lychagin1999}).
We will refer to such a map as an $n$-\emph{quantization}.
Similarly, by \emph{full quantization} we mean a direct sum of $n$-quantizations for all $n$, which is to say, a map from the symbol algebra to the algebra of differential operators, which, when restricted to any given order $n$, is a section of the corresponding symbol map (note that the classical quantization terminology is not standardized).

A full quantization yields a notion of total symbol for differential operators.
In the absence of this structure, the symbol of a differential operator (also called principal symbol, to distinguish it from the total symbol), only captures the term of leading order.
In the presence of a quantization, one obtains an element of the symbol algebra, called the \emph{total symbol}, that also captures the lower order terms.
Essentially, in a local chart, one may identify a linear differential operator with a (multivariate) polynomial with coefficients in $\smooth{\R^n}$.
A differential operator can be interpreted as the Fourier multiplier of the polynomial corresponding to its total symbol, cf.\ \cite[Example~3.1, p.~27]{microlocaldiffop}.
We generalize the theory of total symbols to our noncommutative setting in §\ref{s:Total_symbols}.

In a broad sense, the term quantization refers to any procedure which generalizes the “canonical quantization”.
The canonical quantization consists of replacing commuting position $q_i$ and momentum $p^i$ variables on $\R^{2n} \cong T^\ast \R^n$, with new operator variables
\begin{align}
	\label{eq:CanonicalQuantization}
	\hat q_i = L_{q_i},
	&\hfill&
	\hat p^i = -i\hbar \partial_{q_i},
\end{align}
which famously satisfy the canonical commutation relations $[\hat q_i, \hat p^j] = i \hbar \delta^j_i L_1$, where $L$ denotes the left multiplication operator.
These differential operators are then represented as unbounded linear operators on the Hilbert space $L^2(\R^n)$.
See \cite{quantizationguide} for many examples of quantizations in this broader sense.
We will note that the most common pattern is to promote a commutative algebra of functions (on a phase space) to an algebra of differential operators, and then to represent this as an algebra of unbounded linear operators on a Hilbert space.

The generalization of canonical quantization from $\R^n$ to a smooth manifold $M$ was made by I.~R.~Segal (cf.\ \cite{SegalQuantization}).
Here, functions on the cotangent space $T^\ast M$ are turned into differential operators on $\smooth{M}$.
Then one can equip $M$ with a measure, and consider the linear differential operators arising in this way as unbounded operators on the Hilbert space $L^2(M)$ (two different measures give equivalent results up to unitary transformations, cf.\ \cite[p.~474]{SegalQuantization}).

The first part of the former framework is developed in geometric language in \cite[Section~4.1]{Lychagin1999}.
Here, the algebra of smooth functions over $T^\ast M$ is studied via the dense subalgebra of fiberwise polynomial functions, and the symbol algebra of differential operators on $\smooth{M}$ is identified with this subalgebra.
Then, the mapping from functions to differential operators is implemented in terms of splittings of the symbol map.
This construction is further extended to include \emph{quantizations of mechanical systems equipped with inner structure}.
This involves a vector bundle $\mathcal{E}$, where the r\^ole of $\smooth{M}$ is played by $\Gamma(M,\mathcal{E})$, and that of the algebra $\smooth{T^\ast M}$ is played by $\smooth{T^\ast M, \End(\mathcal{E},\mathcal{E})}$ and is studied via the dense subalgebra corresponding to symbols of differential operators $\Gamma(M,\mathcal{E})\to \Gamma(M,\mathcal{E})$.

Our approach to quantization in the noncommutative setting proceeds in the spirit of the latter work, which we develop in §\ref{s:Quantization}.
Since, in noncommutative geometry, there is no clear analogue of smooth functions on the cotangent space, we will follow a similar approach and consider the symbol algebra $\Symb^\bullet_d(E,E)$ for a left $A$-module $E$ instead.
The notion of quantization thus generalizes naturally to our setting by considering sections of the corresponding symbol projections.
The ingredients used in \cite[Section~4.1]{Lychagin1999} turn out to be connections on vector bundles, (tensor) connections on symmetric forms with values in bundles, and, implicitly, the symmetric tensor product.
We develop noncommutative analogues of the aforementioned ingredients, and we show that they are sufficient to replicate this type of quantization procedure, cf.\ Corollary \ref{cor:full_quantisation} and Corollary \ref{cor:generalization_of_classical_quantization}.

Classically, the functions on the cotangent bundle that correspond to symbols of differential operators as in \cite{Lychagin1999} are necessarily polynomial in the fiber variables (cf.\ \cite[§5.4]{alekseevskij28geometry} and \cite[Proposition~10.12]{nestruev2020smooth}), unless pseudodifferential operators are also admitted, which would take us into the domain of microlocal and semiclassical analysis (cf.\ \cite{semiclassical, microlocaldiffop}), and far outside the scope of the present paper.
We note that quantizations in the sense of mapping symbols to (Fourier integral) operators have been developed also in this field (cf.\ \cite[Chapter~4, p.~27]{Hintz}, \cite[Chapter~4]{semiclassical}).

Although we can generalize the first step in the quantization procedure, of promoting functions to differential operators, the second step, of promoting differential operators to unbounded operators on a Hilbert space, would necessitate departing from our setting.
This is because, as we are working with generic algebras over commutative rings, there is no apparent analogue of the notion of Hilbert space available.
To overcome this, in §\ref{ss:starproduct}, we construct a noncommutative generalization of an alternative classical notion, called phase space quantization, which arose from the works \cite{groenewold1946, moyal1949quantum}.
In the modern classical formulation (absent from the two original works) one equips the algebra of functions on the phase space with a new product, such that the algebra generated by the relevant classical observables (e.g.\ $q_i$ and $p^i$), satisfies the appropriate commutation relations (cf.\ \cite{originalstarproduct}).
This product is often denoted by $\star$, and is called the \emph{star} (or sometimes \emph{Moyal}) \emph{product}.
This enters the larger realm of deformation quantization (cf.\ e.g.\ \cite{sternheimer1998deformation}).

In summary, this work contributes to the theoretical foundations of noncommutative differential geometry by providing tools with applications in both mathematics and physics.
We point out a link between the geometric theory of higher order connections and the theory of quantization, extending these ideas to the noncommutative setting.
This relation is not only relevant to noncommutative geometry but also of interest to those studying the quantization of classical geometric structures.

\subsection{Notation and terminology}
In this article we work with the data of an an exterior algebra $\Omega^\bullet_d$ over an associative unital $\bk$-algebra $A$, where $\bk$ is a unital commutative ring (cf.\ \cite[Definition~6.1, p.~29]{FMW}).
We will be making use of some key notions which were developed in previous work \cite{FMW,Symbol}.
In particular the notion of nonholonomic, semiholonomic, sesquiholonomic, and holonomic jet functors, as well as their associated natural transformations, including jet projections and prolongations, can be found in \cite{FMW}.
For the notions of elemental and primitive jet functors see \cite{Symbol}.
The notion of linear differential operator of order at most $n$ and their jet lifts can be found in \cite{FMW}, and their elemental and primitive counterpart, as well as the notion of symbols and restriction symbols for all types of differential operators can be found in \cite{Symbol}.

\subsection*{Acknowledgments}
The authors thank Shahn Majid for useful discussions on split jet bundles, and Jan Slov\'{a}k for useful discussions on natural operators in differential geometry.
K.~J.~F.~ was partially supported by the DFG priority program {\it Geometry at Infinity} SPP 2026: ME 4899/1-2.
M.~M.~\ was supported by the GA\v{C}R/NCN grant \emph{Quantum Geometric Representation Theory and Non-commutative Fibrations} 24-11728K, HORIZON-MSCA-2022-SE-01-01 CaLIGOLA, and \emph{SCREAM: Symmetry, Curvature Reduction, and EquivAlence Methods} funded by the Norwegian Financial Mechanism 2014-2021 (project registration number 2019/34/H/ST1/00636).
H.~W.~\ was partially supported by the UiT Aurora project MASCOT.
This article/publication is based upon work from COST Action CA21109, supported by COST (European Cooperation in Science and Technology).

\section{Natural differential operators}\label{s:Natural_differential_operators}
The topic of natural operators is a central one in differential geometry.
Classically, these are natural maps between sections of natural bundles over manifolds.
In particular, they comprise the diffeomorphism invariant operations that one has available on a smooth manifold.
Giving an account of the theory of natural operators is one of the aims of the classical book \cite{NaturalOperations}.
In particular, there it is shown that local natural operators are locally differential operators, hence, natural differential operators.
The goal of this section is to generalize this categorical picture to our noncommutative setting.
We will define natural differential operators, where differential operator is meant in the sense of \cite[\jetssdifferentialoperators]{FMW}.
\begin{defi}\label{def:natural_DO}
Let $F,G\colon \AMod\to \AMod$, and let $U\colon \AMod\to \Mod$ be the canonical forgetful functor.
Let $\Delta\colon U F\to U G$ be a natural transformation of functors $\AMod\to \Mod$.
We say that $\Delta$ is a \emph{natural linear (holonomic) differential operator} of order at most $n$ in $\AMod$ if there exists a natural transformation $\widetilde{\Delta}\colon J^n_d F\to G$ of functors $\AMod\to \AMod$, called \emph{natural lift} of $\Delta$ to $J^n_d F$, such that the following diagram of natural transformations of functors $\AMod\to \Mod$ commutes.
\begin{equation}
\begin{tikzcd}
U J^n F\ar[dr,dashed,"U(\widetilde{\Delta})"]\\
U F\ar[u,"j^n_{d,F}"]\ar[r,"\Delta"]&U G
\end{tikzcd}
\end{equation}
When $n$ is minimal, we say that $\Delta$ is a \emph{natural (holonomic) differential operator of order $n$}.
If $\Delta$ is a differential operator of order at most $n$ for some $n$, then we say that it is a \emph{differential operator of finite order}.

An analogous definition can be given for natural nonholonomic, semiholonomic, sesquiholonomic, elemental, primitive differential operators of order at most $n$, by choosing the respective jet functor and prolongation in place of the holonomic one.
\end{defi}
Notice that in Definition \ref{def:natural_DO}, the natural transformation $j^n_{d,F\cdot}$ is denoted as a natural transformation from $U$ to $UJ^n_d$, which is what is done implicitly by seeing it as a natural transformation of $\bk$-linear maps.
\begin{eg}\label{eg:prolongation_natural_DO}
The holonomic (resp.\ nonholonomic, semiholonomic, sesquiholonomic, elemental, primitive) $n$-jet prolongation $j^n_d\colon \id_{\AMod}\to J^n_d$ is a natural holonomic (resp.\ nonholonomic, semiholonomic, sesquiholonomic, elemental, primitive) differential operator of order at most $n$.
The lift is the identity natural transformation at $J^n_d$, that is the natural transformation $\id_{J^n_d}\colon J^n_d\to J^n_d$ of functors $\AMod\to\AMod$ with the identity on $J^n_d E$ in each component $E$.
\end{eg}
We can characterize natural differential operators as follows
\begin{prop}
\label{prop:characterisation_natural_DOs}
Every component of a natural (holonomic, nonholonomic, semiholonomic, sesquiholonomic, elemental, primitive) differential operator of order at most $n$ is a differential operator of order at most $n$.

Consider now the functors $F,G\colon \AMod\to \AMod$, and let $U\colon \AMod\to \Mod$ be the canonical forgetful functor.
Let $\Delta\colon UF\to UG$ be a natural transformation between the functors $UF,UG\colon\AMod\to \Mod$.
\begin{enumerate}
\item\label{prop:characterisation_natural_DOs:1} If for all $E$ in $\AMod$, $\Delta_E$ is an elemental (resp.\ primitive) differential operator of order at most $n$, then $\Delta$ is a natural elemental (resp.\ primitive) differential operator of order at most $n$.
\item\label{prop:characterisation_natural_DOs:2} If $J^n_dF=\EJ^n_dF$ (where $\EJ^n_d$ denotes the elemental jets functor, cf.\ \cite{Symbol}), and for all $E$ in $\AMod$, $\Delta_E$ is a holonomic differential operator of order at most $n$, then $\Delta$ is a natural holonomic differential operator of order at most $n$.
\end{enumerate}
\end{prop}
\begin{proof}
The first statement follows from the definition of differential operator by taking $\widetilde{\Delta}_E$ as lift of $\Delta_E$.

The remaining points follow from the fact that elemental and primitive jets, and in the conditions of \eqref{prop:characterisation_natural_DOs:2}, also holonomic jets, are $A$-linear combinations of prolongations of elements in $E$.
When this happens, we have that the collection $\widetilde{\Delta}$ having as component in $E$ the (unique) lift $\widetilde{\Delta}_E$ of $\Delta_E$, is a natural transformation.
In fact, given a map $\varphi\colon E_1\to E_2$, we have to prove that the following square commutes
\begin{equation}
\begin{tikzcd}
J^n_d F E_1\ar[d,"J^n_d F(\varphi)"']\ar[r,"\widetilde{\Delta}_{E_1}"]& G E_1\ar[d,"G(\varphi)"]\\
J^n_d F E_2\ar[r,"\widetilde{\Delta}_{E_2}"]& G E_2\
\end{tikzcd}
\end{equation}
Given the stated hypotheses, it is sufficient to verify the equality on elements of the form $\sum_i a_i j^n_{d,FE_1}(e_i)$ for $a_i\in A$ and $e_i\in FE_1$.
By $A$-linearity of $G(\varphi)$, $J^n_d F(\varphi)$, $\widetilde{\Delta}_{E_1}$, and $\widetilde{\Delta}_{E_2}$, and by naturality of $j^n_d$ and $\Delta$, we have the following.
\begin{equation}
\begin{split}
G(\varphi)\circ \widetilde{\Delta}_{E_1} \left( \sum_i a_i j^n_{d,FE_1}(e_i) \right)
&=\sum_i a_i G(\varphi)\circ \widetilde{\Delta}_{E_1} j^n_{d,FE_1}(e_i)\\
&=\sum_i a_i G(\varphi)\circ \Delta_{E_1} (e_i)\\
&=\sum_i a_i \Delta_{E_2}\circ J^n_d F(\varphi) (e_i)\\
&=\sum_i a_i \widetilde{\Delta}_{E_2}\circ j^n_{d, J^n_d FE_2}\circ J^n_d F(\varphi) (e_i)\\
&=\sum_i a_i \widetilde{\Delta}_{E_2}\circ J^n_d F(\varphi) \circ j^n_{d, FE_1} (e_i)\\
&=\widetilde{\Delta}_{E_2}\circ J^n_d F(\varphi) \left( \sum_i a_i j^n_{d, FE_1} (e_i)\right)
\end{split}
\end{equation}
This proves the desired commutativity, and the naturality of $\widetilde{\Delta}$.
\end{proof}
An immediate consequence is the following result.
\begin{cor}
\label{cor:natDO_order_0,1}\
\begin{enumerate}
\item\label{cor:natDO_order_0,1:0} Natural linear (holonomic, nonholonomic, semiholonomic, sesquiholonomic, elemental, primitive) differential operators $F\to G$ of order $0$ are precisely $A$-linear natural transformations, i.e.\ natural transformations $F\to G$ of functors $\AMod\to \AMod$.
\item\label{cor:natDO_order_0,1:1} Natural linear (holonomic, nonholonomic, semiholonomic, sesquiholonomic, elemental, primitive) differential operators $F\to G$ of order at most $1$ are precisely natural transformations where each component is a differential operator of order at most $1$.
\end{enumerate}
\end{cor}
\begin{proof}
By definition, at degrees $0$ and $1$, all jet functors coincide, so we can apply Proposition \ref{prop:characterisation_natural_DOs}.\eqref{prop:characterisation_natural_DOs:2}, obtaining the desired statements.
\end{proof}
\begin{eg}
The natural transformations $\iota^n_d$, $\pi^{n,m}_d$, $l^{n,m}_d$, are all differential operators of order $0$.
\end{eg}
\begin{eg}
The natural transformation $\DH_d\colon J^1_d\to \Omega^1_d \ltimes \Omega^2_d$, cf.\ \cite[\jetseqDHnaturalDO]{FMW}, is a natural differential operator of order at most $1$ with lift $\widetilde{\DH}_d$.
\end{eg}
We will now prove some properties of natural differential operators that are analogous to those of differential operators, cf.\ \cite[\jetssdifferentialoperators]{FMW}.
\begin{prop}
\label{prop:properties_natDOs}\
\begin{enumerate}
\item\label{prop:degree_natDOs_increasable} Let $m\le n$, then a natural differential operator of order at most $m$ is also a natural differential operator of order at most $n$.
\item\label{prop:composition_natDOs_is_natDO} Consider the functors $F,G,H\colon \AMod\to\AMod$ and let $\Delta_1\colon F \to G$ and $\Delta_2\colon G \to H$ be natural differential operators of order at most $n$ and $m$, respectively.
	Then the composition $\Delta_2 \circ \Delta_1\colon F \to H$ is a differential operator of order at most $n+m$.
\item Natural differential operators of finite order form a category $\NatDiff_d$, where the objects are functors $\AMod\to\AMod$ and the arrows between two such functors $F,G$ are given by natural differential operators of finite order between them.
The set of morphisms is denoted by $\NatDiff_d(F,G)\colonequals\bigcup_n \NatDiff^n_d(F,G)$, where $\NatDiff^n_d(F,G)$ is the set of natural differential operators $F\to G$ of order at most $n$.
This is a subcategory of the category of functors $\AMod\to\Mod$ and natural transformations between them.
\item The category $\NatDiff_d$ is enriched over filtered $\bk$-vector spaces with filtration given by the grade.
\item $\NatDiff^0_d(F,G)$ is the space of natural transformations $F\to G$ of functors $\AMod\to\AMod$.
\end{enumerate}
\end{prop}
\begin{proof}\
\begin{enumerate}
\item It can be proven \textit{mutatis mutandis} as \cite[\jetspropoperatorsalsohigherorder]{FMW}, since $\pi^{m,n}_d$ is an $A$-linear natural transformation.
\item Proven \textit{mutatis mutandis} as in \cite[\jetspropdifferentialoperatorcomposition]{FMW}, since $l^{m,n}_d$ is an $A$-linear natural transformation.
\item It follows from \eqref{prop:composition_natDOs_is_natDO}, since the composition is inherited from the composition of natural transformations and closed in this subcategory.
\item The enrichment over $\Mod$ is inherited from that of the category of natural transformations of functors $\AMod\to\Mod$, and the $\Hom$-spaces of $\NatDiff_d$ are subspaces of it.
The filtration is given by the subspaces $\NatDiff^n_d(F,G)$ of $\NatDiff_d(F,G)$ and the composition preserves the grading by \eqref{prop:composition_natDOs_is_natDO}.
\item It follows from Corollary \ref{cor:natDO_order_0,1}.\eqref{cor:natDO_order_0,1:0}.\qedhere
\end{enumerate}
\end{proof}
For many statements that hold naturally in the setting of differential operators, in the sense of \cite[\jetssdifferentialoperators]{FMW}, one can prove, \textit{mutatis mutandis}, analogues in the setting of natural differential operators.
Furthermore, under the appropriate conditions, one can do the same for all other types of jet functors appearing in \cite{FMW} and \cite{Symbol}, obtaining their corresponding natural differential operators and analogous results.
\subsection{Symbols of natural differential operators}
\label{ss:Symbols of natural differential operators}
The notion of symbol of a natural differential operator will be the symbol in each component, i.e.\ 
\begin{defi}
Let $\Delta\colon F\to G$ be a natural differential operator of order at most $n$, then its \emph{$n$-symbol} $\symb^n_d(\Delta)$ is a collection of symbols indexed by objects $E$ in $\AMod$, where we define the $E$ component as $\symb^n_{d,E}(\Delta)\colonequals \symb^n_d (\Delta_E)$.

If $\im(\iota^n_{d,F})\subseteq \EJ^n_dF$, we can define the \emph{$n$-restriction symbol} of $\Delta$, with natural lift $\widetilde{\Delta}\colon J^n_d F\to G$, as the natural transformation $\widetilde{\Delta}\circ \iota^n_{d,F}\colon S^n_d F\to G$, cf.\ \cite[\symbolsdefrestrictionsymbol]{Symbol}.
\end{defi}
The natural restriction symbol is well-defined, because given another natural lift $\widetilde{\Delta}'\colon J^n_d F\to G$ of $\Delta$, the difference $\widetilde{\Delta}-\widetilde{\Delta}'$ is a natural lift of the zero map, and thus we obtain $(\widetilde{\Delta}-\widetilde{\Delta}')\circ \iota^n_{d,F}=0$, cf.\ \cite[\symbolspropcharacterizationsymmgeneratedbyprol]{Symbol}.

In general, the vanishing of the $n$-symbol of a natural differential operator of order at most $n$ is not enough to prove that it is a natural differential operator of order $n-1$, but we have the following.
\begin{prop}
\label{prop:symb_and_reduction_n->n-1_NDOs}
Let $\Delta\colon F\to G$ be a natural differential operator of order at most $n$.
If $\Delta$ is a differential operator of order at most $n-1$, then $\symb^n_d(\Delta)=0$.
The opposite implication holds if $\pi^{n,n-1}_{d,F}$ is a natural epimorphism and $J^n_d F=\EJ^n_d F$.
Under this condition, a natural differential operator of order at most $n$ such that all of its components are differential operators of order $n-1$, is a natural differential operator of order at most $n-1$.
\end{prop}
\begin{proof}
The first implication follows from Proposition \ref{prop:characterisation_natural_DOs}.

For the opposite implication, let $\Delta$ be a natural differential operator of order at most $n$ such that $\symb^n_d(\Delta)=0$, then by definition, every component of $\Delta$ is a differential operator of order at most $n-1$.
Let $\widetilde{\Delta}^{n-1}_{d,E}\colon J^{n-1}_d FE\to GE$ be a natural lift of $\Delta_E$ for all $E$ in $\AMod$.
We want to prove that this lift is natural.
For all $\phi\colon E_1\to E_2$, consider the following diagram
\begin{equation}
\label{diag:symb_and_reduction_n->n-1_NDOs}
\begin{tikzcd}[column sep=40pt, row sep=30pt]
J^n_d FE_1\ar[d,"J^n_d F(\phi)"']\ar[rr,bend left=20pt,"\widetilde{\Delta}_{E_1}"]\ar[r,two heads,"\pi^{n,n-1}_{d,FE_1}"']&J^{n-1}_d F E_1\ar[d,"J^{n-1}_d F(\phi)"]\ar[r,"\widetilde{\Delta}^{n-1}_{E_1}"']&GE_1\ar[d,"G(\phi)"]\\
J^n_d FE_2\ar[rr,bend right=20pt,"\widetilde{\Delta}_{E_2}"']\ar[r,two heads,"\pi^{n,n-1}_{d,FE_2}"]&J^{n-1}_d FE_2\ar[r,"\widetilde{\Delta}^{n-1}_{E_2}"]&GE_2
\end{tikzcd}
\end{equation}
The top triangle of \eqref{diag:symb_and_reduction_n->n-1_NDOs} commutes because both the curved morphism and the composition lift the same differential operator $\Delta_{E_1}$ and, since $J^n_dE_1=\EJ^n_dE_1$, the lift is unique, cf.\ \cite[\symbolspropelementaljetproperties]{Symbol}.
The same holds for the bottom curved triangle.
The left square of \eqref{diag:symb_and_reduction_n->n-1_NDOs} commutes by naturality of $\pi^{n,n-1}_{d,F}$ with respect to $F(\phi)$.
The external square commutes because $\widetilde{\Delta}$ is the natural lift of $\Delta$.
It follows that $G(\phi)\circ \widetilde{\Delta}^{n-1}_{E_1}\circ \pi^{n,n-1}_{d,FE_1}=\widetilde{\Delta}^{n-1}_{E_2}\circ J^{n-1}_d F(\phi)\circ \pi^{n,n-1}_{d,FE_1}$, and since $\pi^{n,n-1}_{d,FE_1}$ is an epi, we can cancel it from this equality.
This yield the commutativity of the right square of \eqref{diag:symb_and_reduction_n->n-1_NDOs}, and hence the naturality of $\widetilde{\Delta}^{n-1}$ defined to be $\widetilde{\Delta}^{n-1}_E$ at each component $E$ in $\AMod$.
This makes $\widetilde{\Delta}^{n-1}$ a natural lift of $\Delta$.
\end{proof}

Under sufficient regularity, the restriction symbol also allows us to determine whether a natural differential operator of order at most $n$ is a natural differential operator of order at most $n-1$.
\begin{prop}
Let $\im(\iota^n_{d,F})\subseteq \EJ^n_d F$ and let the $n$-jet sequence be right exact.
If $\Delta\colon F\to G$ is a natural differential operator of order at most $n$ such that for any (and hence all) lift $\widetilde{\Delta}$ we have $\widetilde{\Delta}\circ \iota^n_{d,F}=0$, then $\Delta$ is a natural differential operator of order at most $n-1$.
\end{prop}
\begin{proof}
This result can be proven as a consequence of Proposition \ref{prop:symb_and_reduction_n->n-1_NDOs} and \cite[\symbolspropsymbolsrepresentationwelldefined]{Symbol}.
For a more direct proof, the natural lift of $\Delta$ is given by the cokernel universal property of $\iota^n_{d,F}$.
\end{proof}
\section{Spencer operators and the Spencer complex}
\label{s:Spencer_operators_and_the_Spencer_complex}
In this section we will introduce noncommutative generalizations of one of the most central examples of natural differential operators from differential geometry.
These are the Spencer operators, which operate on bundles of jet-valued differential forms, and generally have the effect of lowering the jet order while increasing the form degree.
Of course, this generalization will form examples of natural differential operators in our sense, and we will show that our generalization extends to capture most of their classical properties, such as their relationship to the Spencer $\delta$-operators (generalized to our setting in \cite[\jetsssSpencer]{FMW}), the Spencer complex and bicomplex.
Further, we use the Spencer operators to characterize holonomic jets amongst semiholonomic jets in our setting, c.f.\ Lemma \ref{lemma:sesquiholonomic_Spencer} and Remark \ref{rmk:holonomic_alternate_def_spencer}.
\subsection{Spencer operators on holonomic jets}
\label{ss:Spencer operators on holonomic jets}
In this section we generalize the Spencer complex to the setting of noncommutative geometry in the spirit of \cite{Spencer} and \cite{goldschmidt1967existence}.
\begin{defi}\label{def:Spenceroperator}
	We define the \emph{(holonomic) Spencer operators} as the following natural transformations for $n\ge 1$ and $m\ge 0$ with the following component at each $E$ in $\AMod$:
\begin{align}
\label{eq:def_Spencer_operator}
\Spenc_{d,E}^{n,m} \colon \Omega^m_d J^n_d E \longrightarrow \Omega^{m+1}_dJ^{n-1}_d E,
&\hfill&
\omega \otimes_A \sum_j [y_j \otimes z_j] \otimes_A \xi_j \longmapsto \sum_j d(\omega y_j) z_j \otimes_A \xi_j,
\end{align}
for all $\omega\in\Omega^m_d$ and $\sum_j [y_j\otimes z_j]\otimes_A \xi_j\in J^n_d E\subseteq J^1_d J^{n-1}_d E$.
\end{defi}
\begin{rmk}
When $E$ is in $\AModB$ for a $\bk$-algebra $B$, then we can extend this result to bimodules, obtaining that $\Spenc^{n,m}_{d,E}$ is a morphism in $\ModB$.
We can thus see $\Spenc^{n,m}_d$ as a natural transformation of functors $\AModB\to\ModB$.
\end{rmk}
\begin{prop}
\label{prop:Spencercorrespondence}
	The Spencer operator coincides with the operator given by the formula
	\begin{align}\label{eq:classicSpencer}
	\Spenc^{n,m}_{d,E} \colon \Omega^m_d J^n_dE \longrightarrow \Omega^{m+1}_d J^{n-1}_dE,
	&\hfill&
	\omega \otimes_A \xi\longmapsto
	d\omega \otimes_A \pi^{n,n-1}_{d,E} (\xi) + (-1)^{\deg(\omega)} \omega \wedge \Spenc^{n,0}_{d,E}(\xi).
	\end{align}
\end{prop}
\begin{proof}
	We apply the graded Leibniz rule for $d$ to \eqref{eq:def_Spencer_operator}.
	For $\xi=\sum_j [y_j,z_j]\otimes_A \xi_j$, this yields
	\begin{equation}
		\Spenc^{n.m}_d(\omega \otimes_A \xi)
		=\sum_j d(\omega y_j) z_j \otimes_A \xi_j
		=\sum_j \left(d\omega \otimes y_j z_j \xi_j + (-1)^{\deg \omega} \omega \wedge (dy_j)z_j \otimes_A \xi_j\right),
	\end{equation}
	where we have $\pi^{n,n-1}_{d,E}([y_j \otimes z_j] \otimes_A \xi_j) = y_jz_j \xi_j$ and $\Spenc_{d, E}^{n,0}(\xi_j) = (dy_j)z_j \otimes_A \xi_j$.
\end{proof}
\begin{rmk}
	In the classical setting $A = \smooth{M}$ with the de Rham exterior derivative $d=d_{dR}$, equation \eqref{eq:classicSpencer} coincides with the classical formula for the Spencer operator, cf.\ \cite[Proposition~1.3.1, p.~187]{Spencer}.
\end{rmk}

\begin{prop}
\label{prop:spencopsymbol}
	The Spencer operator $\Spenc^{n,m}_d$ is a natural linear differential operator of order at most $1$ for $n\ge 1$ and $m\ge 0$.
	Moreover, it has (restriction) symbol $\wedge^{1,m} \otimes_A \pi^{n,n-1}_d$.
\end{prop}
\begin{proof}
For $n\ge 1$ and $m\ge 0$, we have that $\Spenc^{n,m}_d$ is a natural differential operator of order at most $1$ since we can construct a jet lift of each component $E$, extending by $A$-linearity, the following map:
	\begin{align}\label{eq:Spencerlift}
		\widetilde{\Spenc}^{n,m}_{d,E}\colon J^1_d \Omega^m_d J^n_d E\longrightarrow \Omega^{m+1}_d J^{n-1}_d E,
		&\hfill&
		[a \otimes b] \otimes_A \omega \otimes_A \sum_j [y_j \otimes z_j] \otimes_A \xi_j \longmapsto \sum_j a d(b \omega y_j) z_j \otimes_A \xi_j,
	\end{align}
	where $[a\otimes b]\in J^1_d A$, $\omega\in\Omega^m_d$, and $\sum_j [y_j\otimes z_j]\otimes_A \xi_j\in J^n_d E$.
	This map is well-defined by the Leibniz rule, and precomposing it with $j^1_{d,\Omega^m_d J^n_d E}$ yields $\Spenc^{n,m}_{d,E}$.
	Furthermore, we show that these are the components of a natural transformation since, given a morphism $\varphi\colon E\to E'$ in $\AMod$, we have the following
	\begin{equation}
	\begin{split}
	&\Omega^{m+1}_d J^{n-1}_d(\varphi)\circ\widetilde{\Spenc}^{n,m}_{d,E}\left([a \otimes b] \otimes_A \omega \otimes_A \sum_j [y_j \otimes z_j] \otimes_A \xi_j\right)\\
	&\qquad =\Omega^{m+1}_d J^{n-1}_d(\varphi)\left(\sum_j a d(b \omega y_j) z_j \otimes_A \xi_j\right)\\
	&\qquad =\sum_j a d(b \omega y_j) z_j \otimes_A J^{n-1}_d(\varphi)(\xi_j)\\
	&\qquad =\widetilde{\Spenc}^{n,m}_{d,E'}\left([a \otimes b] \otimes_A \omega \otimes_A \sum_j [y_j \otimes z_j] \otimes_A J^{n-1}_d(\varphi)(\xi_j)\right)\\
	&\qquad =\widetilde{\Spenc}^{n,m}_{d,E'}\left([a \otimes b] \otimes_A \omega \otimes_A J^n_d(\varphi)\left(\sum_j [y_j \otimes z_j] \otimes_A \xi_j\right)\right)\\
	&\qquad =\widetilde{\Spenc}^{n,m}_{d,E'}\circ J^1_d \Omega^m_d J^n_d(\varphi)\left([a \otimes b] \otimes_A \omega \otimes_A \sum_j [y_j \otimes z_j] \otimes_A \xi_j\right).
	\end{split}
	\end{equation}
	Since $\Spenc^{n,m}_{d,E}$ is a natural linear differential operator of order at most $1$, its symbol can be identified with the restriction of its jet lift along $\iota^1_{d,E}$, cf.\ \cite[\symbolsremidentifyrestrictionsymbol]{Symbol}.
	This amounts to evaluating $\widetilde{\Spenc}^{n,m}_{d,E}$, as in \eqref{eq:Spencerlift}, on elements of the form $a db\otimes_A \omega \otimes_A \sum_j [y_j \otimes z_j] \otimes_A \xi_j\in \Omega^1_d \Omega^m_d J^n_d E$.
	Applying $\widetilde{\Spenc}^{n,m}_{d,E}$ to the element $\iota^1_{d,\Omega^m_d J^n_d E}(a db\otimes_A \omega \otimes_A \sum_j [y_j \otimes z_j] \otimes_A \xi_j) = [a \otimes b - ab \otimes 1]\otimes_A \omega \otimes_A \sum_j [y_j \otimes z_j] \otimes_A \xi_j$, yields, via Leibniz, the desired equation.
\end{proof}
\begin{rmk}
\label{rmk:Spencer_m=0}
	For $m=0$, the Spencer operator $\Spenc^{n,m}_d$ at the component $E$ can be written as
	\begin{equation}\label{eq:explicit_Sh0}
		\Spenc_{d,E}^{n,0}
		=-\rho_{d,J^{n-1}_d E}\circ l^{1,n-1}_{d,E}\colon J^n_d E \longrightarrow \Omega^1_d J^{n-1}_dE,
	\end{equation}
	where $\rho_d\colon J^1_d \to \Omega^1_d$, cf.\ \cite[\jetseqrhon]{FMW}, is the natural transformation with component $\rho_{d,E}=\rho_{d,A}\otimes_A \id_E$ at $E$, such that $\rho_{d,A}\colon J^1_d A \to \Omega^1_d$ is the map
	\begin{equation}
		[a \otimes b] \longmapsto -(da)b
	\end{equation}
	from \cite[\jetsdefrho]{FMW}.
	Therefore, Proposition \ref{prop:Spencercorrespondence} gives us an an alternative explicit description of a generic Spencer operator as follows
	\begin{equation}\label{eq:classicSpencerexplicit}
		\Spenc^{n,m}_{d,E}(\omega\otimes_A \xi)
		=d\omega \otimes_A \pi^{n,n-1}_{d,E} (\xi) - (-1)^{\deg(\omega)} \omega \wedge \rho_{d,J^{n-1}_d E}\circ l^{1,n-1}_{d,E}(\xi)
	\end{equation}
	In particular, for $\Spenc^{n,0}_d$, equation \eqref{eq:classicSpencer} can be regarded as a (graded) Leibniz rule
	\begin{equation}
		\Spenc^{n,0}_{d,E}(f\xi)
		=df \otimes_A \pi^{n,n-1}_{d,E} (\xi) + (-1)^{\deg(\omega)} f \Spenc_{d,E}^{n,0}(\xi),
	\end{equation}
	in accordance with \cite[\jetsrmkrhodo]{FMW}.
	The natural jet lift of $\Spenc^{n,0}_{d,E}$ is
	\begin{equation}
	\label{eq:lift_Spencer_n,0}
		\widetilde{\Spenc}^{n,0}_d
		=\widetilde{\DH}^I_{J^{n-1}_d} \circ J^1_d(l^{1,n-1}_d )\colon J^1_d J^n_d\longrightarrow\Omega^1_d J^{n-1}_d,
	\end{equation}
	Further, the symbol of $\Spenc^{n,0}_{d,E}$ from Proposition \ref{prop:spencopsymbol} reduces to
	\begin{equation}
		\Omega^1_d(\pi^{n,n-1}_{d,E})
		=\id_{\Omega^1_d}\otimes_A \pi^{n,n-1}_{d,E}.
	\end{equation}
\end{rmk}
\begin{rmk}
The map $\mu\colon \Omega^2_d J^1_d\to \Omega^3_d$ appearing in \cite[\jetslemmaextwedge]{FMW} coincides with $-\Spenc^{1,2}_d$, cf.\ \eqref{eq:classicSpencerexplicit}.
As such, it is a natural differential operator of order at most $1$.
\end{rmk}

Before continuing, we give the following result, expressing Spencer operators on a jet functor in terms of Spencer operators of lower index.
\begin{lemma}\label{lemma:SpencerDO_wrt_low_indices}
The Spencer differential operator $\Spenc^{1,m}_d$ is a natural epimorphism and for all $n\ge 1$ and $m\ge 0$ we have $\Spenc^{n,m}_d=\Spenc^{1,m}_{d,J^{n-1}_d}\circ \Omega^m_d(l^{1,n-1}_d)$, i.e.\ the following diagram commutes
\begin{equation}\label{diag:SpencerDO_wrt_low_indices}
\begin{tikzcd}[column sep=50pt,row sep=30pt]
\Omega^m_d J^n_d\ar[d,"\Omega^m_d(l^{1,n-1}_d)"']\ar[dr,"\Spenc^{n,m}_d"]\\
\Omega^m_d J^1_d J^{n-1}_d\ar[r,two heads,"\Spenc^{1,m}_{d,J^{n-1}_d}"']&\Omega^{m+1}_d J^{n-1}_d
\end{tikzcd}
\end{equation}
\end{lemma}
\begin{proof}
We show that $\Spenc^{1,m}_{d,A}\colon \Omega^m_dJ^1_d A\to \Omega^{m+1}_d$ is an epimorphism via the surjectivity condition, which ensures that $\Omega^{m+1}_d$ is generated by elements of the form $da_m\wedge\dots\wedge (da_1)a_0$ for $a_0,\dots a_m\in A$.
We will generate such an element as follows
\begin{equation}
\Spenc^{1,m}_{d,A}\left((-1)^m da_m\wedge \cdot \wedge da_1\otimes_A [1,a_0]\right)
=da_m\wedge\dots\wedge (da_1)a_0.
\end{equation}
The commutativity of \eqref{diag:SpencerDO_wrt_low_indices} follows from the definition of a Spencer operator, as $\Spenc^{1,m}_{d,E}(\omega\otimes_A [y,z])=d(\omega y)z$.
\end{proof}
Since $J^1_d$ is the holonomic, elemental, primitive, nonholonomic, semiholonomic, sesquiholonomic jet functor of order $1$, Lemma \ref{lemma:SpencerDO_wrt_low_indices} provides us with a way of generalizing the Spencer operators to the case of all other notions of jets, as long as that notion of jet admits a natural transformation $l^{1,n-1}_d\colon J^n_d\to J^1_d J^{n-1}_d$ compatible with the jet projections.
\subsection{Spencer operators on other jet functors}
\label{ss:Spencer operators on other jet functors}
We would further like to define Spencer operators for more general notions of jet functors than just the holonomic jets, which we treated in §\ref{ss:Spencer operators on holonomic jets}.
The case of the sesquiholonomic jet functor is particularly relevant to our subsequent developments, cf.\ Lemma \ref{lemma:sesquiholonomic_Spencer}.
In order to obtain all relevant cases with a minimum of redundancy, we gather in the following list the necessary data to generalize the results of \S\ref{ss:Spencer operators on holonomic jets}, as well as several results from the theory of jets.
\begin{enumerate}[label=(\textbf{J\arabic*})]
	\item\label{desid:Spencerable:1} A family of functors $\itJ^n_d \colon \AMod\to \AMod$ for $n\ge 0$, such that $\itJ^0_d \simeq \id_{\AMod}$ and $\itJ^1_d \simeq J^1_d$;
	\item\label{desid:Spencerable:2} A family of $A$-linear natural transformations $\itl^{1,n-1}_d\colon \itJ^n_d\to \itJ^1_d\itJ^{n-1}_d$, for $n\ge 1$, where $\itl^{1,0}_d\colonequals \id_{J^1_d}$;
	\item\label{desid:Spencerable:3} A family of $A$-linear natural transformations $\itpi^{n,n-1}_d \colonequals \pi^{1,0}_{d,\itJ^n_d}\circ \itl^{1,n-1}_d$, for $n\ge 1$.
\end{enumerate}
Under suitable conditions, the jet constructions presented in \cite{FMW} and \cite{Symbol}, are of this form.
\begin{prop}
\label{prop:Spencerable_jets}
The following collections of data satisfy \ref{desid:Spencerable:1}, \ref{desid:Spencerable:2}, \ref{desid:Spencerable:3}.
\begin{enumerate}
\item Holonomic jet functors: $J^n_d$, $l^{1,n-1}_d=l^n_d$, $\pi^{n,n-1}_d$;
\item Holonomic jet tensor functors: $J^n_d A\otimes -$, $l^{1,n-1}_{d,A}\otimes_A -=l^n_{d,A}\otimes_A -$, $\pi^{n,n-1}_{d,A}\otimes_A -$;
\item Nonholonomic jet functors: $J^{(n)}_d$, $l^{(n)}_d=\id_{J^{(n)}_d}$, $\pi^{(n,n-1)}_d$;
\item\label{prop:Spencerable_jets:4} Sesquiholonomic jet functors: $J^{\{n\}}_d$, $l^{\{1,n-1\}}_d\colonequals J^1_d(\sqhiota{n-1})\circ l^{\{n\}}_d$, $\sqhiota{n-1}\circ \pi^{(n,n-1)}_d$, where $\sqhiota{n}\colon J^n_d \hookrightarrow J^{\{n\}}_d$ is the natural inclusion.
\end{enumerate}
Further, the following also satisfy \ref{desid:Spencerable:1}, \ref{desid:Spencerable:2}, \ref{desid:Spencerable:3}, provided that $\Omega^1_d$ is flat in $\ModA$.
\begin{enumerate}[resume]
\item Semiholonomic jet functors: $J^{[n]}_d$, $l^{[n]}_d$, $\pi^{[n,n-1]}_d$;
\item Elemental jet functors: $\EJ^n_d$, $\El^{1,n-1}_d$, $\Epi^{n,n-1}_d$;
\item Primitive jet functors: $\PJ^n_d$, $\Pl^{1,n-1}_d$, $\Ppi^{n,n-1}_d$.
\end{enumerate}
\end{prop}
\begin{proof}\
\begin{enumerate}
\item All the conditions are true by definition.
In particular, for \ref{desid:Spencerable:1} and \ref{desid:Spencerable:2}, cf.\ \cite[\jetsdefnjet]{FMW}, and \cite[\jetsdefpiiotan]{FMW}.
\item In this case, the conditions follow from the previous point restricted in the component $A$, and via the bifunctoriality of $-\otimes_A -$.
\item All three desiderata follow by definition, cf.\ \cite[\jetsdefnonholjetfunctor]{FMW} and \cite[\jetseqnhpink]{FMW}.
\item We obtain \ref{desid:Spencerable:1} by definition, cf.\ \cite[\jetsdefsesqui]{FMW}.
 The following diagram yields the essentially unique natural monomorphism $\sqhiota{n}\colon J^n_d\hookrightarrow J^{\{n\}}_d$ via the kernel universal property
\begin{equation}\label{diag:inclusion_hol_sesqui}
\begin{tikzcd}
0\ar[r]&J^n_d\ar[d,hook,dashed,"\sqhiota{n}"']\ar[r,hook,"l^{n}_d"]&J^1_d J^{n-1}_d\ar[d,equals] \ar[r,"\widetilde{\DH}_{J^{n-2}_d}\circ J^1_d(l^{n-1}_d)"]&[60pt](\Omega^1_d\ltimes \Omega^2_d)J^{n-2}_d\ar[d,two heads]\\
0\ar[r]&J^{\{n\}}_d\ar[r,hook,"l^{\{n\}}_d"]&J^1_d J^{n-1}_d \ar[r,"\widetilde{\DH}^I_{J^{n-2}_d}\circ J^1_d(l^{n-1}_d)"]&\Omega^1_d J^{n-2}_d
\end{tikzcd}
\end{equation}
where the rightmost vertical map is the natural projection of $\Omega^1_d\ltimes \Omega^2_d$ onto its first component.
Notice that we can set $\sqhiota{0}=\id$ and $\sqhiota{1}=\id_{J^1_d}$.

We can now define the maps for \ref{desid:Spencerable:2} via $l^{\{n\}}_d$, cf.\ \cite[\jetsdefsesqui]{FMW}, as
\begin{equation}\label{eq:defi_sesqui_l}
l^{\{1,n-1\}}_d\colonequals J^1_d(\sqhiota{n-1})\circ l^{\{n\}}_d \colon J^{\{n\}}_d\longrightarrow J^1_d J^{\{n-1\}}_d.
\end{equation}
In low dimension we have $l^{\{1,0\}}_d=\id_{J^1_d}$ by definition.

Finally, we obtain the maps for \ref{desid:Spencerable:3} as
\begin{equation}
\sqhiota{n}\circ\pi^{\{n,n-1\}}_d\colon J^{\{n\}}_d \hookrightarrow J^{\{n-1\}}_d.
\end{equation}
The required property is given by the following commutative diagram
\begin{equation}
\begin{tikzcd}[column sep=60pt,row sep=20pt]
J^{\{n\}}_d\ar[rr,bend left=15pt,"\pi^{\{n,n-1\}}_d"]\ar[r,near end,hook,"l^{\{n\}}_d"']\ar[dr,"l^{\{1,n-1\}}_d"']&J^1_d J^{\{n-1\}}_d\ar[d,"J^1_d(\sqhiota{n-1})"]\ar[r,two heads,"\pi^{1,0}_{d,J^{n-1}_d}"']&J^{n-1}_d\ar[d,hook,"\sqhiota{n-1}"]\\
&J^1_d J^{n-1}_d\ar[r,two heads,"\pi^{1,0}_{d,J^{\{n-1\}}_d}"']&J^{\{n-1\}}_d
\end{tikzcd}
\end{equation}
The top triangle commutes by definition, cf.\ \cite[\jetsdefisesquisequence]{FMW}.
The left triangle commutes by definition, cf.\ \eqref{eq:defi_sesqui_l}, and finally the right square commutes by naturality of $\pi^{1,0}_d$ with respect to $\sqhiota{n-1}$.
\item We obtain \ref{desid:Spencerable:1} by definition, cf.\ \cite[\jetsdefshj]{FMW} and \cite[\jetsrmklowsemihol]{FMW}.
If $\Omega^1_d$ is flat in $\ModA$, then we obtain \ref{desid:Spencerable:2} by \cite[\jetstheosjchar]{FMW}, while for \ref{desid:Spencerable:3} it is a consequence of the commutativity of the following diagram because by definition $\pi^{[n,n-1]}_d$ is the unique restriction of each, and hence of all, the morphisms $\pi^{(n,n-1;m)}_d$ to $J^{[n]}_d$.
\begin{equation}\label{diag:semihol_projection_is_J3}
\begin{tikzcd}[column sep=40pt, row sep=20pt]
J^{[n]}_d\ar[r,hook,"l^{[n]}_d"]\ar[d,hook,"\iota_{J^{[n]}_d}"']&J^1_d J^{[n-1]}_d\ar[r,two heads,"\pi^{1,0}_{d,J^{[n-1]}_d}"]\ar[d,hook,"J^1_d(\iota_{J^{[n-1]}_d})"']&J^{[n-1]}_d\ar[d,hook,"\iota_{J^{[n-1]}_d}"]\\
J^{(n)}_d\ar[r,equals,"l^{(n)}_d"]\ar[rr,bend right=15pt,two heads,"\pi^{(n,n-1)}_d"']&J^1_d J^{(n-1)}_d\ar[r,two heads,"\pi^{1,0}_{d,J^{(n-1)}_d}"]&J^{(n-1)}_d
\end{tikzcd}
\end{equation}
The left square of \eqref{diag:semihol_projection_is_J3} commutes by \cite[\jetseqdecsjn]{FMW}, and the right square commutes by naturality of $\pi^{1,0}_d$ with respect to $\iota^{[n]}_d$.
By uniqueness, it follows that $\pi^{[n,n-1]}_d=\pi^{1,0}_{d,J^{[n-1]}_d}$.
\item We obtain \ref{desid:Spencerable:1} by \cite[\symbolspropelementaljetproperties]{Symbol}, and \ref{desid:Spencerable:2} by \cite[\symbolslemmaelementalsmmtwo]{Symbol}.
Since $\Epi^{n,n-1}_d$ is defined as the restriction of $\pi^{n,n-1}_d$ to $\EJ^n$, we obtain \ref{desid:Spencerable:3}, \textit{mutatis mutandis}, as in the proof for the semiholonomic analogue, where, for instance, $\EJ^n_d$ takes the r\^{o}le of $J^{[n]}_d$, and $J^{n}_d$ takes the r\^{o}le of $J^{(n)}_d$.
\item Under the given hypotheses, elemental and primitive jets coincide, cf.\ \cite[\symbolslemmaskelprimitivecompare]{Symbol}.\qedhere
\end{enumerate}
\end{proof}

Whenever we have \ref{desid:Spencerable:1}, \ref{desid:Spencerable:2}, and \ref{desid:Spencerable:3}, we can define Spencer operators as follows.
\begin{defi}\label{def:generalised_Spencer}
We define the \emph{Spencer differential operator} as the map
\begin{equation}
\itSpenc^{n,m}_d\colonequals \Spenc^{1,m}_{d,\itJ^{n-1}_d}\circ \Omega^m_d(\itl^{1,n-1}_d)\colon \Omega^m_d \itJ^n_d\longrightarrow\Omega^{m+1}_d \itJ^{n-1}_d.
\end{equation}
We will denote the Spencer operators for the cases of Proposition \ref{prop:Spencerable_jets} as the corresponding notation, namely $\Spenc^{n,m}_d$, $\Spenc^{n,m}_{d,A}\otimes_A -$, $\Spenc^{(n,m)}_d$, $\Spenc^{\{n,m\}}_d$, $\Spenc^{[n,m]}_d$, $\ESpenc^{n,m}_d$, $\PSpenc^{n,m}_d$.
When necessary, we will refer to them by appending the corresponding adjective in front of Spencer operator.
\end{defi}
We will now prove that whenever we have \ref{desid:Spencerable:1}, \ref{desid:Spencerable:2}, and \ref{desid:Spencerable:3}, we automatically obtain the results of §\ref{ss:Spencer operators on holonomic jets}.
\begin{prop}
	The following results hold whenever we have \ref{desid:Spencerable:1}, \ref{desid:Spencerable:2}, and \ref{desid:Spencerable:3}.
	\begin{enumerate}
		\item $\itSpenc_{d,E}^{n,m}(\omega \otimes_A \xi)= d\omega \otimes_A \itpi^{n,n-1}_{d,E} (\xi) + (-1)^{\deg(\omega)} \omega \wedge \itSpenc^{n,0}_{d,E}(\xi)$, for all $\omega\otimes_A \xi\in \Omega^m_d \itJ^n_dE$.
		\item $\itSpenc^{n,m}_d$ is a natural differential operator of order at most $1$ with restriction symbol $\wedge^{1,m}_d\otimes \itpi^{n,n-1}_d$.
	\end{enumerate}
\end{prop}
\begin{proof}\
	\begin{enumerate}
		\item We prove this equality by a direct computation and via Proposition \ref{prop:Spencercorrespondence} and \ref{desid:Spencerable:3}
		\begin{equation}
		\begin{split}
			\itSpenc^{n,m}_{d,E}(\omega \otimes_A \xi)
			&=\Spenc^{1,m}_{d,\itJ^{n-1}_d E}\circ \Omega^m_d(\itl^{1,n-1}_{d,E})(\omega\otimes_A \xi)\\
			&=\Spenc^{1,m}_{d,\itJ^{n-1}_d E}(\omega\otimes_A \itl^{1,n-1}_{d,E}(\xi))\\
			&= d\omega \otimes_A \pi^{1,0}_{d,\itJ^{n-1}_d E}\circ \itl^{1,n-1}_{d,E} (\xi) + (-1)^{\deg(\omega)} \omega \wedge \Spenc^{1,0}_{d,\itJ^{n-1}_d E}\circ \itl^{1,n-1}_{d,E}(\xi)\\
			&= d\omega \otimes_A \itpi^{n,n-1}_{d,E} (\xi) + (-1)^{\deg(\omega)} \omega \wedge \itSpenc^{n,0}_{d,E}(\xi).
		\end{split}
		\end{equation}
		\item By definition, $\itSpenc^{n,m}_d\colonequals \Spenc^{1,m}_{d,\itJ^{n-1}_d}\circ \Omega^m_d(\itl^{1,n-1}_d)$.
		By Proposition \ref{prop:spencopsymbol} we know that $\Spenc^{1,m}_d$ is a natural differential operator of order at most $1$, and since it is $A$-linear, $\Omega^m_d(\itl^{1,n-1}_d)$ is a natural differential operator of order $0$, cf.\ Corollary \ref{cor:natDO_order_0,1}.\eqref{cor:natDO_order_0,1:0}.
		By Proposition \ref{prop:properties_natDOs}.\eqref{prop:composition_natDOs_is_natDO}, the composition is also a natural differential operator of order at most $1$.
		
		In order to compute the restriction symbol of $\itSpenc^{n,m}_d$, we consider the following diagram
		\begin{equation}
		\label{diag:acorn}
		\begin{tikzcd}[column sep=70pt,row sep=30pt]
			\Omega^1_d\Omega^m_d\itJ^n_d\ar[drr,bend left=55pt,"\wedge^{1,m}\otimes_A\itpi^{n,n-1}_d"'] \ar[d,hook,"\iota^1_{d,\Omega^m_d\itJ^n_d}"']\ar[r,"\Omega^1_d\Omega^m_d(\itl^{1,n-1}_d)"]&\Omega^1_d\Omega^m_d J^1_d \itJ^{n-1}_d\ar[d,hook,"\iota^1_{d,\Omega^m_dJ^1_d \itJ^{n-1}_d}"']\ar[dr,near start,two heads,"\wedge^{1,m}\otimes_A \pi^{1,0}_{d,\itJ^{n-1}_d}"]\\
			J^1_d\Omega^m_d\itJ^n_d \ar[r,"J^1_d\Omega^m_d(\itl^{1,n-1}_d)"]&J^1_d\Omega^m_d J^1_d \itJ^{n-1}_d\ar[r,near start,"\widetilde{\Spenc}^{1,m}_{d,\itJ^{n-1}_d}"]&\Omega^{m+1}_d \itJ^{n-1}_d\\
			\Omega^m_d\itJ^n_d\ar[urr,bend right=55pt,"\itSpenc^{n,m}_d"]\ar[u,hook,"j^1_{d,\Omega^m_d\itJ^n_d}"] \ar[r,"\Omega^m_d(\itl^{1,n-1}_d)"']&\Omega^m_d J^1_d \itJ^{n-1}_d\ar[u,hook,"j^1_{d,\Omega^m_dJ^1_d \itJ^{n-1}_d}"]\ar[ur,near start,"\Spenc^{1,m}_{d, \itJ^{n-1}_d}"']
		\end{tikzcd}
		\end{equation}
		The top and bottom squares of \eqref{diag:acorn} commute by naturality with respect to $\Omega^m_d(\itl^{1,n-1}_d)$ of $\iota^1_d$ and $j^1_d$, respectively.
		The internal top right triangle commutes by Proposition \ref{prop:spencopsymbol}, and the internal bottom right triangle commutes by definition of lift of a differential operator.
		The top triangle commutes because of \ref{desid:Spencerable:3}, since
		\begin{equation}
			\wedge^{1,m}\otimes_A\pi^{1,0}_{d,\itJ^{n-1}_d}\circ \Omega^1_d\Omega^m_d(\itl^{1,n-1}_d)
			=\wedge^{1,m}\otimes_A(\pi^{1,0}_{d,\itJ^{n-1}_d}\circ \itl^{1,n-1}_d)
			=\wedge^{1,m}\otimes_A\itpi^{n,n-1}_d.
		\end{equation}
		The bottom triangle commutes by definition of $\itSpenc^{n,m}_d$.
		Diagram \eqref{diag:acorn}, proves that the lift of $\itSpenc^{n,m}_d$ to $J^1_d\Omega^m_d \itJ^n_d$ is $\widetilde{\Spenc}^{1,m}_{d,\itJ^{n-1}_d}\circ J^1_d\Omega^m_d(\itl^{1,n-1}_d)$, and also that its restriction symbol is $\wedge^{1,m}\otimes_A\itpi^{n,n-1}_d$.\qedhere
	\end{enumerate}
\end{proof}
The construction of the Spencer operator from the data of \ref{desid:Spencerable:1}, \ref{desid:Spencerable:2}, and \ref{desid:Spencerable:3} is functorial, in a sense made more precise by the following result.
\begin{lemma}
\label{lemma:functoriality_Spencer_construction}
	Let $\dot{J}^n_d$ and $\ddot{J}^n_d$ be two families of functors as in \ref{desid:Spencerable:1}, and let $\dot{l}^{1,n-1}_d$ and $\ddot{l}^{1,n-1}_d$ be their respective natural transformations as in \ref{desid:Spencerable:2}.
	Let $\alpha^n\colon\dot{J}^n_d\to \ddot{J}^n_d$ be a family of $A$-linear natural transformations such that
	\begin{enumerate}
		\item\label{lemma:functoriality_Spencer_construction:1} $\alpha^0=\id$ and $\alpha^1=\id_{J^1_d}$.
		\item\label{lemma:functoriality_Spencer_construction:2} the following diagram commutes for all $n\ge 1$
		\begin{equation}
		\label{diag:l_compatible_with_alphas}
		\begin{tikzcd}[column sep=40pt,row sep=20pt]
			\dot{J}^n_d\ar[d,"\alpha^n"']\ar[r,"\dot{l}^{1,n-1}_d"]&J^1_d\dot{J}^{n-1}_d\ar[d,"J^1_d(\alpha^{n-1})"]\\
			\ddot{J}^n_d\ar[r,"\ddot{l}^{1,n-1}_d"]&J^1_d\ddot{J}^{n-1}_d
		\end{tikzcd}
		\end{equation}
	\end{enumerate}
	Then the natural transformations $\alpha^n$ are compatible with $\dot{\pi}^{n,n-1}_d$ as given in \ref{desid:Spencerable:3} and with the Spencer operators as given in Definition \ref{def:generalised_Spencer}, i.e.\ the following diagrams commute for all $n\ge 1$ and $m\ge 0$.
	\begin{equation}
	\label{diag:pi_and_Spencer_functorial}
		\begin{tikzcd}[column sep=40pt,row sep=20pt]
			\dot{J}^n_d\ar[d,"\alpha^n"']\ar[r,"\dot{\pi}^{n,n-1}_d"]&\dot{J}^{n-1}_d\ar[d,"\alpha^{n-1}"]&\Omega^m_d\dot{J}^n_d\ar[d,"\Omega^m_d(\alpha^n)"']\ar[r,"\dot{\Spenc}^{n,m}_d"]&\Omega^{m+1}_d\dot{J}^{n-1}_d\ar[d,"\Omega^{m+1}_d(\alpha^{n-1})"]\\
			\ddot{J}^n_d\ar[r,"\ddot{\pi}^{n,n-1}_d"]&\ddot{J}^{n-1}_d&\Omega^m_d\ddot{J}^n_d\ar[r,"\ddot{\Spenc}^{n,m}_d"]&\Omega^{m+1}_d\ddot{J}^{n-1}_d
		\end{tikzcd}
	\end{equation}
\end{lemma}
\begin{proof}
First, observe that conditions \eqref{lemma:functoriality_Spencer_construction:1} and \eqref{lemma:functoriality_Spencer_construction:2} are compatible, since when $n=1$, the diagram \eqref{diag:l_compatible_with_alphas} commutes.

In order to prove the commutativity of the left diagram of \eqref{diag:pi_and_Spencer_functorial}, we consider the following diagram, which commutes by definition of the maps \ref{desid:Spencerable:3}, by \eqref{diag:l_compatible_with_alphas}, and by naturality of $\pi^{1,0}_d$ with respect to $\alpha^{n-1}$.
 \begin{equation}
		\begin{tikzcd}[column sep=50pt,row sep=30pt]
			\dot{J}^n_d\ar[rr,bend left=15pt,"\dot{\pi}^{n,n-1}_d"]\ar[d,"\alpha^n"']\ar[r,"\dot{l}^{1,n-1}_d"']&J^1_d\dot{J}^{n-1}_d\ar[d,"J^1_d(\alpha^{n-1})"']\ar[r,two heads,"\pi^{1,0}_{d,\dot{J}^{n-1}_d}"']&\dot{J}^{n-1}_d\ar[d,"\alpha^{n-1}"]\\
			\ddot{J}^n_d\ar[rr,bend right=15pt,"\ddot{\pi}^{n,n-1}_d"']\ar[r,"\ddot{l}^{1,n-1}_d"]&J^1_d\ddot{J}^{n-1}_d\ar[r,two heads,"\pi^{1,0}_{d,\ddot{J}^{n-1}_d}"]&\ddot{J}^{n-1}_d
		\end{tikzcd}
		\end{equation}

In order to prove the commutativity of the right square in \eqref{diag:pi_and_Spencer_functorial} instead we consider the following diagram
 \begin{equation}
		\begin{tikzcd}[column sep=50pt,row sep=30pt]
			\Omega^m_d\dot{J}^n_d\ar[rr,bend left=15pt,"\dot{\Spenc}^{n,m}_d"]\ar[d,"\Omega^m_d(\alpha^n)"']\ar[r,"\Omega^m_d(\dot{l}^{1,n-1}_d)"']&\Omega^m_dJ^1_d\dot{J}^{n-1}_d\ar[d,"\Omega^m_d J^1_d(\alpha^{n-1})"']\ar[r,two heads,"\Spenc^{1,m}_{d,\dot{J}^{n-1}_d}"']&\Omega^{m+1}_d\dot{J}^{n-1}_d\ar[d,"\Omega^{m+1}_d(\alpha^{n-1})"]\\
			\Omega^m_d\ddot{J}^n_d\ar[rr,bend right=15pt,"\ddot{\Spenc}^{n,n-1}_d"']\ar[r,"\Omega^m_d(\ddot{l}^{1,n-1}_d)"]&\Omega^m_dJ^1_d\ddot{J}^{n-1}_d\ar[r,two heads,"\Spenc^{1,m}_{d,\ddot{J}^{n-1}_d}"]&\Omega^{m+1}_d\ddot{J}^{n-1}_d
		\end{tikzcd}
		\end{equation}
		The top and bottom triangles commute by definition of the Spencer operators.
		The left square commutes by applying the functor $\Omega^m_d$ to \eqref{diag:l_compatible_with_alphas}, and the right square commutes by naturality of $\Spenc^{m,1}_d$ with respect to $\alpha^{n-1}$.
\end{proof}
By using Lemma \ref{lemma:functoriality_Spencer_construction}, we can now prove the relations between the Spencer operators corresponding to the jet functors of Proposition \ref{prop:Spencerable_jets}.
\begin{prop}
\label{prop:relations_Spencer_ops}
	The following natural transformations $\alpha^n\colon \dot{J}^n_d\to \ddot{J}^n_d$ commute with the corresponding Spencer operators, i.e.\ the following diagram commutes.
		\begin{equation}
	\label{diag:Spencer_functorial}
		\begin{tikzcd}[column sep=40pt,row sep=20pt]
			\Omega^m_d\dot{J}^n_d\ar[d,"\Omega^m_d(\alpha^n)"']\ar[r,"\dot{\Spenc}^{n,m}_d"]&\Omega^{m+1}_d\dot{J}^{n-1}_d\ar[d,"\Omega^{m+1}_d(\alpha^{n-1})"]\\
			\Omega^m_d\ddot{J}^n_d\ar[r,"\ddot{\Spenc}^{n,m}_d"]&\Omega^{m+1}_d\ddot{J}^{n-1}_d
		\end{tikzcd}
	\end{equation}
	\begin{enumerate}
		\item\label{prop:relations_Spencer_ops:1} $\alpha^n=\gamma^n_d\colon J^n_d A\otimes_A -\to J^n_d$, cf.\ \cite[\jetsproptensorcomparison]{FMW};
		\item\label{prop:relations_Spencer_ops:2} $\alpha^n=\iota_{J^n_d}\colon J^n_d\to J^{(n)}_d$, cf.\ \cite[\jetseqiotajn]{FMW};
		\item\label{prop:relations_Spencer_ops:3} $\alpha^n=\sqhiota{n}\colon J^n_d\hookrightarrow J^{\{n\}}_d$, cf.\ Proposition \ref{prop:Spencerable_jets}.\eqref{prop:Spencerable_jets:4}.
	\end{enumerate}
	If we also assume $\Omega^1_d$ flat in $\ModA$, we also obtain \eqref{diag:Spencer_functorial} for the following natural transformations
	\begin{enumerate}[resume]
		\item\label{prop:relations_Spencer_ops:4} $\alpha^n=\iota_{J^{[n]}_d}\colon J^{[n]}_d\hookrightarrow J^{(n)}_d$, cf.\ \cite[\jetsdefshj]{FMW};
		\item\label{prop:relations_Spencer_ops:5} $\alpha^n=h^n_d\colon J^n_d\hookrightarrow J^{[n]}_d$, cf.\ \cite[\jetspropholinsemi]{FMW};
		\item\label{prop:relations_Spencer_ops:6} $\alpha^n=\iota_{\EJ^{n}_d}\colon \EJ^{n}_d\hookrightarrow J^{n}_d$, cf.\ \cite[\symbolseqphatnepimonocomponents]{Symbol};
		\item\label{prop:relations_Spencer_ops:7} $\alpha^n=\Pphat^n_d\colon\EJ^{n}_d\xrightarrow{\sim} \PJ^{n}_d$, cf.\ \cite[\symbolseqskelprimitivecompare]{Symbol}, and in particular, the natural isomorphisms induce an isomorphism between the Spencer operators.
	\end{enumerate}
\end{prop}
\begin{proof}
The proof for each map consists in verifying that the family of maps considered satisfies the assumptions of Lemma \ref{lemma:functoriality_Spencer_construction}.
Notice that all of the maps considered satisfy the condition Lemma \ref{lemma:functoriality_Spencer_construction}.\eqref{lemma:functoriality_Spencer_construction:1}, essentially by definition.
Now we will prove that Lemma \ref{lemma:functoriality_Spencer_construction}.\eqref{lemma:functoriality_Spencer_construction:2} is satisfied for each map.
\begin{enumerate}
	\item Cf.\ \cite[\jetsproptensorcomparison]{FMW}.
	\item This follows by construction, cf.\ \cite[\jetseqiotajn]{FMW}.
	\item It follows from composing both sides of the left commutative square in \eqref{diag:inclusion_hol_sesqui} with $J^1_d(\sqhiota{n-1})$.
	\item Cf.\ \cite[\jetseqdecsjn]{FMW}.
	\item Cf.\ \cite[\jetspropsemiholcpxfact]{FMW}.
	\item Cf.\ \cite[\symbolslemmaelementalsmmtwo]{Symbol}.
	\item It follows from \cite[\symbolslemmaprimitivesmm]{Symbol} because $\PJ^1_d=J^1_d$.\qedhere
\end{enumerate}
\end{proof}
\begin{rmk}
If we assume $\Omega^1_d$ to be flat in $\ModA$, the natural transformations $\alpha^n$ involved in Proposition \ref{prop:relations_Spencer_ops}, with the exception of $\gamma^n_d$, are all monomorphisms (for \eqref{prop:relations_Spencer_ops:2}, cf.\ \cite[\jetsrmkiotaholinj]{FMW}).
Furthermore, if for any given $m$ we assume $\Omega^m_d$ and $\Omega^{m+1}_d$ are flat in $\ModA$, then the homonymous functors preserve monomorphisms.
Thus, for a pair of functors related by a monomorphism $\dot{J}^n_d \hookrightarrow \ddot{J}^n_d$ as in Proposition \ref{prop:relations_Spencer_ops}, the corresponding Spencer operator on the domain $\dot{\Spenc}^{n,m}_d$ is the restriction of the one on the codomain $\ddot{\Spenc}^{n,m}_d$.
\end{rmk}

\subsection{Spencer complex}
By catenating the jet prolongation and the Spencer operators, we obtain the so-called \emph{Spencer sequence}.
\begin{equation}\label{diag:Spencer_complex}
\begin{tikzcd}[column sep=28pt]
0\ar[r]&[-10pt]\id_{\AMod}\ar[r,"j^n_d"]&J^n_d\ar[r,"\Spenc^{n,0}_d"]&\Omega^1_dJ^{n-1}_d\ar[r,"\Spenc^{n-1,1}_d"]&\Omega^2_dJ^{n-2}_d\ar[r,"\Spenc^{n-2,2}_d"]&\cdots\ar[r,"\Spenc^{2,n-2}_d"]&\Omega^{n-1}_dJ^1_d\ar[r,"\Spenc^{1,n-1}_d"]&\Omega^n_d\ar[r]&[-10pt]0
\end{tikzcd}
\end{equation}
In order to prove that the Spencer sequence is a complex, we first need the following technical lemma involving the functor $\Omega^1_d\ltimes \Omega^2_d$, cf.\ \cite[\jetsssfunctorialityoftwojet]{FMW}.
\begin{lemma}\
\label{lemma:Spencersquared}
\begin{enumerate}
\item\label{lemma:Spencersquared:1} For all $m\ge 0$, the following map is a natural linear differential operator of order at most $1$
\begin{align}
\nu^m_d\colon \Omega^m_d(\Omega^1_d\ltimes\Omega^2_d)\longrightarrow \Omega^{m+2}_d,
&\hfill&
\omega\otimes_A (\alpha+\beta)\longmapsto (-1)^{\deg(\omega)}d\omega\wedge \alpha + \omega\wedge \beta.
\end{align}
Therefore, it induces a natural linear differential operator $\nu^\bullet_d\colon \Omega^\bullet_d(\Omega^1_d\ltimes\Omega^2_d)\to \Omega^\bullet_d$ of order at most $1$.
\item\label{lemma:Spencersquared:2} The $\bk$-linear projection
\begin{align}
\Omega^1_d\ltimes \Omega^2_d\longrightarrow \Omega^2_d,
&\hfill&
\alpha+\beta\longmapsto \beta.
\end{align}
coincides with $\nu^0_d$, and as such it is a natural differential operator of order at most $1$.
\item\label{lemma:Spencersquared:3} $\Spenc^{1,m+1}_d \circ \Spenc^{1,m}_{d, J^1_d}= -\nu^m_d\circ\Omega^m_d(\widetilde{\DH}_d)$.
\item\label{lemma:Spencersquared:4} $\Spenc^{1,1}_d \circ \Spenc^{1,0}_{d, J^1_d}= -\widetilde{\DH}^{II}_d$.
\end{enumerate}
\end{lemma}
\begin{proof}\
\begin{enumerate}
\item We first prove that $\nu^m_d$ is well-defined.
Let $E$ be in $\AMod$, $\omega\in\Omega^m_d$, $\alpha+\beta\in\Omega^1_d\ltimes \Omega^2_d(E)$, and $\lambda\in A$.
By the graded Leibniz rule, we compute the following:
\begin{equation}
\begin{split}
\nu^m_{d,E}(\omega\otimes_A\lambda (\alpha+\beta))
&=\nu^m_{d,E}(\omega\otimes_A(\lambda \alpha+d\lambda\wedge\alpha+\lambda\beta))\\
&=(-1)^{\deg(\omega)}d\omega\wedge\lambda \alpha+\omega\wedge (d\lambda\wedge\alpha+\lambda\beta)\\
&=(-1)^{\deg(\omega)}(d\omega)\lambda\wedge \alpha+\omega\wedge d\lambda\wedge\alpha+\omega\lambda\wedge\beta\\
&=(-1)^{\deg(\omega)}d(\omega\lambda)\wedge\alpha+\omega\lambda\wedge\beta\\
&=\nu^m_{d,E}(\omega\lambda\otimes_A (\alpha+\beta)).
\end{split}
\end{equation}
Thus, $\nu^m_d$ is well-defined.
Furthermore, it forms a natural transformation via the tensor product.

We show that $\nu^m_d$ is a differential operator of order at most $1$ by showing that each component $E$ is a differential operator of order at most $1$, cf.\ Corollary \ref{cor:natDO_order_0,1}\eqref{cor:natDO_order_0,1:1}.
We do so via the criterion given by \cite[\jetsstorderwrtd]{FMW}.
We thus show that the universal lift of $\nu^m_{d,E}$ vanishes on elements in $N^1_d\Omega^m_d(\Omega^1_d\ltimes\Omega^2_d)(E)$.
Thanks to \cite[\jetsrmktensorNddiffop]{FMW}, it is sufficient to show that
\begin{equation}
\sum_j a_j\nu^m_{d,E}(b_j\omega\otimes_A (\alpha+\beta))
=0
\end{equation}
for all $\sum_j a_j\otimes b_j\in N^1_d(A)$ and $\omega\otimes_A (\alpha+\beta)\in\Omega^m_d(\Omega^1_d\ltimes \Omega^2_d)(E)$, cf.\ \cite[\jetsrmktensorNddiffop]{FMW}.
We have indeed
\begin{equation}
\begin{split}
&\sum_j a_j\nu^m_{d,E}(b_j\omega\otimes_A (\alpha+\beta))\\
&\qquad=\sum_j a_j(-1)^{\deg(b_j\omega)}d(b_j\omega)\wedge\alpha+\sum_j a_jb_j\omega\wedge\beta\\
&\qquad=(-1)^{\deg(\omega)}\sum_j a_jdb_j\wedge\omega\wedge\alpha+(-1)^{\deg(\omega)}\sum_j a_jb_jd\omega\wedge\alpha+\sum_j a_jb_j\omega\wedge\beta.
\end{split}
\end{equation}
Every term in the last expression vanishes since $\sum_j a_j\otimes b_j\in N^1_d(A)$, cf.\ \cite[\jetseqdefinitionNdE]{FMW}.
Explicitly, the lift has the form:
\begin{align}
\widetilde{\nu}^m_d\colon J^1_d\Omega^m_d(\Omega^1_d\ltimes\Omega^2_d)\longrightarrow \Omega^{m+2}_d,
&\hfill&
[a\otimes b]\otimes_A \omega\otimes_A (\alpha+\beta)\longmapsto a(-1)^{\deg(\omega)}d(b\omega)\wedge \alpha + ab\omega\wedge \beta.
\end{align}

The natural differential operator $\nu^\bullet_d$ is obtained by acting as $\nu^m_d$ on the component $m$.
\item This follows directly by restricting \eqref{lemma:Spencersquared:1} to the case $m=0$.
\item We will show this for every component $E$ in $\AMod$, so let $\omega \otimes_A [a\otimes b] \otimes_A [c\otimes e]$ be an element of $\Omega^m_d J^{(2)}_d E$.
We compute the following
\begin{equation}\label{eq:Spencersquared}
\begin{split}
&\Spenc^{1,m+1}_{d,E} \circ \Spenc^{1,m}_{d,J^1_d E} (\omega \otimes_A[a\otimes b] \otimes_A [c\otimes e])\\
&\qquad=\Spenc^{1,m+1}_{d,E} (d(\omega a)b \otimes_A [c\otimes e])\\
&\qquad = d(d(\omega a)bc)\otimes_A e\\
&\qquad = 0-(-1)^{\deg(\omega)} d(\omega a) \wedge d(b c)\otimes_A e\\
&\qquad = -(-1)^{\deg(\omega)} d\omega\wedge ad(bc)\otimes_A e -\omega\wedge da \wedge d(bc) \otimes_A e\\
&\qquad = -(-1)^{\deg(\omega)}d\omega \wedge \widetilde{\DH}^I_{d,E}([a\otimes b] \otimes_A [c\otimes e]) - \omega\wedge \widetilde{\DH}^{II}_{d,E}([a\otimes b] \otimes_A [c\otimes e])\\
&\qquad = -\nu^m_d\circ\Omega^m_d(\widetilde{\DH}_{d,E})(\omega \otimes_A [a\otimes b] \otimes_A [c\otimes e]),
\end{split}
\end{equation}
which completes the proof.
\item This formula is obtained by \eqref{lemma:Spencersquared:3} for $m=0$, where $\nu^0_d\circ\widetilde{\DH}=\widetilde{\DH}^{II}_d$ by definition of $\widetilde{\DH}$ and \eqref{lemma:Spencersquared:2}.
\qedhere
\end{enumerate}
\end{proof}
\begin{theo}[Spencer complex]\label{theo:Spencer_is_complex}
The Spencer sequence is a complex, i.e.
\begin{enumerate}
\item $\Spenc^{n,0}_d \circ j^n_d=0$;
\item $\Spenc^{n-1,m+1}_d \circ \Spenc^{n,m}_d = 0$ for all $n\ge 2$ and $m\ge 0$.
\end{enumerate}
\end{theo}
\begin{proof}\
\begin{enumerate}
\item By \eqref{eq:explicit_Sh0}, we have the following equality.
\begin{equation}
\Spenc^{n,0}_d\circ j^n_d
= -\rho_{d,J^{n-1}_d}\circ l^{1,n-1}_d \circ j^n_d
= -\rho_{d,J^{n-1}_d}\circ j^1_{d,J^{n-1}_d} \circ j^{n-1}_d
=0.
\end{equation}
The second equality follows from the definition of $j^n_d$, cf.\ \cite[\jetslemmaholprol]{FMW}, and the last follows from the fact that $\rho_d\circ j^1_d=0$, essentially by definition of $\rho_d$ as a split in the biproduct structure given by the $1$-jet exact sequence in $\Mod$, cf.\ \cite[\jetssssSplitting]{FMW}.

\item Now let $n\ge 2$ and $m\ge 0$ and consider the following diagram.
\begin{equation}\label{diag:beak}
\begin{tikzcd}[column sep=50pt,row sep=30pt]
\Omega^m_d J^n_d\ar[d,"\Omega^m_d(l^{1,n-1}_d)"']\ar[dr,"\Spenc^{n,m}_d"]\\
\Omega^m_d J^1_d J^{n-1}_d\ar[d,"\Omega^m_dJ^1_d(l^{1,n-2}_d)"']\ar[r,two heads,"\Spenc^{1,m}_{d,J^{n-1}_d}"]&\Omega^{m+1}_d J^{n-1}_d\ar[d,"\Omega^{m+1}_d(l^{1,n-2}_d)"']\ar[dr,"\Spenc^{n-1,m+1}_d"]\\
\Omega^m_d J^1_d J^1_d J^{n-2}_d\ar[d,two heads,"\Omega^m_d (\widetilde{\DH}_{J^{n-2}_d})"']\ar[r,two heads,"\Spenc^{1,m}_{d,J^1_d J^{n-2}_d}"']&\Omega^{m+1}_d J^1_d J^{n-2}_d\ar[r,two heads,"\Spenc^{1,m+1}_{d,J^{n-2}_d}"']&\Omega^{m+2}_d J^{n-2}_d\\
\Omega^m_d(\Omega^1_d\ltimes \Omega^2_d) J^{n-2}_d\ar[rru,bend right=10pt,"-\nu^m_{d J^{n-2}_d}"']
\end{tikzcd}
\end{equation}
By Lemma \ref{lemma:SpencerDO_wrt_low_indices}, the two triangles commute.
By the naturality of $\Spenc^{1,m}_d$ with respect to $l^{1,n-2}_d$ we obtain the commutativity of the central square, and the commutativity of the bottom square follows from Lemma \ref{lemma:Spencersquared}.\eqref{lemma:Spencersquared:3}.

Since \eqref{diag:beak} commutes, we can prove that the top right diagonal composition vanishes by proving that the left vertical composition vanishes.
This follows from the definition of holonomic jet functor, cf.\ \cite[\jetsdefnjet]{FMW}.\qedhere
\end{enumerate}
\end{proof}
\begin{defi}
We call the cohomology of the Spencer complex \eqref{diag:Spencer_complex} the \emph{Spencer cohomology}, and we denote the cohomology group at $\Omega^m_d J^n_d$ by $H^{n,m}_{\Spenc_d}$.
\end{defi}
The reason we do not need a symbol for the cohomology at $J^n_d$ is that it is always zero, cf.\ Proposition \ref{prop:Spencer_cohomology_vanishing_at_extremals}.

We will prove that the Spencer complex is always exact in the extremal degrees, but in order to do that, we first need the following lemma.
\begin{lemma}
\label{lemma:pullback_j1_l}
The following is a pullback square in $\Mod$
\begin{equation}
\begin{tikzcd}
\id_{\AMod} \arrow[dr, phantom, "\lrcorner", very near start]\ar[r,hook,"j^n_d"]\ar[d,hook,"j^{n-1}_d"']&J^n_d\ar[d,hook,"l^{1,n-1}_d"]\\
J^{n-1}_d\ar[r,hook,"j^1_{d,J^{n-1}_d}"']&J^1_d J^{n-1}_d
\end{tikzcd}
\end{equation}
In other words, the intersection of $J^n_d E$ and $j^1_{d,J^{n-1}_d E}(J^{n-1}_d E)$ in $J^1_d J^{n-1}_d E$ is $l^{1,n-1}_{d,E}(j^n_d(E))$.
\end{lemma}
\begin{proof}
Since the pullback of monomorphisms corresponds to their intersection, we will prove this result in this second formulation, and we will do so by induction on $n$ and at each component $E$ in $\AMod$.
More precisely, we will prove that for all $\xi\in J^{n-1}_d E$, if $j^1_{d,J^{n-1} E}(\xi)\in J^n_d E$, then $j^1_{d,J^{n-1}_d E}(\xi)=l^{1,n-1}_{d,E}\circ j^n_{d,E}(\pi^{n-1,0}_{d,E} (\xi))$.

For $n=1$ it is tautologically true.

For $n>1$, we compute the conditions for which the element $j^1_{d,J^{n-1}_d E}(\xi)$ belongs to $J^n_d E$, namely when it vanishes if we apply $\widetilde{\DH}_{d,J^{n-2}_d E}\circ J^1_d(l^{1,n-2}_{d,E})$.
Let $l^{1,n-2}_{d,E}(\xi)=\sum_j [a_j\otimes b_j]\otimes_A \xi_j\in J^1_d J^{n-2}_d E$, then we have:
\begin{equation}
\begin{split}
0
&=\widetilde{\DH}_{d,J^{n-2}_d E}\circ J^1_d(l^{1,n-2}_{d,E})\left(j^1_{d,J^{n-1}_d E}(\xi)\right)\\
&=\sum_j \widetilde{\DH}_{d,J^{n-2}_d E}\left(j^1_{d,J^1_d J^{n-2}_d E}(l^{1,n-2}_{d,E}(\xi))\right)\\
&=\sum_j \widetilde{\DH}_{d,J^{n-2}_d E}\left([1\otimes 1]\otimes_A [a_j\otimes b_j]\otimes_A \xi_j\right)\\
&=\sum_j \widetilde{\DH}_{d,A}\left([1\otimes 1]\otimes_A [a_j\otimes b_j]\right)\otimes_A \xi_j\\
&=\sum_j (da_j)b_j\otimes_A \xi_j+0\\
&=-\rho_{d,J^{n-2}_d E}\left(\sum_j [a_j\otimes b_j]\otimes_A \xi_j\right)\\
&=-\rho_{d,J^{n-2}_d E} \circ l^{1,n-2}_{d,E}(\xi)
\end{split}
\end{equation}
We thus infer that $l^{1,n-2}_{d,E}(\xi)\in \ker(\rho_{d,J^{n-2}_d E})$.
Since $\rho_d$ and $j^1_d$ are the left and right splits that realize the $1$-jet short exact sequence as a biproduct, cf.\ \cite[\jetssssSplitting]{FMW}, we know that $\ker(\rho_d)=\im(j^1_d)$.
It follows that there exists $\xi'\in J^{n-2}_d$ such that $l^{1,n-2}_{d,E}(\xi)=j^1_{d,J^{n-2}_d E}(\xi')$.
By applying $\pi^{1,0}_{d,J^{n-2}_d E}$ to both terms of the last equality, we obtain $\xi'=\pi^{1,0}_{d,J^{n-2}_d E}(l^{1,n-2}_{d,E}(\xi))=\pi^{n-1,n-2}_{d,E}(\xi)$, and thus
\begin{equation}
l^{1,n-2}_{d,E}(\xi)
=j^1_{d,J^{n-2}_d E}\left(\pi^{n-1,n-2}_{d,E}(\xi)\right).
\end{equation}
Now we have an element $\pi^{n-1,n-2}_{d,E}(\xi)\in J^{n-2}_d E$ such that $j^1_{d,J^{n-2}_d E}\left(\pi^{n-1,n-2}_{d,E}(\xi)\right)=l^{1,n-2}_{d,E}(\xi)$ and it thus belongs to $J^{n-1}_d E$.
Hence, we can apply the inductive hypothesis, to obtain
\begin{equation}
l^{1,n-2}_{d,E}(\xi)
=j^1_{d,J^{n-2}_d E}(\pi^{n-1,n-2}_{d,E}(\xi))
=l^{1,n-2}_{d,E}\circ j^{n-1}_{d,E}(\pi^{n-2,0}_{d,E} (\pi^{n-1,n-2}_{d,E}(\xi)))
=l^{1,n-2}_{d,E}\circ j^{n-1}_{d,E}(\pi^{n-1,0}_{d,E}(\xi)).
\end{equation}
Since $l^{1,n-2}_{d,E}$ is a mono, we have that $\xi=j^{n-1}_{d,E}(\pi^{n-1,0}_{d,E}(\xi))$, which in turn implies
\begin{equation}
j^1_{d,J^{n-1} E}(\xi)
=j^1_{d,J^{n-1} E}(j^{n-1}_{d,E}(\pi^{n-1,0}_{d,E}(\xi)))
=l^{1,n-1}_d\circ j^n_{d,E}(\pi^{n-1,0}_{d,E}(\xi)),
\end{equation}
thus proving the inductive step.
The statement follows by induction.
\end{proof}
\begin{prop}\label{prop:Spencer_cohomology_vanishing_at_extremals}
The Spencer complex is exact in degrees $0$, $1$, and $n+1$, so in particular $H^{n,0}_{\Spenc_d}=H^{0,n}_{\Spenc_d}=0$.
\end{prop}
\begin{proof}
The Spencer complex is exact in degree $0$ because $j^n_d$ is a (natural) monomorphism.
Exactness in $n-1$ follows from the fact that $\Spenc^{1,n-1}_d$ is a (natural) epimorphism, cf.\ Lemma \ref{lemma:SpencerDO_wrt_low_indices}.

We will now prove the exactness in degree $1$, which in this setting is equivalent to showing that $j^n_d$ is the kernel inclusion of $\Spenc^{n,0}_d$.
Consider the following diagram in the functor category $\AMod\to \Mod$.
\begin{equation}\label{diag:exactnessSpencer1snake}
\begin{tikzcd}[column sep=50pt]
&[-20pt]0\ar[r]\ar[d]&0\ar[r]\ar[d]\ar[draw=none]{ddd}[name=X, anchor=center]{}&\ker(\Spenc^{n,0}_d)\ar[d,hook]
\ar[hook,rounded corners=12pt,
	to path={ -- ([xshift=50pt]\tikztostart.east)
		|- ([yshift=5pt]X.center) \tikztonodes
		-| ([xshift=-60pt]\tikztotarget.west)
		-- (\tikztotarget)},near start,"\partial"]{dddll}
&[-20pt]\\
&0\ar[r]\ar[d]&J^n_d\ar[r,equals]\ar[d,near end,hook,"l^{1,n-1}_d"]&J^n_d\ar[r]\ar[d,near end,"\Spenc^{n,0}_d"]&0\\[10pt]
0\ar[r]&J^{n-1}_d\ar[d,equals]\ar[r,hook,"j^1_{d,J^{n-1}_d}"]&J^1_d J^{n-1}_d\ar[r,two heads,"\rho_{d,J^{n-1}_d}"]\ar[d,two heads]&\Omega^1_d J^{n-1}_d\ar[r]\ar[d,two heads]&0\\
&J^{n-1}_d\ar[r]&\coker(l^{1,n-1}_d)\ar[r,two heads]&\coker(\Spenc^{n,0}_d)\ar[r]&0
\end{tikzcd}
\end{equation}
The diagram is constructed by taking the central right square, obtained from \eqref{eq:explicit_Sh0}, completing it to a morphism of short exact sequences by adding the kernel of the two horizontal maps.
The other maps are the kernels and the cokernels of the vertical maps, which constitute a long exact sequence by the snake lemma.

By the snake lemma, we can see $\partial$ as an inclusion of $\ker(\Spenc^{n,0}_d)$ into $J^{n-1}_d$.
By the way $\partial$ is constructed in this case, we can deduce more about $\ker(\Spenc^{n,0}_d)$.
For that purpose, consider $E$ in $\AMod$ and let $\xi\in \ker(\Spenc^{n,0}_{d,E})$.
We can see $\xi$ in $J^n_d E$ via its natural inclusion, and then in turn embed it into $J^1_d J^{n-1}_d E$ via $l^{1,n-1}_{d,E}$.
In this particular case, where the bottom leftmost vertical morphism in \eqref{diag:exactnessSpencer1snake} is the identity, the construction of $\partial_E$ tells us that $l^{1,n-1}_{d,E}(\xi)=j^1_{d,J^{n-1}_d E}\circ \partial_E(\xi)$.
Consequently, we know that $\xi$ factors through the pullback of $l^{1,n-1}_{d,E}$ and $j^1_{d,J^{n-1}_d E}$, cf.\ Lemma \ref{lemma:pullback_j1_l}.
This yields that $\xi$ is in the image of $j^n_{d,E}$, or, in other words, that every element in $\ker(\Spenc^{n,0}_{d,E})$ is contained in the image of $j^n_{d,E}$.
Vice versa, we know that every element in the image of $j^n_{d,E}$ belongs to $\ker(\Spenc^{n,0}_{d,E})$ by Theorem \ref{theo:Spencer_is_complex}, and thus we have a double inclusion, proving that $\ker(\Spenc^{n,0}_{d,E})=E$ with kernel inclusion given by $j^n_{d,E}$.
This in turns shows that $j^n_d$ is the kernel inclusion of $\Spenc^{n,0}_d$.
\end{proof}

For applications in subsequent sections, we will consider Spencer operators in the context of sesquiholonomic jet functors.
To that end, we define the following map:
\begin{equation}
\overline{\Spenc}^{\{n,m\}}_d\colonequals \Spenc^{1,m}_{d,J^{n-1}_d}\circ \Omega^m_d\left(l^{\{n\}}_d\right)\colon \Omega^m_d J^{\{n\}}_d\longrightarrow\Omega^{m+1}_d J^{n-1}_d.
\end{equation}
We will now show a few results involving this map and the sesquiholonomic Spencer operators.
\begin{prop}
Let $\Omega^\bullet_d$ be an exterior algebra on a $\bk$-algebra $A$, then the following diagram commutes for all $m\ge 0$ and $n\ge 1$.
\begin{equation}\label{diag:sesquiSpenc_sub_holSpenc}
\begin{tikzcd}[column sep=50pt]
\Omega^m_d J^n_d\ar[r,"\Spenc^{n,m}_d"]\ar[d,"\Omega^m_d(\sqhiota{n})"']&\Omega^{m+1}_d J^{n-1}_d\ar[d,"\Omega^{m+1}_d(\sqhiota{n-1})"]\\
\Omega^m_d J^{\{n\}}_d\ar[ru,"\overline{\Spenc}^{\{n,m\}}_d"]\ar[r,"\Spenc^{\{n,m\}}_d"']&\Omega^{m+1}_d J^{\{n-1\}}_d
\end{tikzcd}
\end{equation}
In particular, the holonomic Spencer complex is a subsequence of the semiholonomic Spencer sequence.
\end{prop}
\begin{proof}
Consider the following diagram
\begin{equation}
\begin{tikzcd}[column sep=60pt, row sep=50pt]
\Omega^m_d J^n_d\ar[rr,bend left=12,"\Spenc^{n,m}_d"]\ar[d,"\Omega^m_d(\sqhiota{n})"']\ar[r,"\Omega^m_d (l^{1,n-1}_d)"']&\Omega^m_d J^1_d J^{n-1}_d\ar[r,two heads,"\Spenc^{1,m}_{d,J^{n-1}_d}"']\ar[d,"\Omega^m_d J^1_d(\sqhiota{n-1})"]&\Omega^{m+1}_d J^{n-1}_d\ar[d,"\Omega^{m+1}_d(\sqhiota{n-1})"]\\
\Omega^m_d J^{\{n\}}_d\ar[ru,"\Omega^m_d(l^{\{n\}}_d)"]\ar[rr,bend right=12,"\Spenc^{\{n,m\}}_d"']\ar[r,"\Omega^m_d (l^{\{1,n-1\}}_d)"]&\Omega^m_d J^1_d J^{\{n-1\}}_d\ar[r,two heads,"\Spenc^{1,m}_{d,J^{\{n-1\}}_d}"]&\Omega^{m+1}_d J^{\{n-1\}}_d
\end{tikzcd}
\end{equation}
The top and bottom triangles commute by Lemma \ref{lemma:SpencerDO_wrt_low_indices} and the definition of sesquiholonomic Spencer operators.
The triangles in the left square commute by definition of $\sqhiota{n}$ and $l^{\{1,n-1\}}_d$.
Finally, the square on the right commutes by naturality of $\Spenc^{1,m}_d$ with respect to $\sqhiota{n-1}$.
We can deduce \eqref{diag:sesquiSpenc_sub_holSpenc} by definition of $\overline{\Spenc}^{\{n,m\}}_d$.
\end{proof}
In the following lemma we show, in particular, that the sesquiholonomic Spencer sequence is not a complex unless $J^n_d$ is the zero functor.
\begin{lemma}
\label{lemma:sesquiholonomic_Spencer}
Let $\Omega^\bullet_d$ be an exterior algebra over the $\bk$-algebra $A$, then:
\begin{enumerate}
\item \label{lemma:sesquiholonomic_Spencer:1}$\Spenc^{n-1,1}_d\circ \overline{\Spenc}^{\{n,0\}}_d =-\widetilde{\DH}^{II}_{J^{n-2}_d} \circ J^1_d(l^{1,n-2}_d)\circ l^{\{n\}}_d$;
\item\label{lemma:sesquiholonomic_Spencer:2} $\ker(\Spenc^{n-1,1}_d\circ \overline{\Spenc}^{\{n,0\}}_d)=J^n_d$;
\item\label{lemma:sesquiholonomic_Spencer:3} $\ker(\Spenc^{\{n-1,1\}}_d\circ \Spenc^{\{n,0\}}_d)\supseteq J^n_d$, and the equality holds if $\Omega^2_d$ is flat in $\Mod_A$.
\end{enumerate}
\end{lemma}
\begin{proof}\
\begin{enumerate}
\item Consider the following diagram
\begin{equation}\label{diag:sesquiholonomic_Spencer:1}
\begin{tikzcd}[column sep=60pt, row sep=30pt]
J^{\{n\}}_d\ar[d,hook,"l^{\{n\}}_d"']\ar[dr,"\overline{\Spenc}^{\{n,0\}}_d"]\\
J^1_d J^{n-1}_d\ar[d,"J^1_d(l^{1,n-2}_d)"']\ar[r,two heads,"\Spenc^{1,0}_{d,J^{n-1}_d}"]&\Omega^1_d J^{n-1}_d\ar[d,"\Omega^1_d (l^{1,n-2}_d)"']\ar[dr,"\Spenc^{n-1,1}_d"]\\
J^1_d J^1_d J^{n-2}_d\ar[rr,two heads,bend right=15pt,"-\widetilde{\DH}^{II}_d"']\ar[r,two heads,"\Spenc^{1,0}_{d,J^1_d J^{n-2}_d}"]&\Omega^1_d J^1_d J^{n-2}_d\ar[r,two heads,"\Spenc^{1,1}_{d,J^{n-2}_d}"]&\Omega^2_d J^{n-2}_d
\end{tikzcd}
\end{equation}
The top triangle commutes by the definition of $\overline{\Spenc}^{\{n,0\}}_d$, the right triangle commutes by Lemma \ref{lemma:SpencerDO_wrt_low_indices}, and the bottom triangle commutes by Lemma \ref{lemma:Spencersquared}.\eqref{lemma:Spencersquared:4}.
The square commutes by naturality of $\Spenc^{1,0}_d$ with respect to $l^{1,n-2}_d$.
The commutativity of \eqref{diag:sesquiholonomic_Spencer:1} gives \eqref{lemma:sesquiholonomic_Spencer:1}.
\item It follows from the definition of holonomic jet functor, cf.\ \cite[\jetsdefnjet]{FMW}, and \cite[\jetslemmaimageD]{FMW} that the kernel of $\widetilde{\DH}^{II}_{J^{n-2}_d} \circ J^1_d(l^{1,n-2}_d)\circ l^{\{n\}}_d$ is precisely the subfunctor $J^n_d$.
\item Consider the following commutative diagram, obtained by composing two consecutive diagrams of the form \eqref{diag:sesquiSpenc_sub_holSpenc}.
\begin{equation}\label{diag:sesquiSpenc_sub_holSpenc_lowm}
\begin{tikzcd}[column sep=50pt]
J^n_d\ar[r,"\Spenc^{n,0}_d"]\ar[d,"\sqhiota{n}"']&\Omega^1_d J^{n-1}_d\ar[d,"\Omega^1_d(\sqhiota{n-1})"]\ar[r,"\Spenc^{n-1,1}_d"]&\Omega^2_d J^{n-2}_d\ar[d,"\Omega^2_d(\sqhiota{n-2})"]\\
J^{\{n\}}_d\ar[ru,"\overline{\Spenc}^{\{n,0\}}_d"]\ar[r,"\Spenc^{\{n,0\}}_d"']&\Omega^1_d J^{\{n-1\}}_d\ar[r,"\Spenc^{\{n-1,1\}}_d"']&\Omega^2_d J^{\{n-2\}}_d
\end{tikzcd}
\end{equation}
In particular, it shows that $\Spenc^{\{n-1,1\}}_d\circ \Spenc^{\{n,0\}}_d=\Omega^2_d(\sqhiota{n-2}) \circ\Spenc^{n-1,1}_d\circ \overline\Spenc^{\{n,0\}}_d$, which, together with \eqref{lemma:sesquiholonomic_Spencer:2}, implies the inclusion
\begin{equation}
\ker\left(\Spenc^{\{n-1,1\}}_d\circ \Spenc^{\{n,0\}}_d\right)
=\ker\left(\Omega^2_d(\sqhiota{n-2}) \circ\Spenc^{n-1,1}_d\circ \overline\Spenc^{\{n,0\}}_d\right)
\supseteq
\ker\left(\Spenc^{n-1,1}_d\circ \overline\Spenc^{\{n,0\}}_d\right)
=J^n_d,
\end{equation}
which is an equality if $\Omega^2_d(\sqhiota{n-2})$ is a monomorphism, which in turn happens if $\Omega^2_d$ is flat in $\ModA$.\qedhere
\end{enumerate}
\end{proof}
\begin{rmk}\label{rmk:holonomic_alternate_def_spencer}
	Lemma \ref{lemma:sesquiholonomic_Spencer}.\eqref{lemma:sesquiholonomic_Spencer:2} can be used as an alternative definition of holonomic jet functors $J^n_d$.
	The degrees $n=0,1$ are given as in \cite[\jetssonejetfunctor]{FMW}, and the higher grades are defined by induction on $n$ as $J^n_d\colonequals\ker(\Spenc^{n-1,1}_d\circ \overline{\Spenc}^{\{n,0\}}_d)$ (or even $\ker(\Spenc^{\{n-1,1\}}_d\circ \Spenc^{\{n,0\}}_d)$ when $\Omega^2_d$ is flat in $\ModA$).
	Lemma \ref{lemma:sesquiholonomic_Spencer} shows that this definition is equivalent to \cite[\jetsdefnjet]{FMW}.
	Notice that this definition is well-posed, as the inductive hypothesis gives us all the objects and morphisms appearing in \eqref{diag:sesquiSpenc_sub_holSpenc_lowm} except for those involving $J^n_d$, and once we obtain the latter, we can build the objects and maps that are necessary for proving the step $n+1$.
\end{rmk}
\subsubsection{Elemental and primitive Spencer complex}
In Lemma \ref{lemma:sesquiholonomic_Spencer}, we have shown that holonomic jet functors are, in a certain sense, maximal with respect to the property of the Spencer sequence being a complex.
However, we can still investigate conditions for Spencer sequences on subfunctors of the holonomic jet functors to be complexes.
We will now study the case of elemental and primitive jet functors.
Throughout this subsection we will assume $\Omega^1_d$ to be flat in $\ModA$, so that we can build elemental and primitive Spencer operators, cf.\ Definition \ref{def:generalised_Spencer} and Proposition \ref{prop:Spencerable_jets}.
\begin{rmk}
Under the assumption that $\Omega^1_d$ is flat in $\ModA$, we have that $\EJ^n_d\simeq \PJ^n_d$, so all the results derived in this subsection concerning the elemental jets could equivalently be phrased in terms of the primitive jets.
\end{rmk}

Analogously to the holonomic case, we can construct the \emph{elemental Spencer sequence} by catenating the elemental jet prolongation and the appropriate elemental Spencer operators.
\begin{equation}
\label{diag:elemental_Spencer_complex}
\begin{tikzcd}[column sep=28pt]
0\ar[r]&[-10pt]\id_{\AMod}\ar[r,"\Ej^n_d"]&\EJ^n_d\ar[r,"\ESpenc^{n,0}_d"]&\Omega^1_d\EJ^{n-1}_d\ar[r,"\ESpenc^{n-1,1}_d"]&\Omega^2_d\EJ^{n-2}_d\ar[r,"\ESpenc^{n-2,2}_d"]&\cdots\ar[r,"\ESpenc^{2,n-2}_d"]&\Omega^{n-1}_d\EJ^1_d\ar[r,"\ESpenc^{1,n-1}_d"]&\Omega^n_d\ar[r]&[-10pt]0
\end{tikzcd}
\end{equation}
We can prove that this construction yields a complex via the following result.
\begin{theo}[Elemental Spencer complex]\label{theo:Elemental_Spencer_is_complex}
	The elemental Spencer sequence is a complex, i.e.
	\begin{enumerate}
		\item $\ESpenc^{n,0}_d \circ \Ej^n_d=0$;
		\item $\ESpenc^{n-1,m+1}_d \circ \ESpenc^{n,m}_d = 0$ for all $n\ge 2$ and $m\ge 0$.
	\end{enumerate}
\end{theo}
\begin{proof}\
	\begin{enumerate}
		\item We consider the following diagram:
		\begin{equation}
		\label{diag:Ej_is_jernel}
		\begin{tikzcd}[column sep=28pt]
			\id_{\AMod}\ar[d,equals]\ar[r,hook,"\Ej^n_d"]&\EJ^n_d\ar[r,"\ESpenc^{n,0}_d"]\ar[d,hook,"\iota_{\EJ^n_d}"]&\Omega^1_d\EJ^{n-1}_d\ar[d,hook,"\Omega^1_d(\iota_{\EJ^{n-1}_d})"]\\
			\id_{\AMod}\ar[r,hook,"j^n_d"]&J^n_d\ar[r,"\Spenc^{n,0}_d"]&\Omega^1_d J^{n-1}_d
		\end{tikzcd}
		\end{equation}
		The left square commutes by definition of $\Ej^n_d$, cf.\ \cite[\symbolssselementaljets]{Symbol}, and the right square commutes by Proposition \ref{prop:relations_Spencer_ops}.\eqref{prop:relations_Spencer_ops:6}.
		Since $\Omega^1_d(\iota_{\EJ^{n-1}_d})$ is a monomorphism and the bottom horizontal composition vanishes, so does the top horizontal composition, which proves the statement.
		\item Now let $n\ge 2$ and $m\ge 0$ consider the natural epi $\Ephat^n_d\colon A\otimes -\twoheadrightarrow \EJ^n_d$.
		By the right exactness of tensor functors, we obtain, another natural epimorphism with component at $E$ in $\AMod$ as follows:
		\begin{align}
			\Omega^m_d(\Ephat^n_d)\colon \Omega^m_d \otimes_A A\otimes E\cong\Omega^m_d \otimes E\longtwoheadrightarrow \Omega^m_d\EJ^n_d E,
			&\hfill&
			\omega\otimes e\longmapsto \omega\otimes_A \Ej^n_{d,E}(e).
		\end{align}
		 Now consider the following diagram.
		\begin{equation}
		\label{diag:elemental_Spencer_is_complex}
		\begin{tikzcd}[column sep=45pt,row sep=30pt]
			&\Omega^m_d\otimes \id_{\AMod}\ar[ddl,"\Omega^m_d\otimes\Ej^{n-2}_d"']\ar[d,"\Omega^m_d\otimes\Ej^{n-1}_d"]\ar[r,two heads,"\Omega^m_d(\Ephat^n_d)"]&\Omega^m_d \EJ^n_d\ar[d,"\Omega^m_d(\El^{1,n-1}_d)"']\ar[dr,"\ESpenc^{n,m}_d"]\\
			&\Omega^m_d\otimes\EJ^{n-1}_d\ar[d,"\Omega^m_d\otimes\El^{1,n-2}_d"]\ar[r,two heads,"\Omega^m_d(\phat^1_{d,\EJ^{n-1}_d})"]&\Omega^m_d J^1_d \EJ^{n-1}_d\ar[d,"\Omega^m_dJ^1_d(\El^{1,n-2}_d)"']\ar[r,two heads,"\Spenc^{1,m}_{d,\EJ^{n-1}_d}"]&[-10pt]\Omega^{m+1}_d \EJ^{n-1}_d\ar[d,"\Omega^{m+1}_d(\El^{1,n-2}_d)"']\ar[dr,"\ESpenc^{n-1,m+1}_d"]\\
			\Omega^m_d\otimes \EJ^{n-2}_d\ar[rr,bend right=20pt,"\Omega^m_d\otimes_A j^{(2)}_{d,\EJ^{n-2}_d}=\Omega^m_d(\phat^{(2)}_{d,J^{n-2}_d})"']\ar[r,"\Omega^m\otimes j^1_{d,\EJ^{n-2}_d}"]&\Omega^m_d\otimes J^1_d\EJ^{n-2}_d\ar[r,two heads,"\Omega^m_d(\phat^1_{d,J^1_d\EJ^{n-1}_d})"]&\Omega^m_d J^1_d J^1_d \EJ^{n-2}_d\ar[r,two heads,"\Spenc^{1,m}_{d,J^1_d \EJ^{n-2}_d}"']&\Omega^{m+1}_d J^1_d \EJ^{n-2}_d\ar[r,two heads,"\Spenc^{1,m+1}_{d,\EJ^{n-2}_d}"']&[-10pt]\Omega^{m+2}_d \EJ^{n-2}_d
		\end{tikzcd}
		\end{equation}
		By Definition \ref{def:generalised_Spencer}, the two triangles on the right commute.
		The rightmost square commutes by naturality of $\ESpenc^{1,m}_d$ with respect to $\El^{1,n-2}_d$.
		In order to prove the naturality of the top central square, consider the following commutative square, cf.\ \cite[\symbolslemmaelementalsmm]{Symbol}.
		\begin{equation}
		\label{diag:commutativity_Ejs}
		\begin{tikzcd}
			\id_{\AMod}\ar[r,hook,"\Ej^n_d"]\ar[d,hook,"\Ej^{n-1}_d"']&\EJ^n\ar[d,hook,"\El^{1,n-1}_d"]\\
			\EJ^{n-1}_d\ar[r,hook,"j^1_{d,\EJ^{n-1}_d}"]&J^1_d \EJ^{n-1}_d
		\end{tikzcd}
		\end{equation}
		Applying $A\otimes -$ to \eqref{diag:commutativity_Ejs} and composing it with the naturality square for the left $A$-action with respect to $\El^{1,n-1}_d$, we obtain the following commutative diagram of natural transformations of functors $\AMod\to \Mod$.
		\begin{equation}
		\label{diag:commutativity_phats}
		\begin{tikzcd}[column sep=44pt,row sep=30pt]
			A\otimes\id_{\AMod}\ar[r,"A\otimes\Ej^n_d"']\ar[rr,bend left=20pt,two heads,"\Ephat^n_d"]\ar[d,"\id_A\otimes\Ej^{n-1}_d"']&A\otimes \EJ^n\ar[d,"A\otimes \El^{1,n-1}_d"]\ar[r,two heads,"\cdot"']&\EJ^n\ar[d,hook,"\El^{1,n-1}_d"]\\
			A\otimes \EJ^{n-1}_d\ar[r,"A\otimes j^1_{d,\EJ^{n-1}_d}"]\ar[rr,bend right=20pt,two heads,"\phat^1_{d,\EJ^{n-1}_d}"']&A\otimes J^1_d \EJ^{n-1}_d\ar[r,two heads,"\cdot"]&J^1_d \EJ^{n-1}_d
		\end{tikzcd}
		\end{equation}
		where $\phat^1_d=\Ephat^1_d$, cf.\ \cite[\symbolsdefphatn]{Symbol}.
		If we apply the functor $\Omega^m_d$ to \eqref{diag:commutativity_phats}, we obtain the top central square of \eqref{diag:elemental_Spencer_is_complex}.
		If we now consider \eqref{diag:commutativity_Ejs} for $n-1$ instead of $n$, and apply the functor $\Omega^m_d\otimes -$, we also obtain the leftmost square of \eqref{diag:elemental_Spencer_is_complex}.
		The curved bottom triangle commutes by definition of $j^{(2)}_d$; in fact, we have $j^{(2)}_d=j^1_{J^1_d}\circ j^1_d$, and we obtain the desired triangle if we consider the component $\EJ^{n-2}_d$, we apply the functor $\Omega^m_d\otimes -$, and finally we compose both members with the map $\otimes_A\colon \Omega^m_d\otimes J^1_d J^1_d\EJ^{n-2}_d\to \Omega^m_d J^1_d J^1_d\EJ^{n-2}_d$.
		Finally, the bottom central square in \eqref{diag:elemental_Spencer_is_complex} is the naturality square of $\phat^n_d$ with respect to $\El^{1,n-1}_d$ to which we apply the functor $\Omega^m_d$, and as such, it commutes.
		It follows that \eqref{diag:elemental_Spencer_is_complex} commutes.
		Now consider the bottom composition of \eqref{diag:elemental_Spencer_is_complex}.
		It vanishes because $
		\Spenc^{1,m+1}_d\circ\Spenc^{1,m}_{d,J^1_d}\circ (\Omega^m_d\otimes_A j^{(2)}_d)=0$.
		We verify this by computing the component at $E$ for all $\omega\otimes e\in\Omega^m_d \otimes E$.
		\begin{equation}
		\begin{split}
			\Spenc^{1,m+1}_{d,E}\circ\Spenc^{1,m}_{d,J^1_d E}\circ \left(\Omega^m_d\otimes_A j^{(2)}_{d,E}\right)(\omega\otimes e)
			&=\Spenc^{1,m+1}_{d,E}\circ\Spenc^{1,m}_{d,J^1_d E}(\omega\otimes_A [1\otimes 1]\otimes_A [1\otimes e])\\
		&=\Spenc^{1,m+1}_{d,E}(d\omega\otimes_A [1\otimes e])\\
		&=d^2\omega\otimes_A e\\
		&=0
		\end{split}
		\end{equation}
		It follows from \eqref{diag:elemental_Spencer_is_complex} that $\ESpenc^{n-1,m+1}_d\circ \ESpenc^{n,m}_d\circ \Omega^m_d(\Ephat^n_d)=0$, and since $\Omega^m_d(\Ephat^n_d)$ is an epi, we obtain $\ESpenc^{n-1,m+1}_d\circ \ESpenc^{n,m}_d=0$ as claimed.\qedhere
	\end{enumerate}
\end{proof}
By virtue of this theorem, we can define the elemental analogue of Spencer cohomology as follows.
\begin{defi}
We call the cohomology of the elemental Spencer complex \eqref{diag:elemental_Spencer_complex} the \emph{elemental Spencer cohomology}, and we denote the cohomology group at $\Omega^m_d \EJ^n_d$ by $H^{n,m}_{\ESpenc_d}$.
\end{defi}
As for the holonomic case we can prove the exactness of the elemental Spencer sequence at the extremals.
\begin{prop}\label{prop:elemental_Spencer_cohomology_vanishing_at_extremals}
The elemental Spencer complex is exact in degrees $0$, $1$, and $n+1$, so $H^{n,0}_{\ESpenc_d}=H^{0,n}_{\ESpenc_d}=0$.
\end{prop}
\begin{proof}
	The vanishing in degree $0$ and $1$ is equivalent to proving that $\Ej^n_d$ is the kernel of $\ESpenc^{n,0}_d$.
For this purpose, consider the commutative diagram \eqref{diag:Ej_is_jernel}.
We verify that $\Ej^n_d$ satisfies the kernel universal property for $\ESpenc^{n,0}_d$ in the category of functors $\AMod\to\Mod$ and natural transformations between them.
Consider a functor $\AMod\to \Mod$ and a $\bk$-linear natural transformation $f\colon F\to \EJ^n_d$ such that $\ESpenc^{n,0}_d\circ f=0$.
Then by commutativity of \eqref{diag:Ej_is_jernel}, we obtain
\begin{equation}
	0
	=\Omega^1_d(\iota_{\EJ^{n-1}_d})\circ\ESpenc^{n,0}_d\circ f
	=\Spenc^{n,0}_d\circ \iota_{\EJ^n_d}\circ f
\end{equation}
By the kernel universal property, $\iota_{\EJ^n_d}\circ f$ factors through the kernel of $\Spenc^{n,0}_d$.
That is, there exists a unique $\overline{f}\colon F\to \id_{\AMod}$ such that $\iota_{\EJ^n_d}\circ f=j^n_d\circ\overline{f}$, and thus $f=\Ej^n_d\circ\overline{f}$.
It follows that $\ker(\ESpenc^{n,0}_d)=\Ej^n_d$.

The last grade of the cohomology vanishes because $\ESpenc^{n-1,1}_d=\Spenc^{n-1,1}_d$ is an epi by Proposition \ref{prop:Spencer_cohomology_vanishing_at_extremals}.
\end{proof}
\begin{rmk}
One can prove an analogue of Lemma \ref{lemma:pullback_j1_l} for the elemental (and primitive) case, i.e.\ the following diagram is a pullback square
\begin{equation}
\begin{tikzcd}
\id_{\AMod} \arrow[dr, phantom,"\lrcorner", very near start]\ar[r,hook,"\Ej^n_d"]\ar[d,hook,"\Ej^{n-1}_d"']&\EJ^n_d\ar[d,hook,"\El^{1,n-1}_d"]\\
\EJ^{n-1}_d\ar[r,hook,"\Ej^1_{d,\EJ^{n-1}_d}"']&\EJ^1_d \EJ^{n-1}_d
\end{tikzcd}
\end{equation}
\end{rmk}
\subsection{Spencer bicomplex}
We can now consider the following diagram relating the Spencer $\delta$-complex, cf.\ \cite[\jetsssSpencer]{FMW}, and the Spencer complex.
\begin{equation}\label{diag:Spencer_bicomplex}
\begin{tikzcd}[column sep=60pt,row sep=30pt]
&[-30pt]&0\ar[d]&0\ar[d]\\[-10pt]
0\ar[r]&0\ar[d]\ar[r]&\id_{\AMod}\ar[d,"j^n_d"]\ar[r,equals]&\id_{\AMod}\ar[d,"j^{n-1}_d"]\ar[r]&[-30pt]0\\
0\ar[r]&S^n_d\ar[d,"-\delta^{n,0}_d"']\ar[r,"\iota^n_d"]&J^n_d\ar[d,"\Spenc^{n,0}_d"]\ar[r,"\pi^{n,n-1}_d"]&J^{n-1}_d\ar[d,"\Spenc^{n-1,0}_d"]\ar[r]&0\\
0\ar[r]&\Omega^1_d S^{n-1}_d\ar[d,"-\delta^{n-1,1}_d"']\ar[r,"\Omega^1_d(\iota^{n-1}_d)"]&\Omega^1_dJ^{n-1}_d\ar[d,"\Spenc^{n-1,1}_d"]\ar[r,"\Omega^1_d(\pi^{n-1,n-2}_d)"]&\Omega^1_d J^{n-2}_d\ar[d,"\Spenc^{n-2,1}_d"]\ar[r]&0\\
0\ar[r]&\Omega^2_dS^{n-2}_d\ar[d,"-\delta^{n-2,2}_d"']\ar[r,"\Omega^2_d(\iota^{n-2}_d)"]&\Omega^2_dJ^{n-2}_d\ar[d,"\Spenc^{n-2,2}_d"]\ar[r,"\Omega^2_d(\pi^{n-2,n-3}_d)"]&\Omega^2_dJ^{n-3}_d\ar[d,"\Spenc^{n-3,2}_d"]\ar[r]&0\\
&\vdots&\vdots&\vdots
\end{tikzcd}
\end{equation}
\begin{prop}
\label{prop:Spencer bicomplex_is_bicomplex}
The diagram \eqref{diag:Spencer_bicomplex} is a bicomplex.
\end{prop}
\begin{proof}
All three columns of \eqref{diag:Spencer_bicomplex} are complexes: the first is (up to sign) the Spencer $\delta$-complex, cf.\ \cite[\jetsssSpencer]{FMW}, and the second and third are Spencer complexes, cf.\ Theorem \ref{theo:Spencer_is_complex}.

The rows are complexes because they are obtained by applying the functor $\Omega^k_d$ to the jet sequence of order $n-k$, cf.\ \cite[\jetsprophjc]{FMW}.

We now need to prove that all the squares in \eqref{diag:Spencer_bicomplex} commute.
The top left square commutes because the top left element is the zero object.
The top right square commutes by the compatibility of jet prolongations and jet projections, cf.\ \cite[\jetsrmkholprolpi]{FMW}.

The other squares are of two forms corresponding to the left column and to the right column, respectively.
The left one is of the form
\begin{equation}
\begin{tikzcd}[column sep=60pt,row sep=30pt]
\Omega^k_d S^h_d\ar[d,"-\delta^{h,k}_d"']\ar[r,"\Omega^k_d(\iota^h_d)"]&\Omega^k_d J^h_d\ar[d,"\Spenc^{h,k}_d"]\\
\Omega^{k+1}_d S^{h-1}_d\ar[r,"\Omega^{k+1}_d(\iota^{h-1}_d)"]&\Omega^{k+1}_d J^{h-1}_d
\end{tikzcd}
\end{equation}
This diagram commutes because we can see it as the following composition of commutative diagrams
\begin{equation}
\begin{tikzcd}[column sep=70pt,row sep=30pt]
\Omega^k_d S^h_d\ar[d,"\Omega^k_d(\iota^h_\wedge)"]\ar[dd,bend right=50pt,"-\delta^{h,k}_d"']\ar[rr,"\Omega^k_d(\iota^h_d)"]&&\Omega^k_d J^h_d\ar[d,"\Omega^k_d(l^{1,h-1}_d)"']\ar[dd,bend left=50pt,"\Spenc^{h,k}_d"]\\
\Omega^k_d \Omega^1_d S^{h-1}_d\ar[d,two heads,"-(-1)^k\wedge^{k,1}_{S^{h-1}_d}"]\ar[r,"\Omega^k_d\Omega^1_d(\iota^{h-1}_d)"]&\Omega^k_d\Omega^1_dJ^{h-1}_d\ar[r,"\Omega^k_d(\iota^1_{d,J^{h-1}_d})"]\ar[dr,two heads,"-(-1)^k\wedge^{k,1}_{J^{h-1}_d}"']&\Omega^k_d J^1_d J^{h-1}_d\ar[d,two heads,"\Spenc^{1,k}_{d,J^{h-1}_d}"']\\
\Omega^{k+1}_d S^{h-1}_d\ar[rr,"\Omega^{k+1}_d(\iota^{h-1}_d)"']&&\Omega^{k+1}_d J^{h-1}_d
\end{tikzcd}
\end{equation}
Here the top pentagon commutes by definition of $\iota^n_d$, cf.\ \cite[\jetsdiagdefiiotand]{FMW}, the left triangle commutes by definition of $\delta^{h,k}_d$, cf.\ \cite[\jetseqdefdeltahk]{FMW}, the commutativity of the bottom left square follows from the naturality of $\wedge^{k,1}_d$ with respect to $\iota^{h-1}_d$, and the rightmost triangle commutes by Lemma \ref{lemma:SpencerDO_wrt_low_indices}.
It remains to show that the bottom right triangle commutes, or more generally that the following triangle commutes
\begin{equation}
\begin{tikzcd}[column sep=60pt,row sep=30pt]
\Omega^k_d\Omega^1_d\ar[r,"\Omega^k_d(\iota^1_d)"]\ar[dr,two heads,"-(-1)^k\wedge^{k,1}"']&\Omega^k_d J^1_d \ar[d,two heads,"\Spenc^{1,k}_d"]\\
&\Omega^{k+1}_d
\end{tikzcd}
\end{equation}
This can be seen by writing $\Spenc^{1,k}_d$ in the form \eqref{eq:classicSpencerexplicit}.
For all $E$ in $\AMod$ and $\omega\otimes_A \alpha\in \Omega^k_d\Omega^1_d(E)$, we thus have
\begin{equation}
\Spenc^{1,k}_{d,E}\circ \Omega^k_d(\iota^1_{d,E})(\omega\otimes_A \alpha)
=\Spenc^{1,k}_{d,E}(\omega\otimes_A \iota^1_{d,E}(\alpha))
=d\omega \otimes_A \pi^{1,0}_{d,E}\circ \iota^1_{d,E} (\alpha) - (-1)^k \omega \wedge \rho_{d,E}\circ\iota^1_{d,E}(\alpha)
=0-(-1)^k \omega \wedge\alpha.
\end{equation}

The other type of square appearing in \eqref{diag:Spencer_bicomplex} is of the form
\begin{equation}
\begin{tikzcd}[column sep=70pt,row sep=30pt]
\Omega^k_d J^h_d\ar[d,"\Spenc^{h,k}_d"']\ar[r,"\Omega^k_d(\pi^{h,h-1}_d)"]&\Omega^k_d J^{h-1}_d\ar[d,"\Spenc^{h-1,k}_d"]\\
\Omega^{k+1}_d J^{h-1}_d\ar[r,"\Omega^{k+1}_d(\pi^{h-1,h-2}_d)"]&\Omega^{k+1}_d J^{h-2}_d
\end{tikzcd}
\end{equation}
Similarly, this diagram can be decomposed in commutative diagrams as follows
\begin{equation}\label{diag:cell_Spencer_bicomplex}
\begin{tikzcd}[column sep=70pt,row sep=30pt]
\Omega^k_d J^h_d\ar[dd,bend right=50pt,"\Spenc^{h,k}_d"']\ar[d,"\Omega^k_d(l^{1,h-1}_d)"]\ar[r,"\Omega^k_d(\pi^{h,h-1}_d)"]&\Omega^k_d J^{h-1}_d\ar[dd,bend left=50pt,"\Spenc^{h-1,k}_d"]\ar[d,"\Omega^k_d(l^{1,h-2}_d)"']\\
\Omega^k_d J^1_d J^{h-1}_d\ar[d,two heads,"\Spenc^{1,k}_{d,J^{h-1}_d}"]\ar[r,"\Omega^k_dJ^1_d(\pi^{h-1,h-2}_d)"]&\Omega^k_d J^1_d J^{h-2}_d\ar[d,two heads,"\Spenc^{1,k}_{d,J^{h-2}_d}"']\\
\Omega^{k+1}_d J^{h-1}_d\ar[r,"\Omega^{k+1}_d(\pi^{h-1,h-2}_d)"]&\Omega^{k+1}_d J^{h-2}_d
\end{tikzcd}
\end{equation}
The left and right triangles commute by Lemma \ref{lemma:SpencerDO_wrt_low_indices}, and the bottom square commutes by the naturality of $\Spenc^{1,k}_d$ with respect to $\pi^{h-1,h-2}_d$.
In order to show the commutativity of the top square, consider the following diagram
\begin{equation}
\label{diag:trianglesquarebicomplex}
\begin{tikzcd}[column sep=70pt,row sep=30pt]
J^h_d\ar[d,hook,"l^{1,h-1}_d"']\ar[rr,"\pi^{h,h-1}_d"]&&J^{h-1}_d\ar[d,hook,"l^{1,h-2}_d"]\\
J^1_d J^{h-1}_d\ar[r,"J^1_d(l^{1,h-2}_d)"']\ar[rru,two heads,"\pi^{1,0}_{d,J^{h-1}_d}"]&J^1_d J^1_d J^{h-2}_d\ar[r,"\pi^{1,0}_ {d,J^1_dJ^{h-2}_d}"']&J^1_d J^{h-2}_d
\end{tikzcd}
\end{equation}
The top triangle commutes by definition of the jet projection, cf.\ \cite[\jetsdefpiiotan]{FMW}, whereas the bottom square commutes by the naturality of $\pi^{1,0}_d$ with respect to $l^{1,h-2}_d$.

Since $\widetilde{\DH}^I_{d,J^{h-2}_d}\circ J^1_d(l^{1,h-2}_d)\circ l^{1,h-1}_d=0$, we have that
\begin{equation}
\pi^{1,0}_{d,J^1_d J^{h-2}_d}\circ J^1_d(l^{1,h-2}_d)\circ l^{1,h-1}_d=J^1_d(\pi^{1,0}_{d,J^{h-2}_d})\circ J^1_d(l^{1,h-2}_d)\circ l^{1,h-1}_d.
\end{equation}
Thus, from the commutativity of \eqref{diag:trianglesquarebicomplex}, we obtain the commutativity of the pentagon in the diagram below.
\begin{equation}
\begin{tikzcd}[column sep=70pt,row sep=30pt]
J^h_d\ar[d,hook,"l^{1,h-1}_d"']\ar[rr,"\pi^{h,h-1}_d"]&&J^{h-1}_d\ar[d,hook,"l^{1,h-2}_d"]\\
J^1_d J^{h-1}_d\ar[r,"J^1_d(l^{1,h-2}_d)"']\ar[rr,bend right=25pt,"J^1_d(\pi^{h-1,h-2}_d)"]&J^1_d J^1_d J^{h-2}_d\ar[r,two heads,"J^1_d(\pi^{1,0}_ {d,J^{h-2}_d})"']&J^1_d J^{h-2}_d
\end{tikzcd}
\end{equation}
Here, the bottom triangle is obtained by applying $J^1_d$ to the definition of the jet projection.
By applying $\Omega^k_d$ to this diagram we obtain the top square in \eqref{diag:cell_Spencer_bicomplex}, which is thus also commutative.
\end{proof}
\begin{defi}
We call the diagram \eqref{diag:Spencer_bicomplex} the \emph{Spencer bicomplex}.
\end{defi}
When its rows are exact, the Spencer bicomplex allows us to relate the Spencer cohomology to the Spencer $\delta$-cohomology.
\begin{theo}\label{theo:Spencer_les}
Let $\Omega^{\bullet}_d$ be an exterior algebra over the $\bk$-algebra $A$.
\begin{enumerate}
\item\label{theo:Spencer_les:1} Suppose the following is a short exact sequence
\begin{equation}
\begin{tikzcd}\label{diag:Spencer_bicomplex_row}
0\ar[r]&\Omega^k_d S^{h}_d\ar[r,"\Omega^k_d(\iota^{h}_d)"]&[60pt]\Omega^k_dJ^{h}_d\ar[r,"\Omega^k_d(\pi^{h,h-1}_d)"]&[60pt]\Omega^k_d J^{h-1}_d\ar[r]&0
\end{tikzcd}
\end{equation}
for all $h+k=n$, then we have the following long exact sequence
\begin{equation}
\label{diag:SChLong_exact_sequence}
\begin{tikzcd}[column sep=50pt,row sep=10pt]
0\ar[r]&H^{n-1,1}_{\Spenc_d}\ar[r,hook,"\Omega^1_d(\pi^{n-1,n-2}_d)"]\ar[draw=none]{d}[name=X1, anchor=center]{}&H^{n-2,1}_{\Spenc_d}
\ar[rounded corners=6pt,
	to path={ -- ([xshift=10pt]\tikztostart.east)
		|- ([yshift=1pt]X1.center) \tikztonodes
		-| ([xshift=-10pt]\tikztotarget.west)
		-- (\tikztotarget)},near start,"\partial^1"]{dll}\\
H^{n-2,2}_{\delta_d}\ar[r,"\Omega^2_d(\iota^{n-2}_d)"]&H^{n-2,2}_{\Spenc_d}\ar[r,"\Omega^2_d(\pi^{n-2,n-3}_d)"]\ar[draw=none]{d}[name=X2, anchor=center]{}&H^{n-3,2}_{\Spenc_d}
\ar[rounded corners=6pt,
	to path={ -- ([xshift=10pt]\tikztostart.east)
		|- ([yshift=1pt]X2.center) \tikztonodes
		-| ([xshift=-10pt]\tikztotarget.west)
		-- (\tikztotarget)},near start,"\partial^2"]{dll}\\
H^{n-3,3}_{\delta_d}\ar[r,"\Omega^3_d(\iota^{n-3}_d)"]&\vphantom{H^{n-2,2}_{\Spenc_d}}\cdots\ar[draw=none]{d}[name=X3, anchor=center]{}\ar[r,"\Omega^{n-3}_d(\pi^{3,2}_d)"]& H^{2,n-3}_{\Spenc_d}\ar[rounded corners=6pt,
	to path={ -- ([xshift=10pt]\tikztostart.east)
		|- ([yshift=1pt]X3.center) \tikztonodes
		-| ([xshift=-10pt]\tikztotarget.west)
		-- (\tikztotarget)},near start,"\partial^{n-3}"]{dll}\\
H^{2,n-2}_{\delta_d}\ar[r,"\Omega^{n-2}_d(\iota^2_d)"]&H^{2,n-2}_{\Spenc_d}\ar[r,"\Omega^{n-2}_d(\pi^{2,1}_d)"]\ar[draw=none]{d}[name=X4, anchor=center]{}&H^{1,n-2}_{\Spenc_d}
\ar[rounded corners=6pt,
	to path={ -- ([xshift=10pt]\tikztostart.east)
		|- ([yshift=1pt]X4.center) \tikztonodes
		-| ([xshift=-10pt]\tikztotarget.west)
		-- (\tikztotarget)},near start,"\partial^{n-2}"]{dll}\\
H^{1,n-1}_{\delta_d}\ar[r,two heads,"\Omega^{n-1}_d(\iota^1_d)"]&H^{1,n-1}_{\Spenc_d}\ar[r]&0
\end{tikzcd}
\end{equation}
If, instead, we only know that \eqref{diag:Spencer_bicomplex_row} is exact for $m-1\le k\le M+1$, for given $m\le M$, then we can still deduce that the sequence \eqref{diag:SChLong_exact_sequence} is exact in the rows $k$ for $m\le k \le M$, where the rows are numbered by the second index in the cohomology.
\item\label{theo:Spencer_les:2} If $\Omega^1_d$ is flat in $\ModA$, then $H^{1,1}_{\Spenc_d}=0$.
Therefore, $H^{h,k}_{\Spenc_d}=0$ for $h+k\le 2$.
\item\label{theo:Spencer_les:3} Suppose that $\Omega^1_d$ and $\Omega^2_d$ are flat in $\ModA$.
If the $h$-jet exact sequence is exact for all $h\le n$, then $H^{h,1}_{\Spenc_d}=0$ for all $h< n$.
\item\label{theo:Spencer_les:4} Let $n\ge 3$, and let $2\le M\le n$.
Let $\Omega^k_d$ be flat in $\ModA$ for $k\le M+1$.
If $H^{h,k}_{\delta_d}=0$ for all $h<n$ and $k\le M$ then $H^{h,k}_{\Spenc_d}=0$ for the same $h$ and $k$.
\end{enumerate}
\end{theo}
\begin{proof}\
\begin{enumerate}
\item By the stated hypotheses, the Spencer bicomplex \eqref{diag:Spencer_bicomplex} is a short exact sequence of complexes, and thus, via homological algebra, we obtain the long exact sequence \eqref{diag:SChLong_exact_sequence}.
The cohomology modules that are left and right of the ones appearing in \eqref{diag:SChLong_exact_sequence} vanish, cf.\ Proposition \ref{prop:Spencer_cohomology_vanishing_at_extremals}, and \cite[\jetsproplowspencercohom]{FMW}.

In order to prove the last statement of \eqref{theo:Spencer_les:1}, it is sufficient to truncate the vertical complexes in \eqref{diag:Spencer_bicomplex} by substituting every object in rows $k$ with $0$, for $k< m-1$ and $k>M$.
In this way we obtain a new short exact sequence of complexes, whose corresponding long exact sequence coincides with \eqref{diag:SChLong_exact_sequence} in the rows $m\le k\le M$.
\item Consider \eqref{diag:Spencer_bicomplex} for $n=2$.
Given the stated hypothesis, we have that row $0$ is the $2$-jet sequence, which is exact, cf.\ \cite[\jetspropfunctorialtwojetseq]{FMW}.
Row $1$ is obtained by applying the exact functor $\Omega^1_d$ to the $1$-jet short exact sequence, cf.\ \cite[\jetsproponejses]{FMW}, and row $2$ is evidently exact.
By \eqref{theo:Spencer_les:1}, we obtain the long exact sequence
\begin{equation}
\begin{tikzcd}
0\ar[r]&H^{1,1}_{\Spenc_d}\ar[r]&0,
\end{tikzcd}
\end{equation}
which forces $H^{1,1}_{\Spenc_d}=0$.
The remaining statements follow from Proposition \ref{prop:Spencer_cohomology_vanishing_at_extremals}.
\item By the stated hypotheses, the first rows in \eqref{diag:Spencer_bicomplex} are exact.
More precisely they are exact until row $2$, so by Theorem \ref{theo:Spencer_is_complex}.\eqref{theo:Spencer_les:1}, we know that row $1$ in \eqref{diag:SChLong_exact_sequence} is exact.
Therefore, we obtain $H^{n-1,1}_{\Spenc_d}\subseteq H^{n-2,1}_{\Spenc_d}$.
By the same hypotheses, we have that the same result holds if instead of $n$, we consider any $h\le n$, giving us the following sequence of inclusions.
\begin{equation}
H^{n-1,1}_{\Spenc_d}\subseteq H^{n-2,1}_{\Spenc_d}\subseteq H^{n-3,1}_{\Spenc_d}\subseteq \cdots\subseteq H^{1,1}_{\Spenc_d}=0,
\end{equation}
where the last equality follows from \eqref{theo:Spencer_les:2}.
Hence, all elements in this sequence of inclusions must be $0$.
\item By the stated hypotheses, we have that the top rows of \eqref{diag:Spencer_bicomplex} are exact until row $M+1$, and the same holds if, instead of $n$, we have $h\le n$.
This is due to the fact that the $h$-jet sequences are exact for all $h\le n$, cf.\ \cite[\jetscorspencerdeltajes]{FMW} and that $\Omega^k_d$ is an exact functor.
By \eqref{theo:Spencer_les:1}, we obtain the long exact sequence \eqref{diag:SChLong_exact_sequence} restricted to the top $M$ rows, and for all $h\le n$ instead of $n$.

Furthermore, the vanishing of the Spencer $\delta$-cohomology in the specified degrees induces an isomorphism $H^{h-1,k}_{\Spenc_d}\simeq H^{h-2,k}_{\Spenc_d}$ for all $h\le n$ and $k< M$, and an inclusion $H^{h-M,M}_{\Spenc_d}\subseteq H^{h-M-1,M}_{\Spenc_d}$ for all $h\le n$.
We thus have the following chain of isomorphisms
\begin{equation}\label{eq:inclusion_Spencer<M}
H^{n-1,k}_{\Spenc_d}\simeq H^{n-2,k}_{\Spenc_d}\simeq\cdots\simeq H^{1,k}_{\Spenc_d}\simeq H^{0,k}_{\Spenc_d},
\end{equation}
and the following chain of inclusions
\begin{equation}\label{eq:inclusion_Spencer=M}
H^{n-1,M}_{\Spenc_d}\subseteq H^{n-2,M}_{\Spenc_d}\subseteq\cdots\subseteq H^{1,M}_{\Spenc_d}\subseteq H^{0,M}_{\Spenc_d}.
\end{equation}
In both \eqref{eq:inclusion_Spencer<M} and \eqref{eq:inclusion_Spencer=M}, the rightmost term vanishes by Proposition \ref{prop:Spencer_cohomology_vanishing_at_extremals}, yielding the vanishing of all the terms involved in them, namely $H^{h,k}_{\Spenc_d}$ for all $h<n$ and $k\le M$.\qedhere
\end{enumerate}
\end{proof}
As a consequence, the exactness of the Spencer $\delta$-complex implies the exactness of the Spencer complex.
\begin{cor}\label{cor:Spencer_delta_exact_implies_Spencer_exact}
If $\Omega^{\bullet}_d$ is flat in $\ModA$ with vanishing Spencer $\delta$-cohomology $H^{\bullet,\bullet}_{\delta_d}$, then the Spencer cohomology vanishes, i.e.\ $H^{\bullet,\bullet}_{\Spenc_d}=0$.
\end{cor}
\begin{rmk}
Classically, given a smooth manifold $M$, the algebra $A=\smooth{M}$ endowed with $\Omega^\bullet_{dR}(M)$ satisfies the hypotheses of Corollary \ref{cor:Spencer_delta_exact_implies_Spencer_exact}.
This recovers the classical result on the exactness of the Spencer complex, cf.\ \cite[Proposition~1.3.1, p.~187]{Spencer}.
\end{rmk}
\section{Higher order connections}
\label{s:Higher order connections}
Both in the classical setting (cf.\ \cite{Ehresmannconnectionsdordresup} or \cite[§17.1]{NaturalOperations}) and in the noncommutative setting \cite[\jetspropconnexionsplits]{FMW}, connections are equivalent to splittings of the $1$-jet short exact sequence, as vector bundles or as left modules, respectively.
Correspondingly, one considers the geometric object which is responsible for splittings of the higher order jet exact sequences.
Classically, these objects are termed higher order connections.
We show that given relatively mild assumptions (satisfied trivially in the classical case) the theory of higher order connections from classical differential geometry carries over to the noncommutative setting.
Therefore, for the most part of this section we will assume that the higher order jet sequences are exact, cf.\ \cite[\jetstheohigherwolves]{FMW}.
\subsection{Splittings of jet exact sequences}
\label{ss:Splittings_of_jet_exact_sequences}
Inspired by the equivalence between connections and (right) splittings of the $1$-jet exact sequence, we extend the notion of connection to higher orders with the following definition.
\begin{defi}\label{def:higher_connection}
	Let $E$ be in $\AMod$.
	A (left) \emph{n-connection} on $E$ is a section $C^n\colon J^{n-1}_dE\hookrightarrow J^n_dE$ in $\AMod$ of the jet projection $\pi^{n,n-1}_{d,E}\colon J^n_d E\to J^{n-1}_d E$.
\end{defi}
\begin{rmk}
The existence of an $n$-connection makes $\pi^{n,n-1}_{d,E}$ a retraction, i.e.\ a split epi.
As such, it is an epi and every functor will map it into a retraction.
At the same time, $C^n$ is a section, i.e.\ a split mono, and as such it is mapped into a section by any functor.
\end{rmk}
In particular, if the $n$-jet sequence at $E$ is exact, connections are in bijective correspondence with right splittings of said sequence
	\begin{equation}
		\begin{tikzcd}
			0\ar[r]& S^n_dE\ar[r,hook,"\iota^n_{d,E}"]&[25pt] J^n_dE\ar[r,two heads,"\pi^{n,n-1}_{d,E}"]&[25pt] J^{n-1}_dE\ar[r]\ar[l,bend left=20pt,hook,"C^n"]&0.
		\end{tikzcd}
	\end{equation}
	This definition generalizes the one presented in \cite{eastwood2009higher}.
\begin{rmk}
A (left) $1$-connection on $E$ is a (left) connection on $E$, cf.\ \cite[\jetspropconnexionsplits]{FMW}.
\end{rmk}

	In \cite{eastwood2009higher}, another characterization of higher order connections is also given.
	There, they are also presented as differential operators from a given bundle to the corresponding bundle of symmetric $n$-forms valued in it, having as symbol the identity.
	If we assume that symmetric $n$-forms on $E$ are generated by prolongations, i.e.\ $\im(\iota^n_{d,E})\subseteq Aj^n_{d,E}(E)$, cf.\ \cite[\symbolsdefsymgenprolrelations]{Symbol}, we can generalize this definition to the noncommutative case.
	Under these conditions, the notion of restriction symbol of a differential operator is in fact well defined, cf.\ 
\cite[\symbolspropsymbolsrepresentationwelldefined]{Symbol} and \cite[\symbolsdefrestrictionsymbol]{Symbol}, and it can be used to make sense of the classical definition of higher order connection.
	The following proposition presents this characterization of higher order connections and shows that it is equivalent to Definition \ref{def:higher_connection}, under certain regularity conditions including representability of differential operators, there expressed in terms of the elemental jet functor $\EJ^n_d E\colonequals Aj^n_{d,E} E$, cf.\ \cite[\symbolsdefelementalnjetfunctor]{Symbol}.
\begin{prop}
\label{prop:characterization_higher_connections}
Let $E$ be in $\AMod$ such that the $n$-jet sequence at $E$ is exact, then there is a bijective correspondence between $n$-connections $C^n\colon J^{n-1}_dE\hookrightarrow J^n_dE$ and left splittings $J^n_d E\twoheadrightarrow S^n_d E$ of the $n$-jet exact sequence.

Furthermore, if we also assume $J^n_dE=\EJ^n_d E$, then there is a bijective correspondence between $n$-connections $C^n\colon J^{n-1}_dE\hookrightarrow J^n_dE$, and linear differential operators
\begin{equation}
\nabla^n\colon E \longrightarrow S^n_dE
\end{equation}
of order at most $n$ with restriction symbol $r^n_{d,E,F}(\symb^n_d(\nabla^n))\colonequals \widetilde{\nabla}^n\circ \iota^n_{d,E}=\id_{S^n_d E}$.
The correspondence maps a differential operator into the right splitting associated to the unique left splitting given by its unique lift.
\end{prop}
\begin{proof}
The first part follows from the fact that given a short exact sequence, there is a bijective correspondence between left and right splittings.

For the second part, since we have $J^n_dE=\EJ^n_d E$ and the $n$-jet sequence is exact, we also have $J^{n-1}_d E=\EJ^{n-1}_d E$.
This also yields $\im(\iota^n_{d,E})\subseteq Aj^n_d(E)$, cf.\ \cite[\symbolscorsymmformsNN]{Symbol}.
This in turn implies that the restriction symbol is well-defined and coincides with the notion of symbol by the exactness of the $n$-jet sequence at $E$, cf.\ \cite[\symbolspropsymbolsrepresentationwelldefined]{Symbol}.

Suppose we are given a right splitting $C^n$.
As a right splitting, $C^n$ induces a left splitting in $\AMod$, which we denote by $\lambda^n\colon J^n_d E\to S^n_d E$.
Define $\nabla^n\colonequals \lambda^n \circ j^n_{d,E}\colon E \longrightarrow S^n_dE$.
By construction, $\nabla^n$ is a differential operator of order at most $n$ with lift given by $\widetilde{\nabla}^n=\lambda^n$, whose restriction symbol is
\begin{equation}
\widetilde{\nabla}^n\circ \iota^n_{d,E}
=\lambda^n\circ \iota^n_{d,E}=\id_{S^n_dE},
\end{equation}
since $\lambda^n$ is a left splitting.

Conversely, suppose we are given a differential operator $\nabla^n\colon E \to S^n_dE$ with restriction symbol $\id_{S^n_dE}$.
The lift to $J^n_d E$ of $\nabla^n$, which we denote by $\widetilde{\nabla}^n$, is unique since $J^n_d E=\EJ^n_d E$, cf.\ \cite[\symbolspropelementaljetproperties]{Symbol}.
Furthermore, due to its symbol, it provides a left splitting in $\AMod$ of the $n$-jet short exact sequence.
We define $C^n$ to be the induced unique right splitting in $\AMod$.

These constructions are inverse to one another by the uniqueness of the choices at each step.
\end{proof}
This characterization as differential operators generalizes that of linear connections (viewed as differential operators), cf.\ \cite[\symbolspropsymbolsofconnections]{Symbol}.
It also allows us to generalize the property that linear connections differ by tensors, in a way we make precise in the following statement.
\begin{prop}\label{prop:higherconsformaffinespace}
Let $E$ be in $\AMod$ such that $J^n_dE=\EJ^n_d E$ and such that the $n$-jet sequence at $E$ is exact.
Let $\nabla^n_1,\nabla^n_2\in\Diff^n_d(E, S^n_d E)$ with restriction symbol $\id_{S^n_d E}$ (cf.\ Proposition \ref{prop:characterization_higher_connections}).
Then the following properties hold
\begin{enumerate}
	\item $\nabla^n_2-\nabla^n_1 \in \Diff^{n-1}_d(E,S^n_dE)$;
	\item If $\Upsilon^{n-1} \in \Diff^{n-1}_d(E,S^n_dE)$, then $\nabla^n_1 + \Upsilon^{n-1}\in\Diff^n_d(E,S^n_d E)$ with restriction symbol $\id_{S^n_d E}$.
\end{enumerate}
Hence, the set of $n$-connections on $E$ forms an affine space over the additive group of $\Diff^{n-1}_d(E,S^n_dE)$.
\end{prop}
\begin{proof}\
	\begin{enumerate}
		\item 
	We compute $r^n_E(\symb^n_E( \nabla^n_2-\nabla^n_1)) = \id_{S^n_dE}-\id_{S^n_dE}=0$, so by \cite[\symbolspropsymbolsrepresentationwelldefined]{Symbol}, $\nabla^n_2-\nabla^n_1$ is a differential operator of order at most $n-1$.
		\item Since $\Upsilon^{n-1}\in\Diff^{n-1}_d(E,S^n_d E)\subseteq \Diff^n_d(E,S^n_d E)$, we have $ \nabla^n_1 + \Upsilon^{n-1}\in \Diff^n_d(E,S^n_d E)$.
		Further, $r^n_E(\symb^n_E( \nabla^n_1 + \Upsilon^{n-1})) = \id_{S^n_dE}+0$.\qedhere
	\end{enumerate}
\end{proof}

\subsection{Curvature of higher order connections}
\label{ss:Curvature_of_higher_connections}
Given a higher order connection, we can define an associated notion of curvature in the spirit of \cite[§IV.1]{libermann1997}.
\begin{defi}
\label{def:curvature}
The \emph{curvature} of an $n$-connection $C^n\colon J^{n-1}_dE\hookrightarrow J^n_dE$ is defined to be the map 
\begin{equation}
R_{C^n}
\colonequals -\widetilde{\DH}^{II}_{J^{n-1}_d E}\circ J^1_d(l^{1,n-1}_{d,E})\circ J^1_d(C^n) \circ l^{1,n-1}_{d,E}\circ C^n\colon J^{n-1}_dE \longrightarrow \Omega^2_dJ^{n-1}_dE.
\end{equation}
\end{defi}
We will show in Proposition \ref{prop:higher_connection_curvature_is_curvature} that this is equivalent to the more na\"{i}ve generalization of the notion of curvature to our setting.

We now proceed to study some properties of higher order connections and their curvature, starting with the following lemma involving part of the curvature expression.
\begin{lemma}\label{lemma:theclaw}
	Let $C^n\colon J^{n-1}_dE\hookrightarrow J^n_dE$ be an $n$-connection on $E$.
	The morphism
		\begin{equation}
		J^1_d(C^n) \circ l^{1,n-1}_{d,E}\circ C^n \colon J^{n-1}_dE\longhookrightarrow J^1_dJ^n_dE
		\end{equation}
	factors (uniquely) through the inclusion $l^{\{n+1\}}_{d,E}\colon J^{\{n+1\}}_d E\hookrightarrow J^1_dJ^n_dE$.
\end{lemma}
\begin{proof}
	We will show the unique factorization via the kernel universal property, since, by definition, $J^{\{n+1\}}_d E$ is the kernel of
	\begin{equation}
		\widetilde{\DH}^I_{d,J^{n-1}_d E}\circ J^1_d(l^{1,n-1}_{d,E}) \colon J^1_d J^n_d E\longrightarrow \Omega^1_d J^{n-1}_d E.
	\end{equation}
	Recall that by definition of $\widetilde{\DH}^I_d$, cf.\ \cite[\jetseqDHI]{FMW}, we have
	\begin{equation}
		\iota^1_d\circ\widetilde{\DH}^I_d
		= J^1_d (\pi^{1,0}_d)- \pi^{1,0}_{d,J^1_d}.
	\end{equation}
	Thus, since $\iota^1_d$ is a mono, it is sufficient to prove that the following composition vanishes
	\begin{equation}\label{eq:ImJ1Cn_is_sesqui}
	\begin{split}
	&\iota^1_d\circ\widetilde{\DH}^I_{d,J^{n-1}_d E}\circ J^1_d(l^{1,n-1}_{d,E})\circ J^1_d(C^n) \circ l^{1,n-1}_{d,E}\circ C^n\\
	&\qquad=\left(J^1_d (\pi^{1,0}_{d,J^{n-1}_d E})- \pi^{1,0}_{d,J^1_d J^{n-1}_d E}\right)\circ J^1_d(l^{1,n-1}_{d,E})\circ J^1_d(C^n) \circ l^{1,n-1}_{d,E}\circ C^n\\
	&\qquad=J^1_d (\pi^{1,0}_{d,J^{n-1}_d E})\circ J^1_d(l^{1,n-1}_{d,E})\circ J^1_d(C^n) \circ l^{1,n-1}_{d,E}\circ C^n
	- \pi^{1,0}_{d,J^1_d J^{n-1}_d E}\circ J^1_d(l^{1,n-1}_{d,E})\circ J^1_d(C^n) \circ l^{1,n-1}_{d,E}\circ C^n.
	\end{split}
	\end{equation}
	To show this composition vanishes, we verify that the following diagram commutes:
	\begin{equation}\label{diag:theclaw}
		\begin{tikzcd}[column sep=50pt,row sep=20pt]
			J^n_dE\ar[dr,"\pi^{n,n-1}_{d,E}"]\ar[r,hook,"l^{1,n-1}_{d,E}"]&J^1_dJ^{n-1}_d E\ar[d,two heads,"\pi^{1,0}_{d,J^{n-1}_dE}"]\ar[r,hook,"J^1_d(C^n)"]&J^1_dJ^n_dE\ar[d,two heads,"\pi^{1,0}_{d,J^n_dE}"] \ar[r,"J^1_d (l^{1,n-1}_{d, E})"] & J^1_d J^1_d J^{n-1}_d E \ar[d,two heads,"\pi^{1,0}_{d,J^1_dJ^{n-1}_dE}"]\\
			J^{n-1}_d E\ar[r,equals]\ar[u,hook,"C^n"]	&J^{n-1}_dE\ar[r,hook,"C^n"']&J^n_dE \ar[r,hook,"l^{1,n-1}_{d,E}"']&J^1_dJ^{n-1}_dE\ar[r,equals]\ar[d,hook,"J^1_d (C^n)"']&J^1_d J^{n-1}_d E\\
			&&&J^1_dJ^{n}_dE\ar[r,"J^1_d (l^{1,n-1}_{d,E})"']\ar[ur,two heads,"J^1_d(\pi^{n,n-1}_{d,E})"]&J^1_d J^1_d J^{n-1}_d E\ar[u,two heads,"J^1_d (\pi^{1,0}_{d, J^{n-1}_d E})"']
		\end{tikzcd}
	\end{equation}
	This will end the proof, as the commutativity of \eqref{diag:theclaw} implies the equality between the bottom and the top compositions from the leftmost $J^{n-1}_d E$ to the rightmost $J^1_d J^{n-1}_d E$ herein.
	Thus, \eqref{eq:ImJ1Cn_is_sesqui} will vanish, being the difference between the two compositions.
	In turn, the desired factorization through $l^{\{n+1\}}_{d,E}$ will follow by the kernel universal property.

	The leftmost triangle of \eqref{diag:theclaw} commutes by the definition of a connection $C^n$, since $\pi^{n,n-1}_{d,E} \circ C^n = \id_{J^{n-1}_dE}$.
	The adjacent triangle commutes by definition of $\pi^{n,n-1}_{d,E}\colonequals\pi^{1,0}_{J^{n-1}_dE}\circ l^{1,n-1}_{d,E}$, cf.\ \cite[\jetsdefpiiotan]{FMW}.
	Applying $J^1_d$ to these two triangles yields the two triangles on the right.
	Finally, the morphisms $C^n$ and $l^{1,n-1}_{d,E}$ are left $A$-linear, hence we can apply the functor $J^1_d$ to them, and the naturality of $\pi^{1,0}_d$ with respect to $C^n$ and $l^{1,n-1}_{d,E}$ gives the commutativity of the two middle squares.
	Hence, \eqref{diag:theclaw} commutes as claimed.
\end{proof}

\begin{prop}
\label{prop:fullcurvature}
Given a $n$-connection $C^n\colon J^{n-1}_d E\hookrightarrow J^n_d E$:
\begin{enumerate}
\item\label{prop:fullcurvature:1}
The map
\begin{equation}\label{eq:fullcurvature}
-\widetilde{\DH}_{J^{n-1}_d E}\circ J^1_d(l^{1,n-1}_{d,E})\circ J^1_d(C^n) \circ l^{1,n-1}_{d,E}\circ C^n\colon J^{n-1}_dE \longrightarrow (\Omega^1_d\ltimes \Omega^2_d)J^{n-1}_dE.
\end{equation}
has image in $\Omega^2_d J^{n-1}_d E$, where it coincides with the curvature.
In other words, it factors uniquely through the inclusion $\Omega^2_d J^{n-1}_d E\hookrightarrow (\Omega^1_d\ltimes \Omega^2_d)J^{n-1}_dE$ as the curvature $R_{C^n}$.
\item\label{prop:fullcurvature:2}
$R_{C^n}\colon J^{n-1}_dE \longrightarrow \Omega^2_dJ^{n-1}_dE$ is left $A$-linear.
\item\label{prop:fullcurvature:3}
If $n=1$ or if $n>1$ and $\Omega^2_d$ is flat in $\ModA$ and (for $n\ge 3$) the $(n-1)$-jet sequence at $E$ is left exact, then the curvature $R_{C^n}$ has image in $\Omega^2_dS^{n-1}_d E$.
\item\label{prop:fullcurvature:4}
Assume $\Omega^1_d$ and $\Omega^2_d$ are flat in $\ModA$ and, if $n\ge 3$, assume that $\Omega^3_d$ is flat in $\ModA$ (or, for $n=3$, just $\Tor^A_1(\Omega^3_d,E)=0$) and that the $(n-1)$-jet sequence at $E$ is left exact, then the curvature $R_{C^n}$ has image in $\ker(\delta^{n-1,2}_d)$.
\end{enumerate}
\end{prop}
\begin{proof}\
\begin{enumerate}
\item By Lemma \ref{lemma:theclaw}, the image of $J^1_d(C^n) \circ l^{1,n-1}_{d,E}\circ C^n \colon J^{n-1}_d E\hookrightarrow J^1_dJ^n_dE$ is contained in $J^{\{n+1\}}_d E$, i.e.\ we can write it as $l^{\{n+1\}}_{d,E}\circ f$ for a unique map $f\colon J^{n-1}_dE\to J^{\{n+1\}}_d E$.
We thus have the following equality:
\begin{equation}\label{eq:factorization_fullcurvature}
-\widetilde{\DH}_{J^{n-1}_d E}\circ J^1_d(l^{1,n-1}_{d,E})\circ J^1_d(C^n) \circ l^{1,n-1}_{d,E}\circ C^n
=-\widetilde{\DH}_{J^{n-1}_d E}\circ J^1_d(l^{1,n-1}_{d,E})\circ l^{\{n+1\}}_{d,E}\circ f.
\end{equation}
We also know that $\widetilde{\DH}_{J^{n-1}_d E}\circ J^1_d(l^{1,n-1}_{d,E})\circ l^{\{n+1\}}_{d,E}$ factors uniquely through the inclusion of $\Omega^2_d J^{n-1}_d E$ as
\begin{equation}
\widetilde{\DH}^{II}_{J^{n-1}_d E}\circ J^1_d(l^{1,n-1}_{d,E})\circ l^{\{n+1\}}_{d,E},
\end{equation}
cf.\ \cite[\jetslemmaimageD]{FMW}.
Hence, \eqref{eq:factorization_fullcurvature} factors through $\Omega^2_d J^{n-1}_d E\hookrightarrow (\Omega^1_d\ltimes \Omega^2_d)J^{n-1}_dE$ as
\begin{equation}
-\widetilde{\DH}^{II}_{J^{n-1}_d E}\circ J^1_d(l^{1,n-1}_{d,E})\circ l^{\{n+1\}}_{d,E}\circ f
=-\widetilde{\DH}^{II}_{J^{n-1}_d E}\circ J^1_d(l^{1,n-1}_{d,E})\circ J^1_d(C^n) \circ l^{1,n-1}_{d,E}\circ C^n
=R_{C^n}.
\end{equation}
\item The $A$-linearity follows from \eqref{prop:fullcurvature:1} given that \eqref{eq:fullcurvature} is $A$-linear.

\item Since, by Lemma \ref{lemma:theclaw}, the image of $J^1_d(C^n) \circ l^{1,n-1}_{d,E}\circ C^n \colon J^{n-1}_dE\hookrightarrow J^1_dJ^n_dE$ is contained in $J^{\{n+1\}}_d E$, we can apply \cite[\jetslemmaimageD]{FMW} in the appropriate grade.
\item This statement is trivially true for $n=1$.
The remaining cases follow \textit{mutatis mutandis} from \cite[\jetslemmaedhinkerspencer]{FMW} and \cite[\jetsrmkweakOthreeflat]{FMW}.
\qedhere
\end{enumerate}
\end{proof}
The following result generalizes the classical statement appearing in \cite[§IV.1]{libermann1997}.
\begin{prop}
The curvature $R_{C^n}$ vanishes if and only if $J^1_d(C^n) \circ l^{1,n-1}_{d,E} \circ C^n$ has image in $J^{n+1}_dE$.
\end{prop}
\begin{proof}
By definition of $J^{n+1}_d E$, the map $J^1_d(C^n) \circ l^{1,n-1}_{d,E} \circ C^n$ has image in $J^{n+1}_d E$ if and only if \eqref{eq:fullcurvature} vanishes.
By Proposition \ref{prop:fullcurvature}.\eqref{prop:fullcurvature:1}, we know that \eqref{eq:fullcurvature} coincides with the composition
\begin{equation}
\begin{tikzcd}
J^{n-1}_dE \ar[r,"R_{C^n}"]& \Omega^2_dJ^{n-1}_dE\ar[r,hook]&(\Omega^1_d \ltimes \Omega^2_d)J^{n-1}_d E.
\end{tikzcd}
\end{equation}
Since the right map is a mono, this composition vanishes if and only if $R_{C^n}$ vanishes, completing the proof.
\end{proof}

\subsection{Relation with connections on jet modules}
\label{ss:Relation_with_connections_on_jet_modules}
In this section we characterize higher order connections in terms of connections on the corresponding higher jet bundles.
\begin{defi}\label{def:connection_corresponding_to_higher_connection}
For $n\ge 1$, the \emph{(left) connection associated to a (left) $n$-connection} $C^n\colon J^{n-1}_d E\hookrightarrow J^n_d E$ on $E$ in $\AMod$ is defined as
\begin{equation}
\label{eq:connection_corresponding_to_higher_connection}
\nabla^{C^n} \colonequals \Spenc_{d,E}^{n,0} \circ C^n\colon J^{n-1}_dE \longrightarrow \Omega^1_dJ^{n-1}_dE.
\end{equation}
\end{defi}
\begin{prop}
\label{prop:construction_connection_corresponding_to_higher_connection}
The $\bk$-linear map $\nabla^{C^n}$, cf.\ \eqref{eq:connection_corresponding_to_higher_connection}, associated to a $n$-connection $C^n\colon J^{n-1}_d E\hookrightarrow J^n_d E$ is a connection on $J^{n-1}_d E$.

If $n=1$, a $1$-connection $C^1$ is the right splitting of the $1$-jet sequence induced by its associated connection $\nabla^{C^1}$, cf.\ \cite[\jetspropconnexionsplits]{FMW}, and $\nabla^{C^1}=\nabla^1$, cf.\ Proposition \ref{prop:characterization_higher_connections}.
\end{prop}
\begin{proof}
Being the composition of an $A$-linear map $C^n$ and $\Spenc^{n,0}_{d,E}$, which is a differential operator of order at most $1$, cf.\ Proposition \ref{prop:spencopsymbol}, the map $\nabla^{C^n}\colon J^{n-1}_d E\to \Omega^1_d J^{n-1}_d E$ is a differential operator of order at most $1$ on $J^{n-1}_dE$.
In particular, its (unique) $A$-linear lift to $J^1_d J^{n-1}_dE$ is given by
\begin{equation}
\label{eq:lift_connection_associated_to_higher_connection}
\widetilde{\nabla}^{C^n}
=\widetilde{\Spenc}^{n,0}_{d,E}\circ J^1_d(C^n).
\end{equation}

We will now prove that $\nabla^{C^n}$ is a connection by showing that its restriction symbol is $\id_{\Omega^1_d J^{n-1}_d E}$, cf.\ \cite[\symbolspropsymbolsofconnections]{Symbol}.
In order to compute the restriction symbol, we consider the following diagram: 
\begin{equation}
\begin{tikzcd}[column sep=40pt,row sep=20pt]
\Omega^1_d J^{n-1}_d E\ar[drr,equals,bend left=50pt]\ar[d,hook,"\iota^1_{d,J^{n-1}_d E}"']\ar[r,hook,"\Omega^1_d (C^n)"]&\Omega^1_d J^n_d E \ar[d,hook,"\iota^1_{d,J^n_d E}"'] \ar[dr,two heads,"\Omega^1_d(\pi^{n,n-1}_{d,E})",near start]\\
J^1_d J^{n-1}_d E\ar[rr,bend right=10pt,"\widetilde{\nabla}^{C^n}"']\ar[r,hook,"J^1_d (C^n)"]& J^1_d J^n_d E \ar[r,"\widetilde{\Spenc}^{n,0}_{d,E}",near start]&[40pt]\Omega^1_d J^{n-1}_d E
\end{tikzcd}
\end{equation}
The bottom triangle commutes by \eqref{eq:lift_connection_associated_to_higher_connection}.
The top triangle commutes by definition of $n$-connection and functoriality of $\Omega^1_d$.
The right triangle commutes by Proposition \ref{prop:spencopsymbol}.
The square commutes by the naturality of $\iota^1_d$ with respect to $C^n$.
It follows that the symbol of $\nabla^{C^n}$ is the identity, thus proving that $\nabla^{C^n}$ is a connection.

In particular, if $n=1$, the $1$-jet sequence is always exact, which implies that right splittings and left splittings are in bijective correspondence.
The correspondence described in \cite[\jetspropconnexionsplits]{FMW} induces a correspondence between a connection and the left splitting its lift induces.
Since $\widetilde{\nabla}^{C^1}$ and $C^1$ are left and right splittings, respectively, of the $1$-jet sequence, we are left to prove that they correspond to the same splitting.
Namely, the following equality
\begin{equation}
C^1\circ \pi^{1,0}_{d,E}+\iota^1_{d,E}\circ \widetilde{\nabla}^{C^1}
=\id_{J^1_d E}.
\end{equation}
By \eqref{eq:lift_connection_associated_to_higher_connection}, using \eqref{eq:lift_Spencer_n,0} and the definition of $\widetilde{\DH}^I_{d,E}$, cf.\ \cite[\jetseqiotaprimestuff]{FMW}, we obtain the desired equality via the following computation
\begin{equation}
\begin{split}
\iota^1_{d,E} \circ \widetilde{\nabla}^{C^1}
&=\iota^1_{d,E}\circ\widetilde{\Spenc}^{1,0}_{d,E}\circ J^1_d (C^1)\\
&=\iota^1_{d,E}\circ\widetilde{\DH}^I_{d,E}\circ J^1_d (C^1)\\
&=(J^1_d(\pi^{1,0}_{d,E})-\pi^{1,0}_{d,J^1_d E})\circ J^1_d(C^1)\\
&=J^1_d(\pi^{1,0}_{d,E})\circ J^1_d(C^1)-\pi^{1,0}_{d,J^1_d E}\circ J^1_d(C^1)\\
&=\id_{J^1_d E}- C^1\circ\pi^{1,0}_{d,E},
\end{split}
\end{equation}
where the last equality follows from the functoriality of $J^1_d$, the definition of a $1$-connection, and the naturality of $\pi^{1,0}_d$ with respect to $C^1$.
The last statement then follows from the fact that the differential operator $\nabla^1$ of order at most $1$ in Proposition \ref{prop:characterization_higher_connections} is obtained as $\nabla^1\colonequals \widetilde{\nabla}^{C^1}\circ j^1_{d,E}=\nabla^{C^1}$.
\end{proof}
We can actually do more, and define a notion of exterior covariant derivative corresponding to an $n$-connection as follows.
\begin{defi}\label{def:exterior_covariant_derivative_corresponding_to_higher_connection}
For $m\ge 0$ and $n\ge 1$, the \emph{exterior covariant derivative $d_{C^n}$ associated to $C^n$} is defined via the following commutative diagram.
\begin{equation}
\begin{tikzcd}
	{\Omega^m_dJ^{n-1}_dE} & {\Omega^{m+1}_dJ^{n-1}_dE} \\
	{\Omega^m_dJ^{n}_dE}
	\arrow["{\Spenc^{n,m}_{d,E}}"', from=2-1, to=1-2]
	\arrow["{\Omega^{m}_d(C^n)}"',hook, from=1-1, to=2-1]
	\arrow["{d_{C^n}}", from=1-1, to=1-2]
\end{tikzcd}
\end{equation}
\end{defi}
\begin{prop}
\label{prop:classicalexteriorderivativeformula}
Given an $n$-connection $C^n$ with induced connection $\nabla^{C^n} \colonequals \Spenc_{d,E}^{n,0} \circ C^n$, the exterior covariant derivative $d_{C^n}\colonequals \Spenc_{d,E}^{n,m} \circ C^n$ is the exterior covariant derivative associated to $\nabla^{C^n}$.
In particular, it satisfies
\begin{equation}
\label{eq:exterior_derivative_is_as_expected}
d_{C^n}(\omega\otimes_A\xi)
= d\omega\otimes_A \xi + (-1)^{\deg(\omega)}\omega\wedge\nabla^{C^n}\xi
\end{equation}
for all $\omega\otimes_A \xi\in \Omega^m_d J^n_d E$.
\end{prop}
\begin{proof}
By definition, for $m=0$, we have $d_{C^n}=\nabla^{C^n}$, so we are left to prove \eqref{eq:exterior_derivative_is_as_expected}.
The statement follows from the following computation thanks to Proposition \ref{prop:Spencercorrespondence}.
\begin{equation}
\begin{split}
d_{C^n}(\omega\otimes_A \xi)
&=\Spenc^{n,m}_{d,E}\circ \Omega^m_d(C^n)(\omega\otimes_A\xi)\\
&=\Spenc^{n,m}_{d,E}(\omega\otimes_AC^n(\xi))\\
&=d\omega \otimes_A \pi^{n,n-1}_{d,E} (C^n(\xi)) + (-1)^{\deg(\omega)} \omega \wedge \Spenc^{n,0}_{d,E}(C^n(\xi))\\
&=d\omega \otimes_A \xi + (-1)^{\deg(\omega)} \omega \wedge \nabla^{C^n}(\xi),
\end{split}
\end{equation}
where the last equality follows from the definitions of $n$-connection and its associated connection.
\end{proof}
The following Proposition illuminates Definition \ref{def:curvature}, by showing that the curvature of a higher order connection coincides with the curvature of the associated connection.
\begin{prop}
\label{prop:higher_connection_curvature_is_curvature}
Given an $n$-connection $C^n$ and its associated connection $\nabla^{C^n}$, the curvature $R_{C^n}$ of $C^n$ coincides with the curvature $R_{\nabla^{C^n}}\colonequals d_{C^n}\circ\nabla^{C^n}$ of $\nabla^{C^n}$.
More generally, for $\omega\otimes_A \xi\in \Omega^m_d J^n_d E$, we have
\begin{equation}\label{eq:higher_connection_curvature_is_curvature}
d_{C^n}\circ d_{C^n}(\omega\otimes_A \xi)
=\omega\wedge R_{C^n}(\xi).
\end{equation}
\end{prop}
\begin{proof}
We start by proving \eqref{eq:higher_connection_curvature_is_curvature} in the case $m=0$.
Namely, we need to show that $d_{C^n}\circ \nabla^{C^n}(\xi)=R_{C^n}(\xi)$.
Consider the following diagram
\begin{equation}
\label{diag:torpedo}
\begin{tikzcd}[column sep=50pt,row sep=40pt]
J^1_d J^n_d E\ar[rrd,phantom,"\encircle{\barroman{VI}}" description]\ar[rr,"J^1_d(l^{1,n-1}_{d,E})"]\ar[dr,two heads,"\Spenc^{1,0}_{d,J^n_d E}"]&&J^{(2)}_d J^{n-1}_d E\ar[d,two heads,"\Spenc^{1,0}_{d,J^1_d J^{n-1}_d E}"']\ar[dd,two heads,"-\widetilde{\DH}^{II}_{d,J^{n-1}_d E}", controls={ +(0:3) and +(0:3)}]\ar[dd,bend left=50pt,phantom,"\encircle{\barroman{VII}}",near start]\\
J^1_d J^{n-1}_d E\ar[r,bend left=5pt,phantom,"\encircle{\barroman{V}}" description]\ar[u,hook,"J^1_d (C^n)"]\ar[dr,two heads,"\Spenc^{1,0}_{d,J^{n-1}_d E}"]&\Omega^1_d J^n_d E\ar[dr,"\Spenc^{n,1}_{d,E}"]\ar[r,"\Omega^1_d(l^{1,n-1}_{d,E})"]&\Omega^1_d J^1_d J^{n-1}_d E\ar[d,two heads,"\Spenc^{1,1}_{d,J^{n-1}_d E}",near start]\ar[d,bend right=40pt,phantom,"\encircle{\barroman{IV}}" description]\\
J^n_d E\ar[u,bend right=40pt,phantom,"\encircle{\barroman{III}}" description]\ar[u,hook,"l^{1,n-1}_{d,E}"]\ar[r,"\Spenc^{n,0}_{d,E}"]&\Omega^1_d J^{n-1}_d E\ar[r,"d_{C^n}"']\ar[r,bend left=15pt,phantom,"\encircle{\barroman{II}}" description]\ar[u,hook,"\Omega^1_d(C^n)"']&\Omega^2_d J^{n-1}_d E&\phantom{}\\
J^{n-1}_d E\ar[ur,"\nabla^{C^n}"']\ar[u,hook,"C^n"]\ar[rru,bend right=15pt,"R_{\nabla^{C^n}}"']\ar[ur,bend left=20pt,phantom,"\encircle{\barroman{I}}" description]\ar[urr,bend left=3pt,phantom,"\encircle{\barroman{VIII}}" description]
\end{tikzcd}
\end{equation}
The triangles \encircle{\barroman{I}} and \encircle{\barroman{II}} commute by definition of $\nabla^{C^n}$ and $d_{C^n}$, respectively.
The triangles \encircle{\barroman{III}} and \encircle{\barroman{IV}} commute by Lemma \ref{lemma:SpencerDO_wrt_low_indices}.
The squares \encircle{\barroman{V}} and \encircle{\barroman{VI}} commute by the naturality of $\Spenc^{1,0}_d$ with respect to $C^n$ and $l^{1,n-1}_{d,E}$, respectively.
The diagram \encircle{\barroman{VII}} commutes by Lemma \ref{lemma:Spencersquared}.\eqref{lemma:Spencersquared:4}.
Finally, \encircle{\barroman{VIII}} commutes by the definition of $R_{\nabla^{C^n}}$.
The commutativity of the diagram gives
\begin{equation}
R_{\nabla^{C^n}}
=-\widetilde{\DH}^{II}_{J^{n-1}_d}\circ J^1_d(l^{1,n-1}_{d,E})\circ J^1_d(C^n) \circ l^{1,n-1}_{d,E}\circ C^n
=R_{C_n},
\end{equation}
by definition of $R_{C^n}$.

By Proposition \ref{prop:classicalexteriorderivativeformula}, $d_{C^n}$ is the exterior derivative associated to the connection $\nabla^{C^n}$.
This implies the general formula \eqref{eq:higher_connection_curvature_is_curvature}, cf.\ \cite[\jetslemmaecdcurvaturelinear]{FMW}.
\end{proof}
In the following result we study the properties of connections on jet bundles arising from higher order connections and we show that, under suitable assumptions, we can invert this construction.
\begin{prop}
\label{prop:properties_connections_coming_from_higher_connections}
An $n$-connection $C^n\colon J^{n-1}_d E\hookrightarrow J^n_d E$ induces a connection $\nabla^{C^n}\colonequals \Spenc^{n,0}_{d,E}\circ C^n$ on $J^{n-1}_d E$ satisfying the following properties:
\begin{enumerate}
\item\label{prop:properties_connections_coming_from_higher_connections:1} In degree $m$, we have $\Omega^{m+1}_d(\pi^{n-1,n-2}_{d,E}) \circ d_{\nabla^{C^n}} =\Spenc_{d,E}^{n-1,m}$, and 
hence $\Omega^1_d(\pi^{n-1,n-2}_{d,E}) \circ \nabla^{C^n} =\Spenc_{d,E}^{n-1,0}$;
\item\label{prop:properties_connections_coming_from_higher_connections:2} Moreover, if $n=1$ or $n>1$ and $\Omega^2_d$ is flat in $\ModA$ and (if $n\ge 3$) the $(n-1)$-jet sequence is left exact, then the curvature $R_{\nabla^{C^n}}=d_{\nabla^{C^n}} \circ \nabla^{C^n} \colon J^{n-1}_dE\rightarrow \Omega^2_dJ^{n-1}_dE$ takes values in $\Omega^2_dS^{n-1}_dE$.
\end{enumerate}
\end{prop}
\begin{proof}\
\begin{enumerate}
\item This point follows directly from the commutativity of the following diagram.
\begin{equation}
\begin{tikzcd}[column sep=50pt,row sep=20pt]
\Omega^m_d J^{n-1}_d E\ar[dr,equals]\ar[r,hook,"\Omega^m_d (C^n)"]\ar[rr,bend left=20pt,"d_{\nabla^{C^n}}"]&\Omega^m_d J^n_d E\ar[r,"\Spenc^{n,m}_{d,E}"]\ar[d,two heads,"\Omega^m_d(\pi^{n,n-1}_{d,E})"]&\Omega^{m+1}_d J^{n-1}_d E\ar[d,"\Omega^{m+1}_d(\pi^{n-1,n-2}_{d,E})"]\\
&\Omega^m_d J^{n-1}_d E\ar[r,"\Spenc^{n-1,m}_{d,E}"]&\Omega^{m+1}_d J^{n-2}_d E
\end{tikzcd}
\end{equation}
Here, the left triangle commutes by the definition of a higher order connection, while the top one commutes by definition of $d_{C^n}$, which coincides with $d_{\nabla^{C^n}}$ by Proposition \ref{prop:classicalexteriorderivativeformula}.
The square commutes by Proposition \ref{prop:Spencer bicomplex_is_bicomplex}.
The second formula of \eqref{prop:properties_connections_coming_from_higher_connections:1} is just the case $m=0$ of this formula.
\item Given the stated assumptions, we can apply Proposition \ref{prop:fullcurvature}.\eqref{prop:fullcurvature:3}, which tells us that $R_{C^n}$ has image in $\Omega^2_d S^{n-1}_d E$.
If we now apply Proposition \ref{prop:higher_connection_curvature_is_curvature}, we obtain the desired result.\qedhere
\end{enumerate}
\end{proof}
Further, we see that the statements of Proposition \ref{prop:properties_connections_coming_from_higher_connections}.\eqref{prop:properties_connections_coming_from_higher_connections:1} are equivalent.
\begin{lemma}
\label{lemma:condition_1_HCC_correspondence_1_to_m}
Let $\nabla\colon J^{n-1}_d E\to \Omega^1_d J^{n-1}_d E$ be a connection on $J^{n-1}_d E$, then $\Omega^{m+1}_d(\pi^{n-1,n-2}_{d,E}) \circ d_{\nabla} =\Spenc_{d,E}^{n-1,m}$ holds for all $m$ if and only if it holds for $m=0$, i.e.\ $\Omega^1_d(\pi^{n-1,n-2}_{d,E}) \circ \nabla =\Spenc_{d,E}^{n-1,0}$.
\end{lemma}
\begin{proof}
We only need to check that, if the formula is true for $m=0$, then it is true for all $m\ge 0$.
Given $\omega\otimes_A \xi\in \Omega^m_d J^{n-1}_d E$, we have
\begin{equation}
\begin{split}
\Omega^{m+1}_d(\pi^{n-1,n-2}_{d,E})\circ d_\nabla(\omega\otimes_A \xi)
&=\Omega^{m+1}_d(\pi^{n-1,n-2}_{d,E})(d\omega\otimes_A \xi+(-1)^m\omega\wedge \nabla \xi)\\
&=d\omega\otimes_A \pi^{n-1,n-2}_{d,E}(\xi)+(-1)^m\omega\wedge\Omega^1_d(\pi^{n-1,n-2}_{d,E}) (\nabla \xi)\\
&=d\omega\otimes_A \pi^{n-1,n-2}_{d,E}(\xi)+(-1)^m\omega\wedge\Spenc_{d,E}^{n-1,0}(\xi)\\
&=\Spenc^{n-1,m}_{d,E}(\omega\otimes_A \xi),
\end{split}
\end{equation}
where the last equality follows from Proposition \ref{prop:Spencercorrespondence}.
\end{proof}
We mention here also the following technical lemma that will come in handy in the coming theorem.
This lemma essentially shows that differential operators remain differential operators upon codomain restriction.
\begin{lemma}
\label{lemma:images_of_DOS_are_DOS}
Let $E$, $F$, and $G$ be $A$-modules, let $\Delta\colon E\to F$ be in $\Mod$ and let $m\colon F\hookrightarrow G$ be a monomorphism in $\AMod$ such that $m\circ \Delta\in \Diff^n_d(E,G)$, then $\Delta\in\EDiff^n_d(E,F)$, and if $\EJ^n_d E=J^n_d E$, then $\Delta\in\Diff^n_d(E,F)$.
\end{lemma}
\begin{proof}
We use the criteria proved in \cite[\symbolspropcriterionWDO]{Symbol}.
For all $\sum_i a_i\otimes e_i\in N^n_d(E)$, we have
\begin{equation}
m\left(\sum_i a_i\Delta( e_i)\right)
=\sum_i a_i m\circ\Delta( e_i)
=0.
\end{equation}
Since $m$ is a mono, we have $\sum_i a_i\Delta (e_i)=0$, which implies $\Delta\in\EDiff^n_d(E,F)$, cf.\ \cite[\symbolspropcriterionWDOone]{Symbol}.
If moreover $\EJ^n_d E=J^n_d E$, then $\Delta\in\Diff^n_d(E,F)$ by \cite[\symbolspropcriterionWDOthree]{Symbol}.
\end{proof}
We now show that connections with the properties discussed in Proposition \ref{prop:properties_connections_coming_from_higher_connections} are exactly those arising from higher order connections.
\begin{theo}
\label{theo:higher_connections_are_connections}
Let $\Omega^\bullet_d$ be an exterior algebra over a $\bk$-algebra $A$ and let $E$ be in $\AMod$.
Assume that $n=1$, or that $n>1$ and the following conditions hold: $\Omega^1_d$ and $\Omega^2_d$ are flat in $\ModA$, $J^n_dE=\EJ^n_d E$, and the $n$-jet sequence at $E$ is exact.
Furthermore, assume that, when $n\ge 3$, the $(n-1)$-jet sequence is left exact and (when $n\ge 4$) that $\iota^{n-2}_{d,E}$ is a mono.
Then there is a bijective correspondence between $n$-connections on $E$ and connections on $J^{n-1}_d E$ such that
\begin{enumerate}
\item\label{theo:higher_connections_are_connections:1} $\Omega^1_d(\pi^{n-1,n-2}_{d,E}) \circ \nabla =\Spenc_{d,E}^{n-1,0}$;
\item\label{theo:higher_connections_are_connections:2} The curvature $R_\nabla \colon J^{n-1}_dE\rightarrow \Omega^2_dJ^{n-1}_dE$ has values in $\Omega^2_dS^{n-1}_dE$.
\end{enumerate}
The correspondence maps an $n$-connection $C^n$ to its associated connection $\nabla^{C^n}\colonequals \Spenc_{d,E}^{n,0} \circ C^n$, while the inverse construction maps a connection $\nabla\colon J^{n-1}_dE\to \Omega^1_d J^{n-1}_dE$ to the $n$-connection that corresponds, in the sense of Proposition \ref{prop:characterization_higher_connections}, to the differential operator $\nabla^n\colon E\to S^n_d E$ uniquely identified by the equality
\begin{equation}
\Omega^1_d (\iota^{n-1}_{d,E})\circ \iota^n_{\wedge,E}\circ\nabla^n
=\nabla\circ j^{n-1}_{d,E}\colon E\longrightarrow \Omega^1_d J^{n-1}_d E.
\end{equation}

This correspondence extends the one described by \cite[\jetspropconnexionsplits]{FMW}.
\end{theo}
\begin{proof}
By Proposition \ref{prop:properties_connections_coming_from_higher_connections}, we know that $C^n$ induces a connection $\nabla^{C^n}$ satisfying \eqref{theo:higher_connections_are_connections:1} and \eqref{theo:higher_connections_are_connections:2}.

For the inverse construction, consider $\nabla''^n \colonequals \nabla \circ j^{n-1}_{d,E}\colon E \rightarrow \Omega^1_dJ^{n-1}_dE$.
We first show that $\nabla''^n$ has image in $\Omega^1_d S^{n-1}_d E$.
The case $n=1$ is straightforward.
For $n>1$, since $\Omega^1_d$ is flat in $\ModA$ and the $(n-1)$-jet sequence is left exact, we have the following left exact sequence:
\begin{equation}
\label{es:Omega1_jles}
\begin{tikzcd}[column sep=50pt,row sep=20pt]
0\ar[r]&[-20pt] \Omega^1_d S^{n-1}_d E \ar[r,hook,"\Omega^1_d(\iota^{n-1}_{d,E})"]& \Omega^1_d J^{n-1}_dE \ar[r,"\Omega^1_d(\pi^{n-1,n-2}_{d,E})"]& \Omega^1_d J^{n-2}_d E.
\end{tikzcd}
\end{equation}
Consider the following diagram
\begin{equation}
\begin{tikzcd}[column sep=50pt,row sep=30pt]
&[-20pt]E\ar[d,dashed,"\nabla'^n"']\ar[dr,"\nabla''^n"']\ar[r,near end,"j^{n-1}_{d,E}"']\ar[drr,bend left=40pt,"0"]&J^{n-1}_d E\ar[d,"\nabla"]\ar[dr,"\Spenc^{n-1,0}_{d,E}"']\\
0\ar[r]& \Omega^1_d S^{n-1}_d E \ar[r,hook,"\Omega^1_d(\iota^{n-1}_{d,E})"']& \Omega^1_d J^{n-1}_dE \ar[r,"\Omega^1_d(\pi^{n-1,n-2}_{d,E})"']& \Omega^1_d J^{n-2}_d E.
\end{tikzcd}
\end{equation}
The central triangle commutes by definition of $\nabla''^n$.
The right triangle commutes by \eqref{theo:higher_connections_are_connections:1}.
The top triangle commutes by Theorem \ref{theo:Spencer_is_complex}.
It follows, by the kernel universal property, that that $\nabla''^n$ factors uniquely through $\Omega^1_d(\iota^{n-1}_{d,E})$ as the dashed map $\nabla'^n$.
Now we prove that $\nabla'^n$ factors through $S^n_d E$.
This is straightforward for $n=1$, and for $n>1$, consider the following diagram
\begin{equation}\label{diag:bullet}
\begin{tikzcd}[column sep=50pt,row sep=30pt]
		E \ar[d,"\nabla'^n"']\ar[dr,"\nabla''^n"]\ar[r,"j^{n-1}_{d,E}"]&J^{n-1}_d E\ar[d,"\nabla"]\ar[dr,"R_\nabla"]\ar[dd,"0", controls={ +(0:6.5) and +(0:6.5)}]\\
		\Omega^1_d S^{n-1}_d E \ar[r,hook,"\Omega^1_d(\iota^{n-1}_{d,E})"] \ar[d,"-\delta^{n-1,1}_{d,E}"']	& \Omega^1_d J^{n-1}_d E \ar[d,"\Spenc_{d,E}^{n-1,1}"]\ar[r,"d_\nabla"]&\Omega^2_d J^{n-1}_d E\ar[dl,"\Omega^2_d(\pi^{n-1,n-2}_{d,E})"]\\
		\Omega^2_d S^{n-2}_d E \ar[r,hook,"\Omega^2_d (\iota^{n-2}_{d,E})"]& \Omega^2_d J^{n-2}_d E
\end{tikzcd}
\end{equation}
Each of the top three triangles in \eqref{diag:bullet} commute by definition.
The bottom left square diagram commutes by Proposition \ref{prop:Spencer bicomplex_is_bicomplex}.
The bottom triangle commutes by Lemma \ref{lemma:condition_1_HCC_correspondence_1_to_m} together with \eqref{theo:higher_connections_are_connections:1}.
Finally, the right curved triangle commutes by \eqref{theo:higher_connections_are_connections:2}.
It follows that
\begin{equation}
\Omega^2_d(\iota^{n-2}_{d,E}) \circ \delta^{n-1,1}_{d,E}\circ \nabla'^n
= 0.
\end{equation}
Since $\Omega^2_d$ is flat in $\ModA$ and $\iota^{n-2}_d$ is a mono, we have that $\Omega^2_d(\iota^{n-2}_{d,E})$ is also a mono, and thus $\delta^{n-1,1}_{d,E}\circ \nabla'^n=0$.
In turn, this implies that $\nabla'^n$ factors through $\ker(\delta^{n-1,1}_{d,E})=S^n_d E$, and we call the resulting map $\nabla^n\colon E\to S^n_d E$.
It follows that $\nabla^n$ is in fact a differential operator of order at most $n$ by Lemma \ref{lemma:images_of_DOS_are_DOS}, since $\Omega^1_d (\iota^{n-1}_{d,E})\circ \iota^n_{\wedge,E}$ is a mono and $\nabla''^n=\Omega^1_d (\iota^{n-1}_{d,E})\circ \iota^n_{\wedge,E}\circ \nabla^n$ is a differential operator of order at most $n$.

We will now compute the lift of $\nabla''^n$ via the following diagram obtained from \cite[\jetslemmaholprol]{FMW}.
\begin{equation}
\begin{tikzcd}[column sep=50pt,row sep=30pt]
J^n_d E\ar[r,hook,"l^{1,n}_{d,E}"]&J^1_d J^{n-1}_d E\ar[dr,"\widetilde{\nabla}"]\\
E\ar[r,hook,"j^{n-1}_{d,E}"]\ar[u,hook,"j^n_{d,E}"]\ar[rr,bend right=10pt,"\nabla''^n"']&J^{n-1}_d E\ar[u,hook,"j^1_{d,J^{n-1}_d E}"]\ar[r,"\nabla"]&\Omega^1_d J^{n-1}_d E
\end{tikzcd}
\end{equation}
This shows that the (unique) lift of $\nabla''^n$ to $J^n_d E$ is $\widetilde{\nabla}\circ l^{1,n}_{d,E}$.
Since $\nabla^n$ satisfies $\nabla''^n=\Omega^1_d (\iota^{n-1}_{d,E})\circ \iota^n_{\wedge,E}\circ \nabla^n$, we also know that its lift satisfies
\begin{equation}
\label{eq:lift_nabla^n}
\widetilde\nabla''^n=\Omega^1_d (\iota^{n-1}_{d,E})\circ \iota^n_{\wedge,E}\circ \widetilde\nabla^n.
\end{equation}

Now we will compute the restriction symbol of $\nabla^n$ using the following diagram
\begin{equation}
\label{diag:reverse_hurdy-gurdy}
\begin{tikzcd}
S^n_d E\ar[r,hookrightarrow,"\iota^n_{\wedge,E}"]\ar[d,hook,"\iota^n_{d,E}"]&\Omega^1_d S^{n-1}_d E\ar[r,hook,"\Omega^1_d(\iota^{n-1}_{d,E})"]&[40pt]\Omega^1_d J^{n-1}_d\ar[d,hook,"\iota^1_{d,J^{n-1}_d E}"']\ar[dr,equals]\\
J^n_d E\ar[rrrd,"\widetilde\nabla^n"']\ar[rr,hook,"l^{1,n-1}_{d,E}"]&&J^1_d J^{n-1}_d E\ar[r,"\widetilde{\nabla}"']&\Omega^1_d J^{n-1}_d E\\
&&&S^n_d E\ar[u,"\Omega^1_d (\iota^{n-1}_{d,E})\circ \iota^n_{\wedge,E}"']
\end{tikzcd}
\end{equation}
The pentagon commutes by \cite[\jetsdiagdefiiotand]{FMW}.
The square commutes by \eqref{eq:lift_nabla^n}.
The triangle commutes by \cite[\symbolspropsymbolsofconnections]{Symbol}.
We obtain the following equality
\begin{equation}
\Omega^1_d (\iota^{n-1}_{d,E})\circ \iota^n_{\wedge,E}\circ\widetilde{\nabla}^n\circ \iota^n_{d,E}
=\Omega^1_d (\iota^{n-1}_{d,E})\circ \iota^n_{\wedge,E},
\end{equation}
which in turn implies, by the fact that $\Omega^1_d (\iota^{n-1}_{d,E})\circ \iota^n_{\wedge,E}$ is a monomorphism, that the restriction symbol of $\nabla^n$ is $\id_{S^n_dE}$.
We can thus apply Proposition \ref{prop:characterization_higher_connections} to the differential operator $\nabla^n$ to obtain the associated $n$-connection.

We are left to show that these constructions are inverse to one another.
We start from a connection $\nabla\colon J^{n-1}_d E\to \Omega^1_d J^{n-1}_d E$ satisfying the properties \eqref{theo:higher_connections_are_connections:1} and \eqref{theo:higher_connections_are_connections:2}, and we construct the associated differential operator $\nabla^n\colon E\to S^n_d E$.
Then, we obtain the right split $C^n$ of the $n$-jet sequence corresponding to the left split $\widetilde{\nabla}^n$.
We have to prove the equality $\Spenc^{n,0}_{d,E}\circ C^n=\nabla$.
By definition of $C^n$, $\widetilde{\nabla}^n$, $\widetilde{\nabla}''^n$, and $\delta^{n,0}_{d,E}$, together with Proposition \ref{prop:Spencer bicomplex_is_bicomplex}, we compute
\begin{equation}\label{eq:rise_and_fall_of_the_connection}
\begin{split}
\Spenc^{n,0}_{d,E}\circ C^n\circ \pi^{n,n-1}_{d,E}
&=\Spenc^{n,0}_{d,E}\circ (\id_{J^n_d E}-\iota^n_{d,E}\circ \widetilde{\nabla}^n)\\
&=\Spenc^{n,0}_{d,E} -\Spenc^{n,0}_{d,E}\circ \iota^n_{d,E}\circ \widetilde{\nabla}^n\\
&=\Spenc^{n,0}_{d,E} + \Omega^1_d(\iota^{n-1}_{d,E})\circ\delta^{n,0}_{d,E}\circ \widetilde{\nabla}^n\\
&=\Spenc^{n,0}_{d,E} + \Omega^1_d(\iota^{n-1}_{d,E})\circ\iota^n_{\wedge,E}\circ \widetilde{\nabla}^n\\
&=\Spenc^{n,0}_{d,E} + \widetilde{\nabla}\circ l^{1,n-1}_{d,E}.
\end{split}
\end{equation}
The last equality follows from the square in \eqref{diag:reverse_hurdy-gurdy}.
The splitting in $\Mod$ of the $1$-jet exact sequence gives
\begin{equation}
\label{eq:1jes_split}
\id_{J^{n-1}_d E}
=j^1_{d,J^{n-1}_d E}\circ \pi^{1,0}_{d,J^{n-1}_d E}+\iota^1_{d,J^{n-1}_d E}\circ \rho_{d,J^{n-1}_d E},
\end{equation}
cf.\ \cite[\jetsproponejses]{FMW}.
If we compose both terms of this equality with $\widetilde{\nabla}$, we obtain
\begin{equation}
\label{eq:lift_decomposition_connection}
\widetilde{\nabla}
=\widetilde{\nabla}\circ j^1_{d,J^{n-1}_d E}\circ \pi^{1,0}_{d,J^{n-1}_d E}+\widetilde{\nabla}\circ \iota^1_{d,J^{n-1}_d E}\circ \rho_{d,J^{n-1}_d E}
=\nabla\circ \pi^{1,0}_{d,J^{n-1}_d E}+\rho_{d,J^{n-1}_d E}
\end{equation}
cf.\ \cite[\jetspropconnexionsplits]{FMW}.
If we substitute this expression for $\widetilde{\nabla}$ into \eqref{eq:rise_and_fall_of_the_connection}, by definition of the jet projection and Remark \ref{rmk:Spencer_m=0}, we obtain 
\begin{equation}
\Spenc^{n,0}_{d,E}\circ C^n\circ \pi^{n,n-1}_{d,E}
=\Spenc^{n,0}_{d,E} + \left(\nabla\circ \pi^{1,0}_{d,J^{n-1}_d E}+\rho_{d,J^{n-1}_d E}\right)\circ l^{1,n-1}_{d,E}
=\Spenc^{n,0}_{d,E} + \nabla\circ \pi^{n,n-1}_{d,E}-\Spenc^{n,0}_{d,E}
=\nabla\circ \pi^{n,n-1}_{d,E}.
\end{equation}
Since $\pi^{n,n-1}_{d,E}$ is an epi, it follows that we can cancel it from the extremal terms of this chain of equalities, obtaining that $\Spenc^{n,0}_{d,E}\circ C^n=\nabla$.

Since we proved that the construction of $\nabla$ from $C^n$ is left inverse to the construction of $C^n$ from $\nabla$, in order to show the bijection, it is sufficient to show that the first construction is injective.
In other words, we have to prove that given two $n$-connections $C^n$ and $C'^n$ such that $\Spenc^{n,0}_{d,E}\circ C^n=\Spenc^{n,0}_{d,E}\circ C'^n$, we have that $C^n=C'^n$.
Consider the following chain of equalities
\begin{equation}
\begin{split}
0
&=\iota^{1,0}_{d,J^{n-1}_d E}\circ\left(\Spenc^{n,0}_{d,E}\circ C^n-\Spenc^{n,0}_{d,E}\circ C'^n\right)\\
&=\iota^{1,0}_{d,J^{n-1}_d E}\circ\Spenc^{n,0}_{d,E}\circ(C^n- C'^n)\\
&=-\iota^{1,0}_{d,J^{n-1}_d E}\circ\left(\rho_{d,J^{n-1}_d E}\circ l^{1,n-1}_{d,E}\right)\circ(C^n- C'^n)\\
&=\left(j^1_{d,J^{n-1}_d E}\circ \pi^{1,0}_{d,J^{n-1}_d E}-\id_{J^1_d J^{n-1}_d E}\right)\circ l^{1,n-1}_{d,E}\circ(C^n- C'^n)\\
&=j^1_{d,J^{n-1}_d E}\circ \pi^{1,0}_{d,J^{n-1}_d E}\circ l^{1,n-1}_{d,E}\circ C^n
-j^1_{d,J^{n-1}_d E}\circ \pi^{1,0}_{d,J^{n-1}_d E}\circ l^{1,n-1}_{d,E}\circ C'^n
-l^{1,n-1}_{d,E}\circ(C^n- C'^n)\\
&=j^1_{d,J^{n-1}_d E}\circ \pi^{n,n-1}_{d,E}\circ C^n
-j^1_{d,J^{n-1}_d E}\circ \pi^{n,n-1}_{d,E}\circ C'^n
-l^{1,n-1}_{d,E}\circ(C^n- C'^n)\\
&=j^1_{d,J^{n-1}_d E}-j^1_{d,J^{n-1}_d E}-l^{1,n-1}_{d,E}\circ(C^n- C'^n)\\
&=-l^{1,n-1}_{d,E}\circ(C^n- C'^n)
\end{split}
\end{equation}
Since $l^{1,n}_{d,E}$ is a mono, we obtain that $C^n= C'^n$, completing this portion of the proof.

The final statement follows directly from Proposition \ref{prop:construction_connection_corresponding_to_higher_connection}.
\end{proof}
\begin{rmk}
Analogously to \eqref{eq:lift_decomposition_connection} in the proof of Theorem \ref{theo:higher_connections_are_connections}, we can deduce a decomposition of the lift of a connection $\nabla$ on $E$ as
\begin{equation}
\widetilde{\nabla}
=\nabla\circ \pi^{1,0}_{d,E}+\rho_{d,E}
=\nabla\circ \pi^{1,0}_{d,E}-\Spenc^{1,0}_{d,E}.
\end{equation}
\end{rmk}

\begin{rmk}
	The special case of Theorem \ref{theo:higher_connections_are_connections}, in the smooth setting, has appeared in the differential geometry literature (cf.\ \cite{eastwood2009higher}).
	An analogue in the holomorphic setting appears in \cite{HolomorphicHigherCon}.
\end{rmk}

\section{Quantization}
\label{s:Quantization}
In this section we discuss the notions of quantization, seen as a (right-) splitting of the symbol exact sequences, obtained from the definition of symbol \cite[\symbolsdefsymbolquotientdef]{Symbol}
\begin{equation}\label{eq:symbolexactsequence}
	\begin{tikzcd}
		0 & {\Diff^{n-1}_d(E,F)} &[20pt] {\Diff^{n}_d(E,F)} &[20pt] {\Symb^{n}_d(E,F)} & 0.
		\arrow[from=1-1, to=1-2]
		\arrow[hook, from=1-2, to=1-3]
		\arrow["{\symb^n_{d,E,F}}",two heads, from=1-3, to=1-4]
		\arrow["{q^n}", shift left=3, shift right=3, bend left=22pt,hook, from=1-4, to=1-3]
		\arrow[from=1-4, to=1-5]
	\end{tikzcd}
\end{equation}
\begin{rmk}
In this section, we will develop the notion of quantization in the setting of holonomic linear differential operators, but in principle one could also develop these notions {\it muatatis mutandis} for elemental or primitive differential operators as well.
\end{rmk}

\subsection{Splittings of the symbol exact sequence}
\begin{defi}
An \emph{$n$-quantization} for $(E,F)$ is an $\AHom(F,F)$-linear right splitting of $\symb^n_{d,E,F}$, i.e.\ 
\begin{equation}
q^n\colon \Symb^n_d(E,F)\longrightarrow \Diff^n_d(E,F).
\end{equation}
such that $\symb^n_{d,E,F}\circ q^n=\id_{\Symb^n_d(E,F)}$.

Given a collection $\{q^n |n\in \N\}$, where $q^n\colon \Symb^n_d(E,F)\longrightarrow \Diff^n_d(E,F)$ is an $n$-quantization for $(E,F)$, there exists a unique map
\begin{equation}
q\coloneq \sum_{n\in \N} q^n\colon \Symb^{\bullet}_d(E,F)\longrightarrow \Diff_d(E,F),
\end{equation}
called \emph{(full) quantization} for $(E,F)$.
\end{defi}
\begin{rmk}
A collection $\{q^n|n\in\N\}$, with $q^n$ an $n$-quantization for $(E,F)$, induces a unique full quantization for $(E,F)$ by the universal property of the coproduct, and the sum is well-defined, because elements of $\Symb^{\bullet}_d(E,F)$ are finite sums.
\end{rmk}
We can also give a version of this definition that is natural in $F$.
\begin{defi}
A \emph{natural $n$-quantization} for $E$ in $\AMod$ is a natural transformation of functors $\AMod\to \Mod$
\begin{equation}
q^n\colon \Symb^n_d(E,-)\longrightarrow \Diff^n_d(E,-)
\end{equation}
that is a right splitting of $\symb^n_{d,E}$, i.e.\ such that $\symb^n_{d,E}\circ q^n=\id_{\Symb^n_d(E,-)}$.

Given a collection $\{q^n |n\in \N\}$, where $q^n\colon \Symb^n_d(E,-)\longrightarrow \Diff^n_d(E,-)$ is a natural $n$-quantization for $E$, there exists a unique map
\begin{equation}
q\coloneq \sum_{n\in \N} q^n\colon \Symb^{\bullet}_d(E,-)\longrightarrow \Diff_d(E,-)
\end{equation}
called \emph{natural (full) quantization} for $E$.
\end{defi}
\begin{rmk}
For all $F$ in $\AMod$, the naturality with respect to endomorphisms of $F$ of a natural $n$-quantization implies that a natural $n$-quantization $q^n$ for $E$, in each component $q^n_F$, is an $n$-quantization for $(E,F)$.
\end{rmk}

\begin{rmk}
\label{rmk:0quantizationisidentity}
A $0$-quantization for $E$ is the identity, as it is the unique section of $\symb^0_{d,E,-}=\id_{\Diff^0_d(E,-)}$.
\end{rmk}
If we now assume the representability of differential operators of order $n$, cf.\ \cite[\symbolsssjetmodulesrepresentingobjects]{Symbol}, and of symbols of order $n$, cf.\ \cite[\symbolsssrepresentabilitysymbols]{Symbol}, we get the following result.
\begin{theo}\label{theo:quantizations_higher_connections}
Let $E$ in $\AMod$ be such that $J^n_d E\cong \EJ^n_d E$ and such that the $n$-jet sequence is split exact.
Then, natural $n$-quantizations $q^n\colon \Symb^n_d(E,-)\to \Diff^n_d(E,-)$ are in bijective correspondence with $n$-connections $C^n\colon J^{n-1}_d E\to J^n_d E$.
Explicitly, the correspondence is as follows
\begin{equation}\label{eq:explicit_quantization}
q^n(\sigma^n)=r^n_{d,E,S^n_d E}(\sigma^n)\circ \nabla^n,
\end{equation}
for all symbols $\sigma^n$, where $\nabla^n\colonequals \widetilde{\nabla}^n\circ j^n_{d,E}$ and $\widetilde{\nabla}^n\colon J^n_d E\twoheadrightarrow S^n_d E$ is the left split associated to $C^n$.
\end{theo}
\begin{proof}
The functor $\Diff^n_d(E,-)\colon \AMod\to \Mod$ is representable with representing object $J^n_d E$, cf.\ \cite[\symbolspropelementaljetpropertiesfour]{Symbol}.
Since the $n$-jet sequence is split exact, the functor $\Symb^n_d (E,-)$ is also representable with representing object $S^n_d E$, cf.\ \cite[\symbolspropsymbolsrepresentationwelldefined]{Symbol}.
Via these isomorphisms, the map $\symb^n_{d,E}$ is naturally isomorphic to
\begin{equation}
-\circ \iota^n_{d,E}\colon\AHom(J^n_d E,-)\longrightarrow\AHom(S^n_d E,-).
\end{equation}
Now, given an $n$-connection $C^n_d$, we consider the unique associated left split of the $n$-jet sequence $\widetilde{\nabla}^n$.
We can define
\begin{equation}
q^n\colonequals -\circ \widetilde{\nabla}^n\colon \AHom(S^n_d E,-)\longrightarrow\AHom(J^n_d E,-),
\end{equation}
which is an $n$-quantization, since $\symb^n_d \circ q^n=-\circ \widetilde{\nabla}^n\circ \iota^n_{d,E} =\id_{\AHom(S^n_d E,-)}$.
Vice versa, any natural transformation
\begin{equation}
q^n\colon \AHom(S^n_d E,-)\longrightarrow\AHom(J^n_d E,-)
\end{equation}
is, by naturality, of the form $-\circ \lambda^n$ for some $\lambda^n\colon J^n_d E\to S^n_d E$, by the Yoneda lemma.
One can construct $\lambda^n$ as $q^n(\id_{S^n_d E})$.
We want to prove that $\lambda^n_d$ is a left splitting of the $n$-jet sequence, and thus corresponds uniquely to an $n$-connection by Proposition \ref{prop:characterization_higher_connections}.
We compute the following
\begin{equation}
\lambda^n_d\circ \iota^n_{d,E}
=\symb^n_{d,E}(\lambda^n_d)
=\symb^n_{d,E}(q^n(\id_{S^n_d E}))
=\id_{S^n_d E}.
\end{equation}
The association we have just defined corresponds to \eqref{eq:explicit_quantization}, and it forms a bijection via the equivalence of left and right splittings of a short exact sequence.
\end{proof}
\begin{rmk}
For $n=1$ representability of differential operators and symbols always holds, so the hypotheses of Theorem \ref{theo:quantizations_higher_connections} holds automatically, cf.\ \cite[\jetspropconnexionsplits]{FMW}.
\end{rmk}

\subsection{Building higher order connections}
\begin{prop}
\label{prop:nablatilden+1}
Assume that in $\AMod$ we have a left splitting $\widetilde{\nabla}^n\colon J^n_d E\to S^n_d E$ for $\iota^n_{d,E}$ and a left splitting $\Sym^{1,n}\colon \Omega^1_d S^n_d E\twoheadrightarrow S^{n+1}_d E$ for $\iota^{n+1}_{\wedge,E}$.
Further, suppose we have a connection $\nabla^{S^n_d E}\colon S^n_d E\to \Omega^1_d S^n_d E$, then the composition
	\begin{equation}
		\label{eq:nablatilden+1}
		\widetilde{\nabla}^{n+1}
		= \Sym^{1,n} \circ \widetilde{\nabla}^{S^{n}_dE} \circ J^1_d(\widetilde{\nabla}^n) \circ l^{1,n}_{d,E}
	\end{equation}
	is a left splitting in $\AMod$ for $\iota^{n+1}_{d,E}$.
\end{prop}
\begin{proof}
We prove that \eqref{eq:nablatilden+1} is a retraction of $\iota^{n+1}_{d,E}$.
\begin{equation}
\begin{tikzcd}
	S^{n+1}_d E\ar[rrrrrd,bend left=40pt,equals]\ar[d,"\iota^{n+1}_{d,E}"]\ar[r,"\iota^n_{\wedge,E}"]&\Omega^1_d S^n_d E\ar[rr,bend left=40pt,equals]\ar[r,"\Omega^1_d(\iota^n_{d,E})"]&\Omega^1_d J^n_d E\ar[d,"\iota^1_{d,J^n_d E}"']\ar[r,"\Omega^1_d(\widetilde{\nabla}^n)"]&\Omega^1_d S^n_d E \ar[d,"\iota^1_{d,S^n_d E}"']\ar[dr,equals]\\
	J^{n+1}_d E\ar[rr,"l^{1,n}_{d,E}"']&&J^1_d J^n_d E\ar[r,"J^1_d(\widetilde{\nabla}^n)"']&J^1_d S^n_d E\ar[r,"\widetilde{\nabla}^{S^n_d E}"']&\Omega^1_d S^n_d E\ar[r,"\Sym^{1,n}"']&S^{n+1}_d E
\end{tikzcd}
\end{equation}
The left pentagon commutes by definition of $\iota^{n+1}_{d,E}$, cf.\ \cite[\jetsdiagdefiiotand]{FMW}.
The central square commutes by naturality of $\iota^1_d$ with respect to $\widetilde{\nabla}^n$.
The right triangle commutes because the symbol of a connection is the identity, cf.\ \cite[\symbolspropsymbolsofconnections]{Symbol}.
Finally, the topmost curved maps are identities because $\widetilde{\nabla}^n$ is a retraction of $\iota^n_{d,E}$ with $\Omega^1_d$ a functor, and because $\Sym^{1,n}$ is a retraction of $\iota^n_{\wedge,E}$.
\end{proof}
\begin{cor}
\label{cor:splittingiotan}
Let $n\ge 0$, and suppose for all $0\le k< n$ we have a connection $\nabla^{S^k_d E}\colon S^k_d E\to \Omega^1_d S^k_d E$ on $S^k_d E$, and a left inverse $\Sym^{1,k}\colon \Omega^1_d S^k_d E\to S^{k+1}_d E$ of $\iota^{k+1}_{\wedge,E}$.
Then there exists a splitting for $\iota^n_{d,E}$.

In particular, if for all $n\ge 0$ we have a connection $\nabla^{S^n_d E}\colon S^n_d E\to \Omega^1_d S^n_d E$ on $S^n_d E$, and a left inverse $\Sym^{1,n}\colon \Omega^1_d S^n_d E\twoheadrightarrow S^{n+1}_d E$ for $\iota^{n+1}_{\wedge,E}$, then we have a left splitting $\widetilde{\nabla}^n$ for $\iota^n_{d,E}$ for all $n\ge 0$.
\end{cor}
\begin{proof}
For $n=0$, we have $\iota^0_{d,E}=\id_E$, which is naturally a split monomorphism.
For $n>0$, by inductive hypothesis we have a left splitting $\widetilde{\nabla}^{n-1}$ for $\iota^{n-1}_{d,E}$.
We can thus apply Proposition \ref{prop:nablatilden+1} with the appropriate indices to conclude the existence of a splitting $\widetilde{\nabla}^n$ of $\iota^n_{d,E}$.
\end{proof}
\begin{cor}[Full quantization]
\label{cor:full_quantisation}
Let $E$ in $\AMod$ be such that $J^n_d E\cong \EJ^n_d E$, such that the $n$-jet sequence at $E$ is right exact, and suppose that for all $0\le k<n$ we have a connection $\nabla^{S^k_d E}$ on $S^k_d E$ and a left splitting $\Sym^{1,k}$ for $\iota^{k+1}_{\wedge,E}$.
Then there is an induced natural $n$-quantization $q^n\colon \Symb^n_d(E,-)\to \Diff^n_d(E,-)$ mapping a symbol $\sigma^n$ into
\begin{equation}\label{eq:explicit_quantization_with_connections_on_Sn}
q^n(\sigma^n)
=r^n_{d,E,S^n_d E}(\sigma^n)\circ \Sym^{1,n-1} \circ \nabla^{S^{n-1}_dE} \circ\Sym^{1,n-2} \circ \nabla^{S^{n-2}_dE} \circ\dots\circ \Sym^{1,1} \circ \nabla^{S^1_dE} \circ \nabla^E.
\end{equation}

Let $E$ in $\AMod$ be such that for all $n\ge 0$ we have $J^n_d E\cong \EJ^n_d E$ such that the $n$-jet sequence at $E$ is right exact, and suppose that we have a connection $\nabla^{S^n_d E}$ on $S^n_d E$ and a left splitting $\Sym^{1,n}$ for $\iota^{n+1}_{\wedge,E}$.
Then there is an induced natural full quantization $q$ for $E$.
\end{cor}
\begin{proof}
The hypotheses of Corollary \ref{cor:splittingiotan} being satisfied, we thus have a splitting $\widetilde{\nabla}^n$ for $\iota^n_{d,E}$.
Given the stated hypotheses, this means that the $n$-jet sequence is split exact at $E$.
We can thus apply Theorem \ref{theo:quantizations_higher_connections}, showing that we have an $n$-quantization for $E$.
By \eqref{eq:explicit_quantization}, this quantization is given explicitly on a symbol $\sigma^n$ by
\begin{equation}
q^n(\sigma^n)=r^n_{d,E,S^n_d E}(\sigma^n)\circ \nabla^n.
\end{equation}
Where $\nabla^n=\widetilde{\nabla}^n\circ j^n_{d,E}$, so by \eqref{eq:nablatilden+1}, we write
\begin{equation}
\begin{split}
\nabla^n
&=\widetilde{\nabla}^n\circ j^n_{d,E}\\
&=\Sym^{1,n-1} \circ \widetilde{\nabla}^{S^{n-1}_d E} \circ J^1_d(\widetilde{\nabla}^{n-1}) \circ l^{1,n-1}_{d,E}\circ j^n_{d,E}\\
&=\Sym^{1,n-1} \circ \widetilde{\nabla}^{S^{n-1}_d E} \circ J^1_d(\widetilde{\nabla}^{n-1}) \circ j^1_{J^{n-1}_d E}\circ j^{n-1}_{d,E}\\
&=\Sym^{1,n-1} \circ \widetilde{\nabla}^{S^{n-1}_d E} \circ j^1_{d,E} \circ \widetilde{\nabla}^{n-1} \circ j^{n-1}_{d,E}\\
&=\Sym^{1,n-1} \circ \nabla^{S^{n-1}_d E} \circ \nabla^{n-1}.
\end{split}
\end{equation}
Thus, we obtain \eqref{eq:explicit_quantization_with_connections_on_Sn} by induction.

The second part of this result is obtained by applying the first to all $n$ and summing the $n$-quantizations.
\end{proof}
\begin{rmk}
	Corollary \ref{cor:full_quantisation} generalizes the result \cite[\symbolspropquantisation]{Symbol} in the noncommutative setting, as well as the classical the classical result \cite[Theorem~7, p.~90]{PalaisVectorBundles} from differential geometry.
\end{rmk}

In order to obtain the natural full quantization for $E$ in Corollary \ref{cor:full_quantisation}, we had to assume that each $S^n_d E$ came equipped with a left connection.
We will now show that this sufficient assumption is also necessary.
\begin{prop}
\label{prop:relation_jet_connection_sym_connection}
	Let $E$ in $\AMod$ be such that the $n$-jet sequence is split exact via an $n$-connection $C^n$.
	Given a connection $\nabla^{S^n_d E}$ on $S^n_d E$, there exists a canonical connection $\nabla^{J^n_d E}$ on $J^n_d E$ given by 
	\begin{equation}
	\label{eq:relation_jet_connection_sym_connection:jet}
		\nabla^{J^n_d E}
		\colonequals \Omega^1_d(C^n) \circ \Spenc^{n,0}_{d,E} + \Omega^1_d(\iota^n_{d,E}) \circ \nabla^{S^n_d E} \circ \widetilde{\nabla}^n,
	\end{equation}
	where $\widetilde{\nabla}^n$ is the left split corresponding to $C^n$.
	
	Vice versa, given a connection $\nabla^{J^n_d E}$ on $J^n_d E$, this induces a connection $\nabla^{S^n_d E}$ on $S^n_d E$ of the form
	\begin{equation}
	\label{eq:relation_jet_connection_sym_connection:sym}
		\nabla^{S^n_d E}
		\colonequals \Omega^1_d (\widetilde{\nabla}^n)\circ \nabla^{J^n_d E}\circ \iota^n_{d,E}.
	\end{equation}
	The latter construction is the left inverse of the former.
\end{prop}
\begin{proof}
	In order to prove that \eqref{eq:relation_jet_connection_sym_connection:jet} is a connection, we use \cite[\symbolspropsymbolsofconnections]{Symbol}, and so we only need to prove that $\nabla^{J^n_d E}$ is a differential operator of order at most $1$ with restriction symbol $\id_{\Omega^1_d J^n_d E}$.
	Each summand is a differential operator, being composition of differential operators of order $1$ or $0$, and each summand contains exactly one differential operator of order at most $1$.
	It follows that $\nabla^{J^n_d E}$ is a differential operator of order at most $1$.
	Using this composition, we now construct a lift of this differential operator to $J^1_d J^n_d E$, which will be unique because first order differential operators are always representable by the $1$-jet module.
	The lift is as follows
	\begin{equation}
		\widetilde{\nabla}^{J^n_d E}
		= \Omega^1_d(C^n) \circ \widetilde{\Spenc}^{n,0}_{d,E} + \Omega^1_d(\iota^n_{d,E}) \circ \widetilde{\nabla}^{S^n_d E} \circ J^1_d (\widetilde{\nabla}^n).
	\end{equation}
	We will now compute the restriction symbol via \cite[\symbolspropsymbolsofconnections]{Symbol} and Proposition \ref{prop:spencopsymbol}.
	\begin{equation}
	\begin{split}
		\widetilde{\nabla}^{J^n_d E}\circ \iota^1_{d,J^n_d E}
		&= \Omega^1_d(C^n) \circ \widetilde{\Spenc}^{n,0}_{d,E}\circ \iota^1_{d,J^n_d E} + \Omega^1_d(\iota^n_{d,E}) \circ \widetilde{\nabla}^{S^n_d E}\circ\iota^1_{d,S^n_d E} \circ \Omega^1_d (\widetilde{\nabla}^n)\\
		&= \Omega^1_d(C^n) \circ \Omega^1_d(\pi^{n,n-1}_{d,E}) + \Omega^1_d(\iota^n_{d,E}) \circ \id_{\Omega^1_d S^n_d E} \circ \Omega^1_d (\widetilde{\nabla}^n)\\
		&=\Omega^1_d(C^n\circ \pi^{n,n-1}_{d,E}+\iota^n_{d,E}\circ \widetilde{\nabla}^n)\\
		&=\Omega^1_d(\id_{J^n_d E})\\
		&=\id_{\Omega^1_dJ^n_d E}.
	\end{split}
	\end{equation}
	It follows that $\nabla^{J^n_d E}$ is a connection on $J^n_d E$.
	
	Next, given a connection $\nabla^{J^n_d E}$ on $J^n_d E$, we prove that $\nabla^{S^n_d E}$ as in \eqref{eq:relation_jet_connection_sym_connection:sym} is a connection.
	Being the composition of a differential operators of order $1$ and two of order $0$, $\nabla^{S^n_d E}$ is a differential operator of order at most $1$, and its lift is as follows
	\begin{equation}
		\widetilde{\nabla}^{S^n_d E}
		=\Omega^1_d(\widetilde{\nabla}^n)\circ\widetilde{\nabla}^{J^n_d E}\circ J^1_d(\iota^n_{d,E})
		\colon J^1_d S^n_d E \longrightarrow \Omega^1_d S^n_d E.
	\end{equation}
	Thus, $\nabla^{S^n_d E}$ is a connection if and only if $\widetilde{\nabla}^{S^n_d E}\circ \iota^1_{d,S^n_d}=\id_{\Omega^1_d S^n_d E}$.
	We have
	\begin{equation}
	\begin{split}
		\widetilde{\nabla}^{S^n_d E}\circ \iota^1_{d,S^n_d}
		&= \Omega^1_d(\widetilde{\nabla}^n)\circ \widetilde{\nabla}^{J^n_d E} \circ J^1_d(\iota^n_{d,E})\circ \iota^1_{d,S^n_d E}\\
		&= \Omega^1_d(\widetilde{\nabla}^n)\circ \widetilde{\nabla}^{J^n_d E} \circ \iota^1_{d,J^n_d}\circ \Omega^1_d(\iota^n_{d,E})\\
		&= \Omega^1_d(\widetilde{\nabla}^n)\circ \id_{\Omega^1_d J^n_d E} \circ \Omega^1_d(\iota^n_{d,E})\\
		&= \Omega^1_d(\widetilde{\nabla}^n\circ\iota^n_{d,E})\\
		&=\id_{\Omega^1_d S^n_d E}.
	\end{split}
	\end{equation}
	Hence, $\nabla^{S^n_d E}$ is a connection, as desired.

	It remains to show that applying the latter construction to the former recovers the initial connection $\nabla^{S^n_d E}$.
	This holds by the following computation
	\begin{equation}
	\begin{split}
		&\Omega^1_d(\widetilde{\nabla}^n)\circ \nabla^{J^n_d E} \circ \iota^n_{d,E}\\
		&\qquad=\Omega^1_d(\widetilde{\nabla}^n)\circ(\Omega^1_d(C^n) \circ \Spenc^{n,0}_{d,E} + \Omega^1_d(\iota^n_{d,E}) \circ \nabla^{S^n_d E} \circ \widetilde{\nabla}^n)\circ \iota^n_{d,E}\\
		&\qquad=0+\Omega^1_d(\widetilde{\nabla}^n\circ \iota^n_{d,E}) \circ \nabla^{S^n_d E} \circ \widetilde{\nabla}^n\circ \iota^n_{d,E}\\
		&\qquad=\nabla^{S^n_d E}.
	\end{split}
	\end{equation}
\end{proof}
This immediately implies the following.
\begin{cor}[Partial converse to Corollary \ref{cor:full_quantisation}]
	Let $\Omega^1_d$ and $\Omega^2_d$ be flat in $\ModA$, and let $E$ in $\AMod$ be such that $J^n_d E\cong \EJ^n_d E$ and such that the $n$-jet sequence is exact.
	Suppose we have a natural full quantization $q$ for $E$.
	Then, for every $n\in \N$, we have a left connection $\nabla^{S^n_dE}$ on $S^n_dE$.
\end{cor}
\begin{proof}
	For $n \in \N$, we have the following.
	First, the bijection from Theorem \ref{theo:quantizations_higher_connections} yields a higher order connection $C^n$ from the $n$-quantization $q^n$.
	Next, the bijection from Theorem \ref{theo:higher_connections_are_connections} takes $C^n$ to the left connection $\nabla^{J^{n-1}_dE}$ on $J^{n-1}_d E$.
	Finally, Proposition \ref{prop:relation_jet_connection_sym_connection} constructs the desired connection $\nabla^{S^{n-1}_d E}$ from $\nabla^{J^{n-1}_d E}$.
\end{proof}

Let us also compute an explicit expression for the exterior covariant derivative and the curvature of the associated connection on $J^n_d E$, arising from Proposition \ref{prop:relation_jet_connection_sym_connection}.
These are particularly interesting as they appear in the formulation of Theorem \ref{theo:higher_connections_are_connections}.
\begin{prop}
\label{prop:ECV_and_curvature_connection_Jn}
Let $E$ in $\AMod$ be such that the $n$-jet sequence at $E$ is split exact via an $n$-connection $C^n$.
Let $\nabla^{J^n_d E}\colon J^n_d E\to \Omega^1_d J^n_d E$ be induced by a connection $\nabla^{S^n_d E}$ on $S^n_d E$ as in \eqref{eq:relation_jet_connection_sym_connection:jet}.
Then, the associate exterior covariant derivative $d_{\nabla^{J^n_d E}}\colon \Omega^m_d J^n_d\longrightarrow \Omega^{m+1}_d J^n_d$, is of the form
\begin{equation}
\label{eq:ECV_connection_Jn}
d_{\nabla^{J^n_d E}}=\Omega^{m+1}_d(C^n)\circ d_{C^n}\circ\Omega^m_d(\pi^{n,n-1}_{d,E})+\Omega^{m+1}_d (\iota^n_{d,E})\circ d_{\nabla^{S^n_d E}}\circ \Omega^m_d(\widetilde{\nabla}^n)-\Omega^{m+1}_d(C^n\circ \iota^{n-1}_{d,E})\circ \delta^{n,m}_{d,E}\circ \Omega^m_d(\widetilde{\nabla}^n).
\end{equation}
Where $d_{C^n}$ and $d_{\nabla^{S^n_d E}}$ are the exterior covariant derivatives associated to $\nabla^{C^n}$ (cf.\ Definition \ref{def:exterior_covariant_derivative_corresponding_to_higher_connection}), and $\nabla^{S^n_d E}$, respectively, and $\delta^{n,m}_{d,E}$ is the Spencer $\delta$-operator.
In other words, with respect to the splitting $J^n_d E=J^{n-1}_d E\oplus S^n_d E$ induced by $C^n$, we can write
\begin{equation}
\label{eq:ECV_connection_Jn:matrix}
	d_{\nabla^{J^n_d E}}
	=\left(\begin{array}{c|c}
	d_{C^n} & -\Omega^{m+1}_d(\iota^{n-1}_{d,E})\circ\delta^{n,m}_{d,E}\\
	\hline
	0 & d_{\nabla^{S^n_d E}}
	\end{array}\right).
\end{equation}
Its curvature is
\begin{equation}
\label{eq:curvature_connection_Jn}
\begin{split}
	R_{\nabla^{C^n}}
	&=\Omega^2_d(C^n)\circ R_{C^n}\circ \pi^{n,n-1}_{d,E}
	+\Omega^2_d (\iota^n_{d,E})\circ R_{\nabla^{S^n_d E}}\circ \widetilde{\nabla}^n\\
	&\quad-\Omega^2_d(C^n)\circ \left(d_{C^n}\circ\Omega^1_d(\iota^{n-1}_{d,E})\circ\delta^{n,0}_{d,E}+\Omega^2_d(\iota^{n-1}_{d,E})\circ\delta^{n,1}_{d,E}\circ \nabla^{S^n_d E}\right)
	\circ \widetilde{\nabla}^n,
\end{split}
\end{equation}
or in matrix form
\begin{equation}
\label{eq:curvature_connection_Jn:matrix}
R_{\nabla^{C^n}}
=\left(\begin{array}{c|c}
	R_{C^n} & -d_{C^n}\circ\Omega^1_d(\iota^{n-1}_{d,E})\circ\delta^{n,0}_{d,E}-\Omega^2_d(\iota^{n-1}_{d,E})\circ\delta^{n,1}_{d,E}\circ \nabla^{S^n_d E}\\
	\hline
	0 & R_{\nabla^{S^n_d E}}
	\end{array}\right).
\end{equation}
\end{prop}
\begin{proof}
We prove the formula for the exterior covariant derivative explicitly on an element $\omega\otimes_A \xi\in \Omega^m_d J^n_d E$
\begin{equation}
\label{eq:ECV_and_curvature_connection_Jn}
\begin{split}
	&d_{\nabla^{J^n_d E}}(\omega\otimes_A \xi)\\
	&\qquad=d\omega\otimes_A \xi+(-1)^{\deg(\omega)}\omega\wedge \nabla^{J^n_d E}(\xi)\\
	&\qquad=d\omega\otimes_A \left(\iota^n_{d,E}\circ \widetilde{\nabla}^n+C^n\circ \pi^{n,n-1}_{d,E}\right)(\xi)\\
	&\qquad\quad+ (-1)^{\deg(\omega)}\omega\wedge \Omega^1_d(C^n) \circ \Spenc^{n,0}_{d,E}(\xi) + (-1)^{\deg(\omega)}\omega\wedge \Omega^1_d(\iota^n_{d,E}) \circ \nabla^{S^n_d E} \circ \widetilde{\nabla}^n(\xi)\\
	&\qquad=\Omega^{m+1}(\iota^n_{d,E})\left(d\omega\otimes_A \widetilde{\nabla}^n(\xi)\right)+\Omega^{m+1}(C^n)\left( d\omega\otimes_A \pi^{n,n-1}_{d,E}(\xi)\right)\\
	&\qquad\quad + (-1)^{\deg(\omega)}\Omega^{m+1}_d(C^n) \left(\omega\wedge \Spenc^{n,0}_{d,E}(\xi)\right) + (-1)^{\deg(\omega)}\Omega^{m+1}_d(\iota^n_{d,E}) \left(\omega\wedge \nabla^{S^n_d E} \circ \widetilde{\nabla}^n(\xi)\right)\\
	&\qquad=\Omega^{m+1}(\iota^n_{d,E})\left(d\omega\otimes_A \widetilde{\nabla}^n(\xi)+ (-1)^{\deg(\omega)}\omega\wedge \nabla^{S^n_d E} \circ \widetilde{\nabla}^n(\xi)\right)\\
	&\qquad\quad +\Omega^{m+1}(C^n)\left( d\omega\otimes_A \pi^{n,n-1}_{d,E}(\xi)+ (-1)^{\deg(\omega)}\omega\wedge \Spenc^{n,0}_{d,E}(\xi)\right)\\
	&\qquad=\Omega^{m+1}(\iota^n_{d,E})\circ d_{\nabla^{S^n_d E}} (\omega\otimes_A \widetilde{\nabla}^n(\xi))
	+\Omega^{m+1}(C^n)\circ \Spenc^{n,m}_{d,E}(\omega\otimes_A \xi)\\
	&\qquad=\left(\Omega^{m+1}(\iota^n_{d,E})\circ d_{\nabla^{S^n_d E}}\circ \Omega^m_d(\widetilde{\nabla}^n)
	+\Omega^{m+1}(C^n)\circ \Spenc^{n,m}_{d,E}\right)(\omega\otimes_A \xi)
\end{split}
\end{equation}
The equalities follow from the Leibniz rule, cf.\ \eqref{eq:exterior_derivative_is_as_expected}, the splitting of $J^n_d E$, the naturality of $\wedge$ with respect to $\iota^n_{d,E}$ and $C^n$, and Proposition \ref{prop:Spencercorrespondence}.
Thanks to the splitting given by $C^n$, we have
\begin{equation}
\begin{split}
	\Spenc^{n,m}_{d,E}
	&=\Spenc^{n,m}_{d,E}\circ \Omega^m_d(C^n\circ \pi^{n,n-1}_{d,E}+\iota^n_{d,E}\circ \widetilde{\nabla}^n)\\
	&=\Spenc^{n,m}_{d,E}\circ \Omega^m_d(C^n)\circ \Omega^m_d(\pi^{n,n-1}_{d,E})+\Spenc^{n,m}_{d,E}\circ \Omega^m_d(\iota^n_{d,E})\circ \Omega^m_d(\widetilde{\nabla}^n)\\
	&=d_{C^n}\circ \Omega^m_d(\pi^{n,n-1}_{d,E})-\Omega^{m+1}_d(\iota^{n-1}_{d,E})\circ\delta^{n,m}_{d,E} \circ \Omega^m_d(\widetilde{\nabla}^n).
\end{split}
\end{equation}
Substituting this result into \eqref{eq:ECV_and_curvature_connection_Jn}, we obtain the desired formula \eqref{eq:ECV_connection_Jn} and the corresponding matrix form \eqref{eq:ECV_connection_Jn:matrix}, where we split $\Omega^m_d J^n_d E=\Omega^m_d J^{n-1}_d E\oplus \Omega^m_d S^n_d E$ using $C^n$.

In order to compute the curvature, one can use the explicit formula, but for convenience, we will use \eqref{eq:ECV_connection_Jn:matrix}.
The curvature is thus obtained as follows
\begin{equation}
\begin{split}
	R_{\nabla^{J^n_d E}}
	&=(d_{\nabla^{J^n_d E}})^2\\
	&=\left(\begin{array}{c|c}
	d_{C^n} & -\Omega^2_d(\iota^{n-1}_{d,E})\circ\delta^{n,1}_{d,E}\\
	\hline
	0 & d_{\nabla^{S^n_d E}}
	\end{array}\right)
	\circ
	\left(\begin{array}{c|c}
	d_{C^n} & -\Omega^1_d(\iota^{n-1}_{d,E})\circ\delta^{n,0}_{d,E}\\
	\hline
	0 & d_{\nabla^{S^n_d E}}
	\end{array}\right)\\
	&=\left(\begin{array}{c|c}
	d^2_{C^n} & -d_{C^n}\circ\Omega^1_d(\iota^{n-1}_{d,E})\circ\delta^{n,0}_{d,E}-\Omega^2_d(\iota^{n-1}_{d,E})\circ\delta^{n,1}_{d,E}\circ d_{\nabla^{S^n_d E}}\\
	\hline
	0 & d^2_{\nabla^{S^n_d E}}
	\end{array}\right)\\
	&=\left(\begin{array}{c|c}
	R_{C^n} & -d_{C^n}\circ\Omega^1_d(\iota^{n-1}_{d,E})\circ\delta^{n,0}_{d,E}-\Omega^2_d(\iota^{n-1}_{d,E})\circ\delta^{n,1}_{d,E}\circ d_{\nabla^{S^n_d E}}\\
	\hline
	0 & R_{\nabla^{S^n_d E}}
	\end{array}\right)
\end{split}
\end{equation}
This automatically gives \eqref{eq:curvature_connection_Jn}.
\end{proof}
\begin{rmk}
From the proof of Proposition \ref{prop:ECV_and_curvature_connection_Jn}, we obtain that in the presence of an $n$-connection $C^n$, the Spencer operator can be written as follows
\begin{equation}
	\Spenc^{n,m}_{d,E}
	=d_{C^n}\circ \Omega^m_d(\pi^{n,n-1}_{d,E})-\Omega^{m+1}_d(\iota^{n-1}_{d,E})\circ\delta^{n,m}_{d,E} \circ \Omega^m_d(\widetilde{\nabla}^n),
\end{equation}
or in matrix notation given by the splitting $\Omega^m_d J^n_d E=\Omega^m_d J^{n-1}_d E\oplus \Omega^m_d S^n_d E$ induced by $\Omega^m_d (C^n)$:
\begin{equation}
\Spenc^{n,m}_{d,E}
=\left(\begin{array}{c|c}
	d_{C^n} & -\Omega^{m+1}_d(\iota^{n-1}_{d,E})\circ\delta^{n,m}_{d,E}
	\end{array}\right).
\end{equation}
\end{rmk}

\subsubsection{Constructing connections on modules of symmetric forms}

\begin{prop}
\label{prop:cooking_up_connections_on_symmetric_forms}
	Consider a family of retractions $\Sym^{1,n}\colon \Omega^1_d S^n_d\twoheadrightarrow S^{n+1}_d$ of $\iota^{n+1}_\wedge$ in $\AModA$ for $n\in \N$.
	Suppose we have a bimodule connection $\nabla^{\Omega^1_d}$ on $\Omega^1_d$.
	Let $E$ be in $\AMod$, and equipped with a left connection $\nabla^E$.
	Then we have a family of left connections $\nabla^{S^n_d E}\colon S^n_d E \rightarrow \Omega^1_d S^n_d E$ for $n\in \N$.
\end{prop}
\begin{proof}
	We proceed by induction on $n$.
	For $n=0$, the desired connection on $S^0_d E=E$ is $\nabla^E$.
	For the case $n=1$, we have $S^1_d E = \Omega^1_d E$.
	We recall that, since $\nabla^{\Omega^1_d}$ is a bimodule connection, we obtain a left connection on any tensor module $\Omega^1_d F$ for an $F$ in $\AMod$ equipped with a left connection $\nabla^F$, cf.\ \cite[Theorem~3.78, p.~258]{BeggsMajid} \textit{mutatis mutandis}.
	The left connection on $\Omega^1_d E$ is thus of this form.
	
	Next, for $n\ge 2$, by inductive hypothesis, we assume that we have a left connection $\nabla^{S^{n-1}_d E}$ on $S^{n-1}_d E$.
	Then, by the aforementioned argument, we also have a left connection $\nabla^{\Omega^1_d S^{n-1}_dE}$ on $\Omega^1_d S^{n-1}_d E$.
	We define $\nabla^{S^n_d E}$ as the following composition:
	\begin{equation}
		\begin{tikzcd}[column sep=60pt]
			{S^n_d E} & {\Omega^1_d S^{n-1}_dE} & {\Omega^1_d\Omega^1_d S^{n-1}_dE} &[20pt] {\Omega^1_d S^n_dE}
			\arrow["{\iota^n_{\wedge,E}}", from=1-1, to=1-2]
			\arrow["{\nabla^{\Omega^1_d S^{n-1}_dE}}", from=1-2, to=1-3]
			\arrow["{\Omega^1_d(\Sym^{1,n-1}\otimes_A \id_E)}", from=1-3, to=1-4]
		\end{tikzcd}
	\end{equation}
	Let us show that $\nabla^{S^n_dE}$ is a left connection on $S^n_d E$.
	This is a composition of three differential operators, two of order $0$ and one of order $1$, and as such it is a differential operator of order at most $1$.
	We prove that it is a left connection by showing that its restriction symbol is $\id_{\Omega^1_d S^n_d E}$, cf.\ \cite[\symbolspropsymbolsofconnections]{Symbol}.
	We do so via the following diagram
	\begin{equation}
	\label{diag:whale}
		\begin{tikzcd}[column sep=70pt, row sep=30pt]
			\Omega^1_d S^n_d E\ar[rrrd,bend left=20pt,equals] \ar[r,"\Omega^1_d(\iota^n_{\wedge,E})"']\ar[d,"\iota^1_{d,S^n_d E}"]& \Omega^1_d \Omega^1_d S^{n-1}_dE\ar[d,"\iota^1_{d,\Omega^1_d S^{n-1}_d E}"']\ar[dr,bend left=20pt,equals]\\
			J^1_d S^n_d E\ar[r,"J^1_d (\iota^n_{\wedge,E})"]\ar[rrr,bend right=10pt,"\widetilde{\nabla}^{S^n_d E}"']& J^1_d \Omega^1_d S^{n-1}_dE\ar[r,"\widetilde{\nabla}^{\Omega^1_d S^{n-1}_dE}"]&\Omega^1_d \Omega^1_d S^{n-1}_dE\ar[r,"\Omega^1_d (\Sym^{1,n-1}\otimes_A \id_E)"]&[20pt]\Omega^1_d S^n_d E
		\end{tikzcd}
	\end{equation}
	The commutativity of the bottom square is a consequence of the construction of a lift for a composition of differential operators, cf.\ \cite[\jetspropdifferentialoperatorcomposition]{FMW}.
	The left square commutes by the naturality of $\iota^1_d$ with respect to $\iota^n_{\wedge,E}$.
	The central triangle commutes because $\nabla^{\Omega^1_d S^{n-1}_d E}$ is a connection, and the top square commutes because $\Sym^{1,n-1}$ is a retract of $\iota^n_\wedge$.
	It follows that the exterior triangle in \eqref{diag:whale} commutes.
	Hence the restriction symbol of $\nabla^{S^n_d E}$ is $\id_{\Omega^1_d S^n_d E}$.
\end{proof}

We can now prove the following result generalizing \cite[Corollary, p.~90]{PalaisVectorBundles} and \cite[Theorem~11, p.~18]{Lychagin1999}.
\begin{cor}
\label{cor:generalization_of_classical_quantization}
Let $E$ in $\AMod$ be such that $J^n_d E\cong \EJ^n_d E$ and such that the $n$-jet sequence at $E$ is right exact, for all $n\in \N$.
Suppose that we have a family of retractions $\Sym^{1,n}$ of $\iota^{n+1}_\wedge$ in $\AModA$ for all $n\in \N$ and a bimodule connection $\nabla^{\Omega^1_d}$ on $\Omega^1_d$.
Then, for each left connection $\nabla^E$ on $E$ there is an induced natural full quantization $q$ for $E$.
\end{cor}
\begin{proof}
It follows from Corollary \ref{cor:full_quantisation} together with Proposition \ref{prop:cooking_up_connections_on_symmetric_forms}.
\end{proof}

\subsection{Total symbols}
\label{s:Total_symbols}
In this section we will define the notion of total symbol for a differential operator, which we can interpret as a way to decompose a differential operator in components of homogeneous order.
In the classical setting of microlocal analysis, total symbols are obtained as functions on the cotangent space, via Fourier inversion of differential operators, cf.\ \cite[Example~3.1, p.~27]{microlocaldiffop}.
In order for total symbols to be well-defined in general, one needs a full quantization $q$ for $(E,F)$.
First, we give the following auxiliary definition.
\begin{defi}
\label{def:DOtruncation}
	Let $\Delta$ be a linear differential operator of order at most $n$.
	For all $k\in\N$, we define the \emph{$k$-truncation of $\Delta$ (induced by $q$)}, denoted by $\trunc{\Delta}^k_q$, as follows.
	Let $\trunc{\Delta}^k_q = \Delta$ for all $k\ge n$, and recursively define
	\begin{equation}
		\trunc{\Delta}^k_q
		= \trunc{\Delta}^{k+1}_q -q^{k+1}\circ \symb^{k+1}_{d,E,F}(\trunc{\Delta}^{k+1}_q)
	\end{equation}
	for $0 \le k < n$.
\end{defi}
In the following proposition, we prove that this definition is independent of $n$, and hence, the definition of $k$-truncation can be extended to $\Diff_d(E,F)$.
\begin{prop}
\label{prop:truncation_well_defined}
Let $q$ be a full quantization for $(E,F)$.
Then,
\begin{enumerate}
\item\label{prop:truncation_well_defined:1} Let $m\le n\in \N$, and let $\Delta\in \Diff^m_d (E,F)\subseteq \Diff^n_d(E,F)$.
Then, the $k$-truncation of $\Delta$ seen as a differential operator of order at most $m$ coincides with the $k$-truncation of $\Delta$ seen as a differential operator of order at most $n$.
\item\label{prop:truncation_well_defined:2} In particular, if $\Delta$ is a differential operator of order exactly $n$, we have $\trunc{\Delta}^k_q=\Delta$ for all $k\ge n$.
\item\label{prop:truncation_well_defined:3} For all $\Delta\in \Diff_q(E,F)$, we have $\trunc{\Delta}^k_q\in\Diff^k_d(E,F)$.
\item\label{prop:truncation_well_defined:4}
For $\Delta\in \Diff_d(E,F)$, we have $\trunc{\trunc{\Delta}^k_q}^h_q=\trunc{\Delta}^{\min(h,k)}_q$, for all $h, k\in\N$.
\item\label{prop:truncation_well_defined:5}
For $\Delta\in \Diff_d(E,F)$, if $\trunc{\Delta}^k_q=0$, then $\trunc{\Delta}^h_q=0$, for all $h\le k$.
\end{enumerate}
\end{prop}
\begin{proof}\
\begin{enumerate}
\item The statement is tautologically true for $k\ge n$ by definition.
For all $k$ such that $m\le k< n$, we need to show that $\trunc{\Delta}^k_q=\Delta$.
By the previous point, we know that this is true for $k=n$, so we proceed by induction.
We assume by inductive hypothesis on $k<n$ that $\trunc{\Delta}^{k+1}_q=\Delta$.
Then, $\symb^{k+1}_{d,E,F}(\trunc{\Delta}^{k+1}_q)=\symb^{k+1}_{d,E,F}(\Delta)$, but since $\Delta\in\Diff^m_d (E,F)\subseteq \Diff^k_d (E,F)$, we have $\symb^{k+1}_{d,E,F}(\Delta)=0$.
Since $q^{k+1}$ is $A$-bilinear, it follows by definition that $\trunc{\Delta}^k_q=\trunc{\Delta}^{k+1}_q-0=\Delta$.
The statement is now tautologically true also for $k<m$.
\item It follows as a direct consequence of \eqref{prop:truncation_well_defined:1}.
\item Suppose that $\Delta$ has order at most $n$.
We prove the statement by induction on $k$.
For $k\ge n$, $\trunc{\Delta}^k_q=\Delta\in\Diff^n_d(E,F)\subseteq \Diff^k_d(E,F)$.
Now, let $0\le k<n$ and assume by inductive hypothesis that $\trunc{\Delta}^{k+1}_q\in\Diff^{k+1}_d(E,F)$, then
\begin{equation}
	\trunc{\Delta}^k_q
		= \trunc{\Delta}^{k+1}_q -q^{k+1}\circ \symb^{k+1}_{d,E,F}(\trunc{\Delta}^{k+1}_q)\in\Diff^{k+1}_d(E,F).
\end{equation}
We prove that $\Delta^{(k)}$ has order at most $k$ by proving that its $(k+1)$-symbol vanishes.
We obtain
\begin{equation}
\begin{split}
	\symb^{k+1}_{d,E,F}(\trunc{\Delta}^k_q)
		&= \symb^{k+1}_{d,E,F}(\trunc{\Delta}^{k+1}_q -q^{k+1}\circ \symb^{k+1}_{d,E,F}(\trunc{\Delta}^{k+1}_q))\\
		&= \symb^{k+1}_{d,E,F}(\trunc{\Delta}^{k+1}_q) -\symb^{k+1}_{d,E,F}(q^{k+1}( \symb^{k+1}_{d,E,F}(\trunc{\Delta}^{k+1}_q)))\\
		&= \symb^{k+1}_{d,E,F}(\trunc{\Delta}^{k+1}_q) -\symb^{k+1}_{d,E,F}(\trunc{\Delta}^{k+1}_q)\\
		&=0.
	\end{split}
\end{equation}
\item For $h\ge k$, since $\trunc{\Delta}^k_q\in\Diff^k_d(E,F)$ by \eqref{prop:truncation_well_defined:3}, we obtain that $\trunc{\trunc{\Delta}^k_q}^h_q=\trunc{\Delta}^k_q$ by definition of $h$-truncation.
For $h<k$, we proceed by induction on $h$.
Since $\min(h+1,k)=h+1$, we obtain
\begin{equation}
		\trunc{\trunc{\Delta}^k_q}^h_q
		= \trunc{\trunc{\Delta}^k_q}^{h+1}_q -q^{h+1}\circ \symb^{h+1}_{d,E,F}(\trunc{\trunc{\Delta}^k_q}^{h+1}_q)
		= \trunc{\Delta}^{h+1}_q -q^{h+1}\circ \symb^{h+1}_{d,E,F}(\trunc{\Delta}^{h+1}_q)
		=\trunc{\Delta}^h_d.
	\end{equation}
\item It follows from the definition by the linearity of $q$.
\qedhere
\end{enumerate}
\end{proof}
Throughout this section, in order to simplify the notation, we will denote the $k$-symbol of the $k$-truncation of a linear differential operator $\Delta\colon E\to F$ of finite order as $[\Delta]^k_q = \symb^k_{d,E,F}(\trunc{\Delta}^k_q)$.

\begin{rmk}
\label{rmk:total_symbol_vanishing_for_high_k}
Given $\Delta\in\Diff^n_d(E,F)$, then $[\Delta]^k_q =0$ for $k> n$.
\end{rmk}
When the differential operator comes from an $n$-quantization, we can say more via the following lemma.
\begin{lemma}
\label{lemma:total_symbols_of_pure_quantization}
For all $\sigma\in\Symb^n_d(E,E)$, we have $[q^n(\sigma)]^k_q = \delta^{n,k}\sigma$, where $\delta^{n,k}$ is the Kronecker delta.
\end{lemma}
\begin{proof}
	Since $q^n(\sigma)\in \Diff^n_d (E,F)$, for all $k> n$, we have $[q^n(\sigma)]^k_q =0$ by Remark \ref{rmk:total_symbol_vanishing_for_high_k}.
	For $k=n$, we have by definition, $\trunc{q^n(\sigma)}^n_q=q^n(\sigma)$, and thus
	\begin{equation}
	[q^n(\sigma)]^{n}_q
	=\symb^n_{d,E,F}(q^n(\sigma))
	=\sigma,
	\end{equation}
	by definition of $n$-quantization.
	For $k=n-1$, we have
	\begin{equation}
	\trunc{q^n(\sigma)}^{n-1}_q
		= \trunc{q^n(\sigma)}^n_q -q^n\circ \symb^n_{d,E,F}(\trunc{q^n(\sigma)}^n_q)
		= q^n(\sigma) -q^n(\symb^n_{d,E,F}(q^n(\sigma)))
		= q^n(\sigma) -q^n(\sigma)
		=0.
	\end{equation}
	Proposition \ref{prop:truncation_well_defined}.\eqref{prop:truncation_well_defined:5}, yields $\trunc{q^n(\sigma)}^k_q=0$ for all $k\le n-1$.
	It follows that $[q^n(\sigma)]^k_q =0$ for all $k\le n-1$, which completes the proof.
\end{proof}
\begin{defi}[Total symbol]
\label{def:totalsymbol}
	Let $\Delta$ be a linear differential operator of finite order.
	We call the element
	\begin{equation}
\label{eq:totalsymbol}
		\symb_q(\Delta)\colonequals \sum_{k\in\N}[\Delta]^k_q \in \Symb^\bullet_d(E,F)
	\end{equation}
	the \emph{total symbol of $\Delta$ (with respect to the quantization $q$)}.
	
	We define the \emph{total symbol map (with respect to the quantization $q$)} to be the corresponding mapping 
	\begin{equation}
		\symb_q\colon \Diff_d(E,F)\longrightarrow \Symb^\bullet_d(E,F).
	\end{equation}
\end{defi}

\begin{prop}
\label{prop:isomorphism_quantization_total_symbol}
	A full quantization $q$ for $(E,F)$ has inverse $\symb_q$, realizing an $\AHom(F,F)$-linear isomorphism
	\begin{equation}
		\Diff_d(E,F)\simeq \Symb^\bullet_d(E,F).
	\end{equation}
	Consequently, a natural full quantization $q$ for $E$ realizes a natural isomorphism of functors $\AMod\to \Mod$
	\begin{equation}
		\Diff_d(E,-)\simeq \Symb^\bullet_d(E,-).
	\end{equation}
	Moreover, for $\Delta\in\Diff_d(E,F)$, we have
	\begin{equation}
	\label{eq:DO_decomposition}
		\Delta
		= \sum_{k\in\N} \Delta^{(k)}_q,
	\end{equation}
	where $\Delta^{(k)}_q\colonequals q^k([\Delta]^k_q)$.
\end{prop}
\begin{proof}
	We first prove that $\symb_q\circ q=\id_{\Symb^\bullet_d(E,F)}$.
	By the universal property of the coproduct, it is enough to show that for all $\sigma\in\Symb^n_d(E,F)$, we get $\symb_q\circ q(\sigma)=\sigma$, seen in $\Symb^\bullet_d (E,F)$ by extending it to $0$ in all components different from $n$.
	Lemma \ref{lemma:total_symbols_of_pure_quantization} yields the desired equality via the following computation
	\begin{equation}
	\symb_q \circ q(\sigma)
	=\symb_q\circ q^n(\sigma)
	=\sum_{k\in\N} [q^n(\sigma)]^k_q
	=\sigma.
	\end{equation}

	We now prove $q\circ\symb_q=\id_{\Diff_d(E,F)}$ elementwise on $\Delta\in\Diff_d(E,F)$.
	We proceed by induction on $n$, the order of $\Delta$.
	For the base case $n=0$, we have $\Delta\in\AHom(E,F)$.
	Remark \ref{rmk:total_symbol_vanishing_for_high_k} yields $\symb_q(\Delta)=[\Delta]^0_q=\Delta$.
	Thus, $q(\symb_d(\Delta))=q^0(\Delta)=\Delta$ by Remark \ref{rmk:0quantizationisidentity}.

	Now, assume the result holds for differential operators of order $n-1$, and let $\Delta$ be of order $n$.
	We have 
	\begin{equation}
		\trunc{\Delta}^{n-1}_q
		= \trunc{\Delta}^n_q - q^n([\Delta]^{n}_q)
		= \Delta - q^n(\symb^n_{d,E,F}(\Delta))
		= \Delta - q(\symb^n_{d,E,F}(\Delta)).
	\end{equation}
	Therefore, $\Delta=\trunc{\Delta}^{n-1}_q+q(\symb^n_{d,E,F}(\Delta))$, which by Proposition \ref{prop:truncation_well_defined}.\eqref{prop:truncation_well_defined:3} and the reverse implication yields
	\begin{equation}
	q\circ\symb_q(\Delta)
	=q\circ\symb_q(\trunc{\Delta}^{n-1}_q+q(\symb^n_{d,E,F}(\Delta)))
	=q\circ\symb_q(\trunc{\Delta}^{n-1}_q)+q(\symb_q(q(\symb^n_{d,E,F}(\Delta))))
	=\trunc{\Delta}^{n-1}_q+q(\symb^n_{d,E,F}(\Delta))
	=\Delta.
	\end{equation}
	Completing this portion of the proof.
	Finally, we prove \eqref{eq:DO_decomposition} via the following explicit computation
	\begin{equation}
		\Delta
		= q\circ \symb_q(\Delta)
		= \sum_{j\in\N} q^j\left(\sum_{k\in\N}[\Delta]^k_q \right)
		= \sum_{k\in\N} q^k[\Delta]^k_q.
	\end{equation}
\end{proof}
\begin{defi}
\label{defi:homogeneous_component_DO}
Given $\Delta\in\Diff_d(E,F)$, we term $\Delta^{(k)}_q\colonequals q^k([\Delta]^k_q)= q^k(\symb^k_{d,E,F}(\trunc{\Delta}^k_q))$ the \emph{$k$-homogeneous component} of $\Delta$ with respect to $q$.
\end{defi}
\begin{rmk}
As a classical example, to clarify the terminology, consider $\R^n$, with the canonical coordinates $x=(x_1,\dots,x_n)$.
It comes equipped with a canonical flat torsion-free affine connection for which the forms $dx_i$ are parallel.
This gives a quantization $q$.
Then, if $\Delta=\sum_j\sum_{i_1=1}^n\dots \sum_{i_j=1}^n a_{i_1,\dots,i_j}\partial_{x_{i_1}}\circ\dots\circ\partial_{x_{i_j}}$, we have $\Delta^{(k)}_q=\sum_{i_1=1}^n\dots \sum_{i_k=1}^n a_{i_1,\dots,i_k}\partial_{x_{i_1}}\circ\dots\circ\partial_{x_{i_k}}$.
\end{rmk}
\begin{rmk}
Using this notation, together with Proposition \ref{prop:isomorphism_quantization_total_symbol}, one can show that
\begin{equation}
\label{eq:truncation_decomposition}
	\trunc{\Delta}^m_q
	= \sum_{k=0}^m \Delta^{(k)}_q.
\end{equation}
\end{rmk}
\subsection{Star products}
\label{ss:starproduct}
In this section we show that, for a given $E$ in $\AMod$, a full quantization $q$ for $(E,E)$ induces a deformed algebra product on the symbol algebra $\Symb^\bullet_d(E,E)$.
Classically, the symbol algebra for differential operators on $\smooth{M}$ is identified with a dense subalgebra of $\smooth{T^\ast M}$.
Under this identification, this product is the star, or Moyal, product (cf.\ \cite{originalstarproduct}), as described in §\ref{s:introduction}.
Further, one can also equip the symbol algebra of differential operators on sections of a generic vector bundle with a deformed product, which can be shown to be related to the quantum mechanics of particles with inner structure on the base manifold, cf.\ \cite[Chapter~4]{Lychagin1999}.

\subsubsection{Polynomials with module coefficients}
\label{sss:Polynomials_with_module_coefficients}
Before discussing the star product, we will briefly recall some general algebraic facts that will be used in this section.
Recall that given $M$ in $\Mod$, one can use the extension of scalars given by the inclusion $\bk\hookrightarrow \bk[h]$ to produce a module
\begin{equation}
M[h]
\colonequals M\otimes \bk[h]
\cong
\bk[h]\otimes M.
\end{equation}
We have a split monomorphism mapping $m\in M\mapsto m\otimes 1\in M[h]$.
We can thus identify $m\otimes 1$ with $m$, interpreting an element $m\otimes p(h)$ as $m\cdot p(h)$, and thus viewing the elements in $M[h]$ as polynomials with coefficients in $M$ in a central formal variable $h$.
The adjunction of extension-restriction of scalars ensures the following universal property, cf.\ \cite[Theorem~8, p.~362]{dummit2004abstract}.
\begin{prop}[Universal property of the extension of scalars]
\label{prop:universal_prop_extension_of_scalars}
Given $M$ in $\Mod$, $N$ in ${}_{\bk[h]}\!\Mod$, and a $\bk$-linear map $\phi\colon M\to N$, there exists a unique $\bk[h]$-linear extension of $\phi$ to $M[h]$, i.e.\ a map $\hat{\phi}\colon M[h]\to N$ restricting to $\phi$ on $M$.
\end{prop}
Recall that the extension of scalars a monoidal functor, as
\begin{equation}
(M\otimes M)\otimes \bk[h]
\cong M\otimes (M\otimes \bk[h])
\cong M\otimes M[h]
\cong M\otimes \bk[h] \otimes_{\bk[h]} M[h]
\cong M[h] \otimes_{\bk[h]} M[h].
\end{equation}
As a consequence of Proposition \ref{prop:universal_prop_extension_of_scalars}, given $N$ in ${}_{\bk[h]}\!\Mod$, a $\bk$-bilinear map $\varphi\colon M\otimes M\to N$ extends uniquely to a $\bk[h]$-bilinear map $\hat{\varphi}\colon M[h] \otimes_{\bk[h]} M[h]\to N$.
For example, given $E$ in $\AMod$, the composition of differential operators of finite order lifts to a $\bk[h]$-linear map
\begin{equation}
\circhat 
\colon \Diff_d(E,E)[h]\otimes_{\bk[h]} \Diff_d(E,E)[h]\longrightarrow \Diff_d(E,E)[h].
\end{equation}
Notice that the evaluation map $\ev_\hbar \colon \bk[h]\twoheadrightarrow \bk$ at $\hbar\in\bk$, is a $\bk[h]$-algebra epimorphism.
Therefore, by tensoring via $M$ in $\Mod$, it induces a $\bk[h]$-module epimorphism which we will denote with the same name $\ev_\hbar\colon M[h]\twoheadrightarrow M$, where $h$ acts on $M$ via multiplication by the scalar $\hbar$.
Recall that if $M=R$ is a unital associative $\bk$-algebra, the componentwise multiplication endows $R[h]\colonequals R\otimes \bk[h]$ with a $\bk$-algebra structure.
\subsubsection{Formal star product}
\label{sss:Formal_star_product}
From this point onwards, we will assume that we are given $E$ in $\AMod$ and a full quantization $q$ for $(E,E)$.
Before defining the star product on $\Symb^\bullet_d(E,E)$, we will define one on an algebra of formal polynomials in a central variable $h$.
\begin{defi}
\label{defi:formal_star_product}
We define the following map by $\bk$-linear extension via
\begin{align}
	\label{eq:formal_star_product}
	\star \colon \Symb^\bullet_d(E,E)\otimes \Symb^\bullet_d(E,E)\longrightarrow \Symb^\bullet_d(E,E)[h],
	&\hfill&
	a\star b\colonequals \sum_{k=0}^{n+m}h^k[q^n(a) \circ q^m(b)]^{n+m-k}_q,
\end{align}
where $a\in\Symb^n_d(E,E)$ and $b\in\Symb^m_d(E,E)$.
This map extends to an internal $\bk[h]$-bilinear operation $\starhat$ on $\Symb^\bullet_d(E,E)[h]$, cf.\ Proposition \ref{prop:universal_prop_extension_of_scalars}, which we term the \emph{formal star product} corresponding to $q$.

We call $(\Symb^\bullet_d(E,E)[h],\starhat )$ the \emph{formal $h$-deformed symbol algebra}.
\end{defi}

In order to study the properties of this operation, we also give the following map.
\begin{defi}
We define the following map, via the coproduct universal property, as
\begin{align}
\label{eq:definition__formal_deformed_quantization}
q_h
\colon \Symb^\bullet_d (E,E)\longrightarrow \Diff_d(E,E)[h],
&\hfill&
q_h
\colonequals \sum_{k\in\N} h^k q^k.
\end{align}
By Proposition \ref{prop:universal_prop_extension_of_scalars}, this map extends to a unique $\bk[h]$-linear map, which we call \emph{formal $h$-deformed quantization}
\begin{equation}
\hat{q}_h
\colon \Symb^\bullet_d (E,E)[h]\longrightarrow \Diff_d(E,E)[h].
\end{equation}
\end{defi}
\begin{lemma}
\label{lemma:properties_formal_quantization}
The formal $h$-deformed quantization $\hat{q}_h$ has the following properties:
\begin{enumerate}
\item\label{lemma:properties_formal_quantization:1} $\hat{q}_h$ is a monomorphism;
\item\label{lemma:properties_formal_quantization:2} $\hat{q}_h (a\starhat b)=\hat{q}_h(a)\circhat \hat{q}_h(b)$, or in other words, the following diagram of $\bk[h]$-linear maps commutes
\begin{equation}
\begin{tikzcd}
\Symb^\bullet_d (E,E)[h]\otimes_{\bk[h]} \Symb^\bullet_d (E,E)[h]\ar[d,"\hat{q}_h \otimes_{\bk[h]} \hat{q}_h"']\ar[r,"\starhat "]& \Symb^\bullet_d (E,E)[h]\ar[d,"\hat{q}_h"]\\
\Diff_d(E,E)[h]\otimes_{\bk[h]}\Diff_d(E,E)[h] \ar[r,"\circhat "]& \Diff_d(E,E)[h]
\end{tikzcd}
\end{equation}
\item\label{lemma:properties_formal_quantization:3} $\hat{q}_h (\id_E)=\id_E$.
\end{enumerate}
\end{lemma}
\begin{proof}\
\begin{enumerate}
\item We consider the extension of scalars of the total symbol map, i.e.\ the following map
\begin{equation}
\hat{\symb}_q
\colon \Diff_d(E,E)[h]\longrightarrow \Symb^\bullet_d(E,E)[h],
\end{equation}
which is an isomorphism by Proposition \ref{prop:isomorphism_quantization_total_symbol} and the functoriality of the extension of scalars.

We will complete the proof by showing that the composition $\hat{\symb}_q\circ \hat{q}_h$ is a monomorphism.
This composition maps an element $h^i\sigma\in h^i\Symb^j_d(E,E)$ to
\begin{equation}
\hat{\symb}_q\circ \hat{q}_h(h^i\sigma)
=h^i\hat{\symb}_q(q_h(\sigma))
=h^i\sum_{k\in\N} h^k \symb_q(q^k(\sigma))
=h^{i+j} \symb_q(q^j(\sigma))
=h^{i+j} \sum_{k\in\N} [q^j(\sigma)]^k_q
=h^{i+j} \sigma.
\end{equation}
We can see the space $\Symb^\bullet_d(E,E)[h]$ as a bigraded $\bk$-module, and the map $\hat{\symb}_q\circ \hat{q}_h$ maps the component $h^i\Symb^j_d(E,E)[h]$ into the component $h^{i+j}\Symb^j_d(E,E)[h]$ monomorphically.
Since two distinct components are sent to distinct ones, it follows that $\hat{\symb}_q\circ \hat{q}_h$ is a mono, and thus so is $\hat{q}_h$.
\item For this formula we proceed by direct computation.
By Proposition \ref{prop:universal_prop_extension_of_scalars}, it is enough to prove it on elements $a\in\Symb^n_d(E,E)$ and $b\in\Symb^m_d(E,E)$.
We obtain
\begin{equation}
\label{eq:formal_quantization_preserves_multiplication:1}
\begin{split}
	\hat{q}_h(a\starhat b)
	&=\hat{q}_h(a\star b)\\
	&=\hat{q}_h\left( \sum_{k=0}^{n+m} h^k[q^n(a) \circ q^m(b)]^{n+m-k}_q \right)\\
	&=\sum_{k=0}^{n+m} h^k q_h\left([q^n(a) \circ q^m(b)]^{n+m-k}_q \right)\\
	&=\sum_{k=0}^{n+m} h^k h^{n+m-k} q^{n+m-k}\left([q^n(a) \circ q^m(b)]^{n+m-k}_q \right)\\
	&= h^{n+m}\sum_{k=0}^{n+m} q^k\left( [q^n(a) \circ q^m(b)]^k_q \right)\\
	&= h^{n+m}\sum_{k=0}^{n+m} \left(q^n(a) \circ q^m(b) \right)^{(k)}_q,
\end{split}
\end{equation}
where the last two equalities follow from a reparametrization, and from the definition of homogeneous component, respectively.
Similarly, for the other term we have
\begin{equation}
\label{eq:formal_quantization_preserves_multiplication:2}
\begin{split}
	\hat{q}_h(a)\circhat \hat{q}_h (b)
	&=q_h(a)\circhat q_h b)\\
	&=\left(\sum_{i\in \N} h^i q^i(a)\right)\circhat \left(\sum_{j\in \N} h^j q^j(b)\right)\\
	&=h^n q^n(a)\circhat h^m q^m(b)\\
	&=h^{n+m} q^n(a)\circ q^m(b).
\end{split}
\end{equation}
The expressions \eqref{eq:formal_quantization_preserves_multiplication:1} and \eqref{eq:formal_quantization_preserves_multiplication:2} are seen to be equal by applying \eqref{eq:DO_decomposition} to $q^n(a)\circ q^m(b)\in\Diff^{n+m}_d(E,E)$.
\item It follows from direct computation
\begin{equation}
q_h(\id_E)
=\sum_{k\in\N} h^k q^k(\id_E)
=q^0(\id_E)
=\id_E.
\end{equation}
\qedhere
\end{enumerate}
\end{proof}
\begin{prop}
\label{prop:formal_star_algebra}
The $\bk[h]$-module $\Symb^\bullet_d(E,E)[h]$ equipped with the star product $\starhat$ forms a unital associative filtered $\bk[h]$-algebra where the filtration is given by the partial sums
\begin{equation}
\bigoplus_{k=0}^n \Symb^k_d(E,E)[h].
\end{equation}
In particular, for all $a\in\Symb^n_d(E,E)$ and $b\in\Symb^m_d(E,E)$, we have
\begin{equation}
\label{eq:first_order_approximation_formal_star}
a\starhat b - a\cdot b
\in h \bigoplus_{k=0}^{m+n-1} \Symb^k_d(E,E)[h].
\end{equation}
Thus, $\starhat$ agrees with $\cdot$, the symbol multiplication (cf.\ \cite[\symbolsdefisymbolalgebra]{Symbol}), up to order $0$ in $h$, i.e.\ $a\starhat b=ab+O(h)$.
\end{prop}
\begin{proof}
Associativity follows from the associativity of $\circhat$.
By Lemma \ref{lemma:properties_formal_quantization}.\eqref{lemma:properties_formal_quantization:2}, we have that for all $a\in\Symb^n_d(E,E)$, $b\in\Symb^m_d(E,E)$, and $c\in\Symb^l_d(E,E)$, we have
\begin{equation}
\label{eq:associativity}
\begin{split}
	\hat{q}_h((a\starhat b)\starhat c)
	&=\hat{q}_h((a\starhat b))\circhat \hat{q}_h(c)\\
	&=(\hat{q}_h(a)\circhat \hat{q}_h(b))\circhat \hat{q}_h(c)\\
	&=\hat{q}_h(a)\circhat (\hat{q}_h(b)\circhat \hat{q}_h(c))\\
	&=\hat{q}_h(a)\circhat \hat{q}_h(b\starhat c)\\
	&=\hat{q}_h(a\starhat (b\starhat c)).
\end{split}
\end{equation}
By Lemma \ref{lemma:properties_formal_quantization}.\eqref{lemma:properties_formal_quantization:1}, we conclude that $(a\starhat b)\starhat c=a\starhat (b\starhat c)$.
Full associativity of $\starhat$ follows by $\bk[k]$-bilinearity.

We now show that $\id_E\in\Symb^0_d(E,E)=\AHom(E,E)$ is the unit of $\starhat$.
First, notice that by Lemma \ref{lemma:properties_formal_quantization}, for all $a\in\Symb^n_d(E,E)$, we have
\begin{equation}
q_h(\id_E\starhat a)
=q_h(\id_E)\circhat q_h(a)
=\id_E\circ q^h(a)
=q^h(a).
\end{equation}
By Lemma \ref{lemma:properties_formal_quantization}.\eqref{lemma:properties_formal_quantization:1}, we deduce that $\id_E\starhat a=a$.
Analogously one proves also $a\starhat \id_E=a$, and the unitality on the whole of $\Symb^\bullet_d(E,E)[h]$ follows again by $\bk[h]$-bilinearity of $\starhat$.

In order to prove that $\starhat$ preserves the filtration it is enough to prove directly \eqref{eq:first_order_approximation_formal_star}.
This follows from the definition of $\starhat$, as we have
\begin{equation}
\begin{split}
	a\star b
	&= \sum_{k=0}^{n+m}h^k[q^n(a) \circ q^m(b)]^{n+m-k}_q \\
	&= [q^n(a) \circ q^m(b)]^{n+m}_q + \sum_{k=1}^{n+m}h^k[q^n(a) \circ q^m(b)]^{n+m-k}_q\\
	&= \symb^{n+m}_d(q^n(a) \circ q^m(b))+ h\left(\sum_{k=1}^{n+m}h^{k-1}[q^n(a) \circ q^m(b)]^{n+m-k}_q \right)\\
	&= a\cdot b+ h\left(\sum_{k=1}^{n+m}h^{k-1}[q^n(a) \circ q^m(b)]^{n+m-k}_q \right),
\end{split}
\end{equation}
where the last equality follows from the definition of symbol multiplication, cf.\ \cite[\symbolspropsymbolcomposition]{Symbol}, and the definition of quantization for $(E,E)$.
\end{proof}
\begin{rmk}
\label{rmk:formal_deformed_quantisation_is_morphism}
Lemma \ref{lemma:properties_formal_quantization} makes the map $\hat{q}_h\colon \left(\Symb^\bullet_d(E,E)[h],\starhat \right)\hookrightarrow \left(\Diff_d(E,E)[h],\circhat \right)$ into a unital monomorphism of $\bk[h]$-algebras.
This map is in general not epi (in particular not iso).
In fact, if we take a differential operator $\Delta$ of order exactly $1$, the element $\Delta=h^0 \Delta$ is not in the image of $q_h$.
\end{rmk}
\subsubsection{Parametrized star product}
\label{sss:Parametrized_star_product}
Consider a parameter $\hbar\in\bk$.
As mentioned in §\ref{sss:Polynomials_with_module_coefficients}, there exists a unique $\bk[h]$-linear epimorphism $\ev_{\hbar}\colon\Symb^\bullet_d(E,E)[h]\twoheadrightarrow \Symb^\bullet_d(E,E)$ mapping $h$ to $\hbar$.
This epimorphism induces an operation on $\Symb^\bullet_d(E,E)$ given as follows.
\begin{defi}\label{defi:parametrized_star_product}
The \emph{star product} $\star=\star_\hbar$ corresponding to $q$ for a parameter $\hbar$ is defined via the composition
\begin{equation}
	\label{eq:star_product}
	a \star b
	\colonequals \ev_{\hbar}(a \starhat b)
	= \sum_{k=0}^{n+m}\hbar^k[q^n(a) \circ q^m(b)]^{n+m-k}_q,
\end{equation}
where $a\in\Symb^n_d(E,E)$ and $b\in\Symb^m_d(E,E)$.
We call $(\Symb^\bullet_d(E,E),\star)$ the \emph{$\hbar$-deformed symbol algebra}.
\end{defi}
\begin{prop}
\label{prop:star_algebra}
The $\bk$-module $\Symb^\bullet_d(E,E)$ equipped with $\star$ as in \eqref{eq:star_product} forms a unital associative filtered $\bk$-algebra isomorphic to
\begin{equation}
\Symb^\bullet_d(E,E)[h]/(h-\hbar),
\end{equation}
where the filtration is inherited from the grading of $\Symb^\bullet_d(E,E)[h]$, hence given by the partial sums
\begin{equation}
\bigoplus_{k=0}^n \Symb^k_d(E,E).
\end{equation}
In particular, for all $a\in\Symb^n_d(E,E)$ and $b\in\Symb^m_d(E,E)$, we have
\begin{equation}
\label{eq:first_order_approximation_star}
a\star b -a\cdot b
\in\hbar \bigoplus_{k=0}^{m+n-1} \Symb^k_d(E,E).
\end{equation}
\end{prop}
\begin{proof}
The evaluation at $\hbar$ induces an isomorphism in ${}_{\bk[h]}\!\Mod$:
\begin{equation}
\Symb^\bullet_d(E,E)[h]/\ker(ev_\hbar)
\cong \Symb^\bullet_d(E,E).
\end{equation}
We show that $\ker(\ev_\hbar)$ is $(h-\hbar)$, i.e.\ the two-sided ideal generated by $h-\hbar$.
Since $\ev_\hbar (h-\hbar)=0$, we only have to prove that elements of the kernel are contained in this ideal.
Given an element $p(h)=\sum_{k\in\N}\sigma_k h^k\in\ker(\ev_\hbar)$, we have
\begin{equation}
p(h)
=p(h)-p(\hbar)
=\sum_{k\in\N}\sigma_k (h^k-\hbar^k)
=\sum_{k\in\N}\sigma_k (h^{k-1}+ h^{k-2}\hbar+\dots+\hbar^{k-1})\starhat (h-\hbar)
\in (h-\hbar).
\end{equation}
The isomorphism induces an algebra structure on $\Symb^\bullet_d(E,E)$, and the product $\starhat$ is mapped by the isomorphism (induced by $\ev_\hbar$) to the star product with parameter $\hbar$.
This turns $(\Symb^\bullet_d(E,E),\star)$ into an associative unital $\bk[h]$-algebra (hence $\bk$-algebra).
Applying $\ev_\hbar$ to \eqref{eq:first_order_approximation_formal_star}, yields \eqref{eq:first_order_approximation_star}, and hence the desired filtration.
\end{proof}
\begin{rmk}
\label{rmk:evaluation_algebra_homomorphism}
Proposition \ref{prop:star_algebra} shows that the evaluation map at the parameter $\hbar$ realizes a unital epimorphism of $\bk[h]$-algebras $\ev_\hbar\colon \left(\Symb^\bullet_d(E,E)[h], \starhat \right)\twoheadrightarrow \left(\Symb^\bullet_d(E,E),\star\right)$.
\end{rmk}
If we now consider the evaluation $\ev_\hbar\colon \Diff_d(E,E)[h]\twoheadrightarrow \Diff_d(E,E)$, we construct the following map.
\begin{defi}
We define the \emph{$\hbar$-deformed quantization} as the following composition:
\begin{align}
\label{eq:definition_deformed_quantization}
q_\hbar
\colonequals \ev_\hbar\circ q_h
\colon \Symb^\bullet_d (E,E)\longrightarrow \Diff_d(E,E),
&\hfill&
q_\hbar
\colonequals \sum_{k\in\N} \hbar^k q^k.
\end{align}
\end{defi}
\begin{rmk}
In some treatments of quantization, the deformed quantization is sometimes termed \emph{semiclassical quantization}, (cf.\ \cite[Chapter~4]{semiclassical} or \cite[Definition~5.43, p.~62]{Hintz}).
\end{rmk}
\begin{rmk}
\label{rmk:compatibility_quantizations_evaluations}
By definition, the following square of $\bk[h]$-linear (and hence $\bk$-linear) maps commutes.
\begin{equation}
\begin{tikzcd}
\Symb^\bullet_d (E,E)[h]\ar[r,"\hat{q}_h"]\ar[d,"\ev_{\hbar}"']& \Diff_d(E,E)[h]\ar[d,"\ev_{\hbar}"]\\
\Symb^\bullet_d (E,E)\ar[r,"q_{\hbar}"]& \Diff_d(E,E)
\end{tikzcd}
\end{equation}
\end{rmk}
\begin{prop}\
\label{prop:properties_deformed_quantization}
\begin{enumerate}
\item\label{prop:properties_deformed_quantization:1} The map $q_\hbar \colon (\Symb^\bullet_d(E,E),\star)\to (\Diff_d(E,E),\circ)$ is a filtered unital associative $\bk$-algebra morphism.
\item\label{prop:properties_deformed_quantization:2} The $0$-deformed quantization $q_0$ coincides with the projection to the $0$-grade
\begin{equation}
q_0\colon \Symb^\bullet_d(E,E) \longtwoheadrightarrow \AHom(E,E).
\end{equation}
\item\label{prop:properties_deformed_quantization:3} The $1$-deformed quantization is the quantization, i.e.\ $q_1=q$.
\end{enumerate}
\end{prop}
\begin{proof}
The point \eqref{prop:properties_deformed_quantization:1} follows because $q_{\hbar}$ is a composition of two $\bk$-algebra morphisms, cf.\ Remark \ref{rmk:formal_deformed_quantisation_is_morphism} and Remark \ref{rmk:evaluation_algebra_homomorphism}.
Points \eqref{prop:properties_deformed_quantization:2} and \eqref{prop:properties_deformed_quantization:3} follow from direct computation.
\end{proof}
\begin{prop}\
\label{prop:special_deformations}
\begin{enumerate}
\item\label{prop:special_deformations:1} The $0$-deformed symbol algebra is just the symbol algebra, i.e.\ $(\Symb^\bullet_d(E,E),\star_0) = (\Symb^\bullet_d(E,E),\cdot)$, where $\cdot$ is the symbol multiplication (cf.\ \cite[\symbolsdefisymbolalgebra]{Symbol}).
\item\label{prop:special_deformations:2} The $1$-deformed symbol algebra is isomorphic to the algebra of linear differential operators of finite order, i.e.\ $(\Symb^\bullet_d(E,E),\star_1)\cong (\Diff_d(E,E),\circ)$.
\end{enumerate}
\end{prop}
\begin{proof}\
\begin{enumerate}
\item Proposition \ref{prop:star_algebra} yields \eqref{prop:special_deformations:1} by substituting $\hbar=0$ in \eqref{eq:first_order_approximation_star}.
\item We consider the map $q_1$.
By Proposition \ref{prop:properties_deformed_quantization}.\eqref{prop:properties_deformed_quantization:1}, $q_1$ is a unital homomorphism of $\bk$-algebras, and by Proposition \ref{prop:properties_deformed_quantization}.\eqref{prop:properties_deformed_quantization:3}, we know that $q_1=q$.
By Proposition \ref{prop:isomorphism_quantization_total_symbol}, we know this map is invertible, with inverse $\symb_q$, thus giving the desired isomorphism.\qedhere
\end{enumerate}
\end{proof}
\begin{rmk}
Let $\bk$ be a topological unital commutative ring, and let $\Symb^\bullet_d(E,E)$ be a topological $\bk$-module, then the map $\hbar\in\bk\mapsto a\star_{\hbar} b$ is polynomial with symbol coefficients, and thus obtained by ring operations of $\bk$ and module operations of $\Symb^\bullet_d(E,E)$, and as such it is continuous for all given $a,b\in\Symb^\bullet_d(E,E)$.
This means that $\lim\limits_{\hbar\to 0} a\star_\hbar b=a\star_0 b= a\cdot b$.
Moreover, we have $a\star b=a\cdot b + \mathcal{O}(\hbar)$ whenever $\bk$ and $\Symb^\bullet_d(E,E)$ have enough structure to yield a notion of $\mathcal{O}$.
\end{rmk}
\begin{rmk}
The parametrized star product of Definition \ref{defi:parametrized_star_product} can also be extended to the case where $\hbar$ is a central element of $(\Symb^{\bullet}_d(E,E),\cdot)$.
In that case, $\star$ also yields an associative algebra structure, however the algebra is no longer filtered, unless the degree of the symbol is $0$.
\end{rmk}
\subsubsection{Deformed total symbol}
\label{sss:Deformed total symbol}
We will now assume that the formal variable $h$ is invertible and we consider the space $\bk[h,h^{-1}]$ of finite Laurent series, i.e.\ polynomial expressions with coefficients in $\bk$ where the variable $h$ is allowed to appear with negative powers, or more formally the localization of $\bk[h]$ at the ideal $(h)$.

Given $M$ in $\Mod$, we can localize $M[h]$ at the ideal $(h)$, or equivalently extend the coefficients of $M$ to $\bk[h,h^{-1}]$, obtaining a module which we term $M[h,h^{-1}]$.
We interpret this operation as extending $M$ by a formal invertible central variable $h$.

By applying the localization functor by the ideal $(h)$ to the objects and morphisms (in ${}_{\bk[h]}\!\Mod$) of §\ref{sss:Formal_star_product}, we can obtain the following $\bk[h,h^{-1}]$-linear maps, which we will term and write as their $\bk[h]$-linear counterparts: the formal deformed quantization $\hat{q}_h \colonequals \sum_{k\in\N} h^k q^k
\colon \Symb^\bullet_d (E,E)[h,h^{-1}]\longrightarrow \Diff_d(E,E)[h,h^{-1}]$ and the formal star product
$\starhat \colon \Symb^\bullet_d(E,E)[h,h^{-1}]\otimes_{\bk[h,h^{-1}]} \Symb^\bullet_d(E,E)[h,h^{-1}]\longrightarrow \Symb^\bullet_d(E,E)[h,h^{-1}]$, such that
\begin{equation}
a\starhat b\colonequals \sum_{k=0}^{n+m}h^k[q^n(a) \circ q^m(b)]^{n+m-k}_q ,
\end{equation}
for $a\in\Symb^n_d(E,E)$ and $b\in\Symb^m_d(E,E)$.

Further, we can now define an inverse for the formal deformed quantization as follows.
\begin{defi}
We define the \emph{formal deformed total symbol} as the following map
\begin{align}\label{eq:formal_deformed_total_symbol}
\hat{\symb}_{q_h}\colon \Diff_d(E,E)[h,h^{-1}]\longrightarrow \Symb^\bullet_d(E,E)[h,h^{-1}],
&\hfill&
\hat{\symb}_{q_h}(\Delta)\colonequals \sum_{k\in\N}h^{-k}[\Delta]^k_q.
\end{align}
\end{defi}
\begin{prop}\
\label{prop:properties_formal_deformed_total_symbol}
\begin{enumerate}
\item\label{prop:properties_formal_deformed_total_symbol:1} $\hat{\symb}_{q_h}$ is the inverse of $\hat{q}_h$.
\item\label{prop:properties_formal_deformed_total_symbol:2} $a\starhat b= \hat{\symb}_{q_h}(q_h(a)\circhat q_h(b))$ for all $a,b\in\Symb^\bullet_d(E,E)[h,h^{-1}]$.
\end{enumerate}
\end{prop}
\begin{proof}\
\begin{enumerate}
\item By direct computation on $\sigma\in\Symb^n_d(E,E)$, we obtain
\begin{equation}
\hat{\symb}_{q_h}\circ \hat{q}_h (\sigma)
=\hat{\symb}_{q_h}\left( \sum_{k\in \N} h^k q^k(\sigma)\right)
=h^n \hat{\symb}_{q_h}(q^n(\sigma))
=h^n \sum_{k\in\N}h^{-k}[(q^n(\sigma))]^k_q
=\sigma,
\end{equation}
where the last equality follows from Lemma \ref{lemma:total_symbols_of_pure_quantization}.
Similarly, for all $\Delta\in \Diff^n_d(E,E)$ we obtain
\begin{equation}
\hat{q}_h \circ \hat{\symb}_{q_h} (\Delta)
=\hat{q}_h \left(\sum_{k\in\N}h^{-k}[\Delta]^k_q\right)
=\sum_{k\in \N} h^{-k}\hat{q}_h ([\Delta]^k_q)
=\sum_{k\in \N} h^{-k}\sum_{j\in\N}h^j q^j([\Delta]^k_q)
=\sum_{k\in \N} q^k([\Delta]^k_q)
=\Delta
\end{equation}
where the last equality is \eqref{eq:DO_decomposition} from Proposition \ref{prop:isomorphism_quantization_total_symbol}.
\item It follows from \eqref{prop:properties_formal_deformed_total_symbol:1} by applying the map $\hat{\symb}_{q_h}$ to both terms of the equality $\hat{q}_h (a\starhat b)=\hat{q}_h(a)\circhat \hat{q}_h(b)$ for $a\in\Symb^n_d(E,E)$ and $b\in\Symb^m_d(E,E)$, cf.\ Lemma \ref{lemma:properties_formal_quantization}.\eqref{lemma:properties_formal_quantization:2}, and extending the result to all elements of $\Symb^\bullet_d(E,E)[h,h^{-1}]$.\qedhere
\end{enumerate}
\end{proof}

Given $\hbar\in \bk^{\times}$, i.e.\ a unit in $\bk$, we can uniquely extend the evaluation map to $\ev_{\hbar}\colon \bk[h,h^{-1}]\to \bk$.
More generally, tensoring it by $M$ in $\Mod$, we obtain a map $\ev_{\hbar}\colon M[h,h^{-1}]\to M$, which essentially fixes $M$ and maps $h$ to $\hbar$.
We can now construct the following map.
\begin{defi}
We define the \emph{deformed total symbol} as the following:
\begin{align}\label{eq:deformed_total_symbol}
\symb_{q_{\hbar}}\colon \Diff_d(E,E)\longrightarrow \Symb^\bullet_d(E,E),
&\hfill&
\symb_{q_{\hbar}}(\Delta)
\colonequals\ev_{\hbar}\circ \hat{\symb}_{q_h}(\Delta)
= \sum_{k\in\N}\hbar^{-k}[\Delta]^k_q.
\end{align}
\end{defi}
\begin{rmk}
In the classical setting, the map corresponding to \eqref{eq:deformed_total_symbol} is sometimes termed the \emph{Hamiltonian map} (cf.\ \cite[Section~4]{Lychagin1999}).
\end{rmk}
\begin{rmk}
\label{rmk:compatibility_quantizations_symbols_evaluations}
By definition, the following squares of $\bk[h,h^{-1}]$-linear (and hence $\bk$-linear) maps commute.
\begin{equation}
\label{diag:compatibility_quantizations_symbols_evaluations}
\begin{tikzcd}
\Symb^\bullet_d (E,E)[h,h^{-1}]\ar[r,"\hat{q}_h"]\ar[d,"\ev_{\hbar}"']& \Diff_d(E,E)[h,h^{-1}]\ar[d,"\ev_{\hbar}"]& \Diff_d(E,E)[h,h^{-1}]\ar[r,"\hat{\symb}_{\hat{q}_h}"]\ar[d,"\ev_{\hbar}"']& \Symb^\bullet_d (E,E)[h,h^{-1}]\ar[d,"\ev_{\hbar}"]\\
\Symb^\bullet_d (E,E)\ar[r,"q_{\hbar}"]& \Diff_d(E,E)&\Diff_d(E,E)\ar[r,"\symb_{q_{\hbar}}"]& \Symb^\bullet_d (E,E)
\end{tikzcd}
\end{equation}
\end{rmk}
We can prove the analogue of Proposition \ref{prop:properties_formal_deformed_total_symbol} in the following proposition.
\begin{prop}\
\label{prop:properties_deformed_total_symbol}
\begin{enumerate}
\item $\symb_{q_{\hbar}}$ is the inverse of $q_{\hbar}$.
\item $a \star b= \symb_{q_{\hbar}}(q_{\hbar}(a)\circ q_{\hbar}(b))$ for all $a,b\in\Symb^\bullet_d(E,E)$.
\end{enumerate}
\end{prop}
\begin{proof}\
\begin{enumerate}
\item Consider the following diagram obtained by the composition of the squares in \eqref{diag:compatibility_quantizations_symbols_evaluations}
\begin{equation}
\begin{tikzcd}
\Symb^\bullet_d (E,E)[h,h^{-1}]\ar[rr,bend left=10pt,equals]\ar[r,"\hat{q}_h"']\ar[d,"\ev_{\hbar}"']& \Diff_d(E,E)[h,h^{-1}]\ar[r,"\hat{\symb}_{\hat{q}_h}"']\ar[d,"\ev_{\hbar}"]& \Symb^\bullet_d (E,E)[h,h^{-1}]\ar[d,"\ev_{\hbar}"]\\
\Symb^\bullet_d (E,E)\ar[r,"q_{\hbar}"]& \Diff_d(E,E)\ar[r,"\symb_{q_{\hbar}}"]& \Symb^\bullet_d (E,E)
\end{tikzcd}
\end{equation}
The two squares commute by Remark \ref{rmk:compatibility_quantizations_symbols_evaluations}, and the top triangle commutes by Proposition \ref{prop:properties_formal_deformed_total_symbol}.\eqref{prop:properties_formal_deformed_total_symbol:1}.
Since $\ev_{\hbar}$ is epi, we deduce $\symb_{q_{\hbar}}\circ q_{\hbar}=\id_{\Symb^\bullet_d(E,E)}$.
By considering the composition of the same squares in reverse order, we also prove that $q_{\hbar}\circ \symb_{q_{\hbar}}=\id_{\Diff_d(E,E)}$.
\item This point follows analogously via Remark \ref{rmk:compatibility_quantizations_symbols_evaluations} and Remark \ref{rmk:evaluation_algebra_homomorphism}, together with Proposition \ref{prop:properties_formal_deformed_total_symbol}.\eqref{prop:properties_formal_deformed_total_symbol:2}.\qedhere
\end{enumerate}
\end{proof}
\subsection{Quantization on the quaternions}
Consider the $\R$-algebra $\Hq$ of quaternions, equipped with the maximal exterior algebra generated by the $\{i,j\}$-terminal first order differential calculus, cf.\ \cite[\jetsdefiSterminal]{FMW}.
This has structure equation
\begin{equation}
\label{eq:quaternions_structure_equation}
	dk
	= -j di + i dj.
\end{equation}
The jet modules $J^n_d \mathbb{H}$ and the algebra $\DO_d\colonequals \Diff_d(\Hq,\Hq)$ were computed in \cite[\jetsssquaternions]{FMW}.
It was also shown that there exists a unique quantum metric, up to real scale, i.e.
\begin{equation}
\label{eq:quatmetric}
	g
	= di \otimes_\Hq dj - dj \otimes_\Hq di.
\end{equation}
See \cite[Chapters~1,~8]{BeggsMajid} for details about quantum metrics.

In order to demonstrate the quantization procedure, we will utilize a bimodule connection on $\Omega^1_d$.
A Levi-Civita connection for the metric \eqref{eq:quatmetric}, which in the noncommutative context need neither exist nor be unique, would be a natural choice.
\begin{prop}
\label{prop:quaternionlevicivita}
	There is a unique bimodule connection $\nabla$ on $\Omega^1_d$.
	We denote its corresponding generalized braiding by $\sigma$.
	The connection $\nabla$ is torsion free and metric, so it is a (in fact, the only) Levi-Civita connection for $g$.
	The curvature of this connection vanishes.
\end{prop}
\begin{proof}
	We parametrize the set of connections on the parallelizable calculus $\Omega^1_d = {}_\Hq \langle di, dj \rangle$ as
	\begin{equation}
	\begin{split}
		&\nabla di = \alpha_{ii} di \otimes di + \alpha_{ij} di \otimes dj + \alpha_{ji} dj \otimes di + \alpha_{jj} dj \otimes dj,\\
		&\nabla dj = \beta_{ii} di \otimes di + \beta_{ij} di \otimes dj + \beta_{ji} dj \otimes di + \beta_{jj} dj \otimes dj.
		\end{split}
	\end{equation}
	The connection is a bimodule connection if and only if the generalized braiding, given, for $\theta, dm\in\Omega^1_d$, by
	\begin{equation}
		\sigma(\theta \otimes_\Hq dm)
		= dm \otimes_\Hq \theta + \nabla[\theta,m] -[\nabla \theta, m]
	\end{equation}
	is well-defined as a bilinear map, cf.\ \cite[p.~568]{BeggsMajid}.
	Using the structure equation \eqref{eq:quaternions_structure_equation}, we obtain
	\begin{equation}
		di \otimes_{\Hq} dk
		= j di \otimes_{\Hq} di - i di \otimes_{\Hq} dj.
	\end{equation}
	Applying $\sigma$ to both sides of this equality 	yields equations that the coefficients must satisfy.
	The unique solution is that all coefficients vanish, i.e.\ $\alpha_{ii} = \dots = \alpha_{jj} = \dots = \beta_{jj} = 0$.
	This means that $\nabla$ is the Grassmann connection for the frame $\{di, dj\}$.
	We deduce that the braiding has to be given by extending the following by $\Hq$-bilinearity
	\begin{align}
		\label{eq:quaternionbraiding}
		\sigma\colon \Omega^1_d \otimes_\Hq \Omega^1_d \longrightarrow \Omega^1_d \otimes_\Hq \Omega^1_d,
		&\hfill&
		\omega \otimes_\Hq \nu \longmapsto - \nu \otimes_\Hq \omega, \qquad \text{ for } \omega, \nu \in \{di, dj\}.
	\end{align}
	We compute the torsion of $\nabla$, cf.\ \cite[Definition~3.28, p.~224]{BeggsMajid},
	\begin{equation}
		T_\nabla(a di + b dj)
		= da \wedge di + \wedge(a \nabla di) -da \wedge di + db \wedge dj + \wedge(a \nabla dj) -db \wedge dj
		= 0.
	\end{equation}
	Finally, we compute the covariant derivative of $g$.
	\begin{equation}
		\nabla(g)
		= (\nabla \otimes \id) g + (\sigma \otimes \id) (\id \otimes \nabla) g
		= 0
	\end{equation}
	This vanishes since $g$ is a linear combination of tensor products of parallel forms.
	Hence, $\nabla$ is a Levi-Civita connection for $g$, and it is unique.
	Since the curvature is $\Hq$-linear, and the calculus is parallelizable, it is enough to compute it on the frame $\{di, dj\}$.
	We have $R_\nabla(di)=d_\nabla(\nabla(di))=d_\nabla(0)=0$, and similarly $R_\nabla(dj)=0$, whence $R_\nabla=0$.
\end{proof}
The bimodule connection $\nabla$ from Proposition \ref{prop:quaternionlevicivita}, together with its corresponding braiding $\sigma$, give rise to a full quantization as follows.
\begin{lemma}
\label{lemma:quantization_quaternions}
For all $n\in \N$, $J^n_d \Hq=\EJ^n_d \Hq$ and the $n$-jet sequence is exact.
Moreover, there exists a full quantization $q$ for $(\Hq,\Hq)$ induced by the bimodule connection $\nabla$ and
\begin{equation}
\Sym^{1,1} \colonequals \tfrac{1}{2}(\id_{\Omega^1_d\otimes_\Hq \Omega^1_d} + \sigma)\colon\Omega^1_d\otimes_\Hq \Omega^1_d\to \Omega^1_d\otimes_\Hq \Omega^1_d.
\end{equation}
The nonzero components of the quantization are as follows, where $\nabla^2(h) = -\Re(kh) g$:
\begin{align}
\label{eq:quantization_quaternions}
q^0=\id_{{}_{\Hq}\!\Hom(\Hq,\Hq)},
&\hfill&
q^1(\sigma^1)=r^1_{d,E,S^2_d E}(\sigma^1) \circ d,
&\hfill&
q^2(\sigma^2)=r^2_{d,E,S^2_d E}(\sigma^2) \circ \nabla^2.
\end{align}
\end{lemma}
\begin{proof}
We always have $J^1_d \Hq=\EJ^1_d \Hq$ and $J^0_d \Hq=\EJ^0_d \Hq$, cf.\ \cite[\symbolspropelementaljetpropertiesone]{Symbol}.
In this example we chose the exterior algebra to be the maximal one, and thus $J^2_d \Hq=\EJ^2_d \Hq$, cf.\ \cite[\symbolspropelementaljetpropertiestwo]{Symbol}.
Moreover, this exterior algebra has vanishing Spencer $\delta$-cohomology.
Hence, the $n$-jet sequence at $\Hq$ is exact, cf.\ \cite[\jetscorspencerdeltajes]{FMW} and thus, for $n\ge 3$, we have $J^n_d \Hq=J^2_d \Hq=\EJ^2_d \Hq=\EJ^n_d\Hq$, cf.\ \cite[\jetsssquaternions]{FMW}.

We will produce a quantization using Corollary \ref{cor:generalization_of_classical_quantization} using $\Delta$ as a bimodule connection on $\Omega^1_d$, $\Sym^{1,1}$ as a retraction for $\iota^{n+1}_{\wedge}$, and $d\colon \Hq\to\Omega^1_d$ as connection on $\Hq$.
We are left to prove that $\Sym^{1,1}$ s a retraction for $\iota^2_{\wedge}$.
In fact, we have
\begin{equation}
\Sym^{1,1}\circ \iota^{n+1}_{\wedge}
=\tfrac{1}{2}(\id_{\Omega^1_d\otimes_\Hq \Omega^1_d} + \sigma)\circ \iota^{n+1}_{\wedge}
=\tfrac{1}{2}(\iota^{n+1}_{\wedge} + \sigma\circ \iota^{n+1}_{\wedge}).
\end{equation}
One shows this formula is the identity by checking it on an $\Hq$-generator for $S^2_d$, i.e.\ $g$, cf.\ \cite[\jetsssquaternions]{FMW}.
By \eqref{eq:explicit_quantization_with_connections_on_Sn}, the $n$-quantization is given by
\begin{equation}
q^n(\sigma^n)
=r^n_{d,E,S^n_d E}(\sigma^n)\circ \Sym^{1,n-1} \circ \nabla^{S^{n-1}_dE} \circ\Sym^{1,n-2} \circ \nabla^{S^{n-2}_dE} \circ\dots\circ \Sym^{1,1} \circ \nabla^{S^1_dE} \circ \nabla^E.
\end{equation}
The explicit formula \eqref{eq:quantization_quaternions} for $q^n$ for $n\neq 2$ is straightforward.
For $n=2$, we have
\begin{equation}
q^2(\sigma^2)
=r^2_{d,E,S^2_d E}(\sigma^2)\circ \Sym^{1,1} \circ \nabla^{\Omega^1_d} \circ d.
\end{equation}
We obtain \eqref{eq:quantization_quaternions} if we define $\nabla^2\colonequals \Sym^{1,1} \circ \nabla \circ d \colon \Hq\to \Omega^1_d$.
We compute the latter on a generic element $h=h_1+h_i i+h_j j+h_k k \in\Hq$.
\begin{equation}
\begin{split}
\nabla^2(h_1+h_i i+h_j j+h_k k)
&=\Sym^{1,1} (\nabla(h_i di+h_j dj+h_k dk))\\
&=h_k \Sym^{1,1} (\nabla(dk))\\
&=-\Re(kh) \Sym^{1,1} (\nabla( -j di + i dj))\\
&=-\Re(kh) \Sym^{1,1} (-dj\otimes_\Hq di + di\otimes_\Hq dj)\\
&=-\Re(kh) \Sym^{1,1} (g)\\
&=-\Re(kh) g.
\end{split}
\end{equation}
\end{proof}
It was shown in \cite[\jetsrmkLksecondorderDO]{FMW} that the left multiplication operator $L_k$ is a linear differential operator of order $2$.
By Proposition \ref{eq:DO_decomposition}, using the quantization $q$ of Lemma \ref{lemma:quantization_quaternions}, we can decompose $L_k$ as
\begin{equation}
	L_k
	= (L_k)^{(2)}_q + (L_k)^{(1)}_q + (L_k)^{(0)}_q.
\end{equation}
We compute the homogeneous components in the following proposition that shows that the highest homogeneous component is related to the Laplacian $\Delta$.
\begin{prop}
\label{prop:homogeneous_component_Lk}
	Given the hypotheses of Lemma \ref{lemma:quantization_quaternions}, the non-zero homogeneous components of $L_k$ are
	\begin{align}
	(L_k)^{(2)}_q = 2(\partial_i\circ \partial_j - \partial_j\circ \partial_i)=2\Delta,
	&\hfill&
		(L_k)^{(1)}_q = 2(R_j\circ \partial_i - R_i\circ \partial_j),
		&\hfill&
		(L_k)^{(0)}_q = R_k,
	\end{align}
	where $R_h(p)=ph$.
\end{prop}
\begin{proof}
	We compute $(L_k)^{(2)}_q =q^2(\symb^2_d(L_k))=r^2_{d,E,S^2_d E}(\symb^2_d(L_k)) \circ \nabla^2$.
	We evaluate it on a generic quaternion $h=h_1+h_i i+h_j j+h_k k \in\Hq$.
	\begin{equation}
		(L_k)^{(2)}_q(h)
		=r^2_{d,E,S^2_d E}(\symb^2_d(L_k)) \circ \nabla^2(h)
		=r^2_{d,E,S^2_d E}(\symb^2_d(L_k)) (h_k g)
		=\widetilde{L}_k\circ \iota^2_{d,\Hq} (h_k g).
	\end{equation}
	One can verify the following equality by expanding in $J^{(2)}_d\Hq$.
	\begin{equation}
		\iota^2_d(g)
		= j\cdot j^2_d(i) - i\cdot j^2_d(j) + j^2_d(k) + k\cdot j^2_d(1)
	\end{equation}
	Thus we have
	\begin{equation}
	\begin{split}
		(L_k)^{(2)}_q(h)
		&=h_k\widetilde{L}_k(j\cdot j^2_d(i) - i\cdot j^2_d(j) + j^2_d(k) + k\cdot j^2_d(1))\\
		&=h_k\left(j\cdot \widetilde{L}_k(j^2_d(i)) - i\cdot \widetilde{L}_k(j^2_d(j)) + \widetilde{L}_k(j^2_d(k)) + k\cdot \widetilde{L}_k(j^2_d(1))\right)\\
		&=h_k\left(jki - ikj + k^2 + k^2\right)\\
		&=-4h_k
	\end{split}
	\end{equation}
	This differential operator can be recognized as $-2 (\partial_i\circ \partial_j - \partial_j\circ \partial_i) = 2[\partial_j, \partial_i]$, cf.\ \cite[\jetseqquaternionpartialderivatives]{FMW}, and hence $(L_k)^{(2)}_q =2\Delta$, cf.\ \cite[\jetseqquaternionLaplacian]{FMW}.
	Next, we compute the $1$-truncation of $L_k$ via the definition as $\trunc{L_k}^1_q = L_k - q^2([L_k]^2_q )$.
	The quantization of its symbol can be evaluated as $(L_k)^{(1)}_q(h) = \widetilde{\trunc{L_k}^1_q} \circ\iota^1_{d,\Hq}(dh)$.
	We convert the differential to a combination of prolongations as $\iota^1_{d,\Hq}(dh) = j^1_d(h)-h\cdot j^1_d(1)$, cf.\ \cite[\jetsrmkproldop]{FMW}.
	We now compute
	\begin{equation}
		(L_k)^{(1)}_q(h)
		= \widetilde{\trunc{L_k}^1_q} (j^1_d(h)-h\cdot j^1_d(1))
		= \trunc{L_k}^1_q(h)-h\trunc{L_k}^1_q(1)
		= (L_k-2\Delta)(h)-h(L_k-2\Delta)(1)
	\end{equation}
	As we noticed before, $2\Delta(h)=-4h_k$, thus giving
	\begin{equation}
		(L_k)^{(1)}_q(h)
		= kh+4h_k-hk+0
		= [k,h]+4h_k
		=0+2h_i j -2 h_j i +4 h_k
	\end{equation}
	This operator coincides with $2(R_j\circ \partial_i - R_i\circ \partial_j)$.
	Finally, the $0$-homogeneous component of $L_k$, which coincides with the $0$-truncation $(L_k)^{(0)}_q=\trunc{L_k}^0_q=L^1_k - q^1([L_k]^1_q )$, can be computed by taking its value at $1 \in \Hq$.
	We obtain $\trunc{L_k}^0_q(1) = k +0 = R_k(1)$, and since the order $0$ term is left $\Hq$-linear, the operators coincide on all $\Hq$.
\end{proof}
Finally, let us describe the star product on $\Symb^\bullet_d (\Hq, \Hq)$.
We will write it in terms of a set of generators, letting $x_i = \symb^0_d(R_i)=R_i$ and $x_j=\symb^0_d(R_j)=R_j$, and $p_i = \symb^1_d(\partial_i)$ and $p_j = \symb^1_d(\partial_j)$, play the r\^oles of generalized positions and momenta.
\begin{prop}
	The star product on $\Symb^\bullet_d(\Hq, \Hq)$ defined by the quantization from Lemma \ref{lemma:quantization_quaternions} is given by
	\begin{align}
		\label{eq:quanturmions}
			&x_a \star x_b = x_a \cdot x_b,
		&\hfill&
			&p_a \star x_b = -x_b \cdot p_a + \hbar\delta_{a,b},
		&\hfill&
			&x_a \star p_b = x_a \cdot p_b,
		&\hfill&
			&p_a \star p_b = p_a \cdot p_b,
	\end{align}
	where $a,b\in \{i,j\}$, and $\delta_{a,b}$ is the Kronecker delta.
	In particular, for all values of $\hbar$ we have that $x_i$ and $x_j$ generate a subalgebra isomorphic to $\Hq^{op}$, and $p_i\star p_j = -p_j\star p_i$, $p_i\star p_i = p_j\star p_j = 0$.
	In other words, the subalgebra generated by $x_i$ and $x_j$ and the one generated by $p_i$ and $p_j$ are independent of $\hbar$.
	The original symbol algebra structure is recovered for $\hbar=0$.
\end{prop}
\begin{proof}
It follows from an explicit computation, once we decompose the generators via \eqref{eq:DO_decomposition}.
\begin{align}
R_i=(R_i)^{(0)}_q,
&\hfill&
R_j=(R_j)^{(0)}_q,
&\hfill&
\partial_i=(\partial_i)^{(1)}_q,
&\hfill&
\partial_j=(\partial_j)^{(1)}_q.
\end{align}
From these equalities we obtain \eqref{eq:quanturmions}.
This, in turn shows that $x_i$, $x_j$, $p_i$, and $p_j$ are also generators for the algebra $(\Symb^\bullet_d(\Hq,\Hq),\star)$.
The final statements can be recovered from \eqref{eq:quanturmions} and Proposition \ref{prop:special_deformations}.\eqref{prop:special_deformations:1}.
\end{proof}
Notice that \eqref{eq:quanturmions} yield relations that can be interpreted as the quaternionic analogue of the canonical commutation relations.
\bibliography{../Bibliography}

\newcommand{\etalchar}[1]{$^{#1}$}
\begin{thebibliography}{BFF{\etalchar{+}}78}

\bibitem[AE05]{quantizationguide}
S.~T. Ali and M.~Engli{\v{s}}.
\newblock Quantization methods: a guide for physicists and analysts.
\newblock {\em Reviews in Mathematical Physics}, 17(04):391--490, 2005.

\bibitem[And92]{anderson1992introduction}
I.~M. Anderson.
\newblock Introduction to the variational bicomplex.
\newblock {\em Contemp. Math.}, 132(51), 1992.

\bibitem[AVL91]{alekseevskij28geometry}
D.~V. Alekseevskij, A.~M. Vinogradov, and V.~V. Lychagin.
\newblock {\em Geometry I, Basic Ideas and Concepts of Differential Geometry,
  Encycl}, volume~28 of {\em Math. Sci}.
\newblock Springer-Verlag, 1991.

\bibitem[BFF{\etalchar{+}}78]{originalstarproduct}
F.~Bayen, M.~Flato, C.~Fronsdal, A.~Lichnerowicz, and D.~Sternheimer.
\newblock Deformation theory and quantization. i. deformations of symplectic
  structures.
\newblock {\em Annals of Physics}, 111(1):61--110, 1978.

\bibitem[BM20]{BeggsMajid}
E.~J. Beggs and S.~Majid.
\newblock {\em Quantum Riemannian Geometry}.
\newblock Springer International Publishing, 2020.

\bibitem[DF{\etalchar{+}}04]{dummit2004abstract}
D.~S. Dummit, R.~M. Foote, et~al.
\newblock {\em Abstract algebra}, volume~3.
\newblock Wiley Hoboken, 2004.

\bibitem[Eas09]{eastwood2009higher}
M.~G. Eastwood.
\newblock Higher order connections.
\newblock {\em SIGMA. Symmetry, Integrability and Geometry: Methods and
  Applications}, 5:082, 2009.

\bibitem[Ehr56]{Ehresmannconnectionsdordresup}
C.~Ehresmann.
\newblock Connexions d’ordre sup\'{e}rieur.
\newblock {\em Atti Quinto Congr. Un. Mat. Ital. (Pavia-Torino, 1955)}, 1956.
\newblock 326--328.

\bibitem[FMW22]{FMW}
K.~J. Flood, M.~Mantegazza, and H.~Winther.
\newblock Jet functors in noncommutative geometry.
\newblock {\em arXiv preprint arXiv:2204.12401v3}, 2022.

\bibitem[FMW23]{Symbol}
K.~J. Flood, M.~Mantegazza, and H.~Winther.
\newblock Symbols in noncommutative geometry.
\newblock {\em arXiv preprint arXiv:2308.00835v2}, 2023.

\bibitem[Gol67]{goldschmidt1967existence}
H.~Goldschmidt.
\newblock Existence theorems for analytic linear partial differential
  equations.
\newblock {\em Annals of Mathematics}, pages 246--270, 1967.

\bibitem[Gro46]{groenewold1946}
H.~J. Groenewold.
\newblock {\em On the principles of elementary quantum mechanics}.
\newblock Springer, 1946.

\bibitem[GS94]{microlocaldiffop}
A.~Grigis and J.~Sj\"ostrand.
\newblock {\em Microlocal analysis for differential operators}, volume 196 of
  {\em London Mathematical Society Lecture Note Series}.
\newblock Cambridge University Press, Cambridge, 1994.
\newblock An introduction.

\bibitem[Hin25]{Hintz}
P.~Hintz.
\newblock {\em An Introduction to Microlocal Analysis}.
\newblock Graduate Texts in Mathematics. Springer Cham, Cham, 2025.

\bibitem[JR04]{HolomorphicHigherCon}
P.~Jahnke and I.~Radloff.
\newblock Splitting jet sequences.
\newblock {\em Math. Res. Lett.}, 11(2-3):345--354, 2004.

\bibitem[KMS13]{NaturalOperations}
I.~Kol{\'a}r, P.~W. Michor, and J.~Slov{\'a}k.
\newblock {\em Natural operations in differential geometry}.
\newblock Springer Science \& Business Media, 2013.

\bibitem[Lib64]{libermann1964}
P.~Libermann.
\newblock Sur la g{\'e}om{\'e}trie des prolongements des espaces fibr{\'e}s
  vectoriels.
\newblock {\em Annales de l'institut Fourier}, 14(1):145--172, 1964.

\bibitem[Lib97]{libermann1997}
P.~Libermann.
\newblock Introduction to the theory of semi-holonomic jets.
\newblock {\em Archivum Mathematicum}, 33(2):173--189, 1997.

\bibitem[Lyc99]{Lychagin1999}
V.~Lychagin.
\newblock Quantum mechanics on manifolds.
\newblock {\em Acta Applicandae Mathematica}, 56:231--251, 1999.

\bibitem[MH16]{variationsjets}
J.~Musilov{\'a} and S.~Hronek.
\newblock The calculus of variations on jet bundles as a universal approach for
  a variational formulation of fundamental physical theories.
\newblock {\em Communications in mathematics}, 24, 2016.

\bibitem[Moy49]{moyal1949quantum}
J.~E. Moyal.
\newblock Quantum mechanics as a statistical theory.
\newblock {\em Mathematical Proceedings of the Cambridge Philosophical
  Society}, 45(1):99--124, 1949.

\bibitem[Nes20]{nestruev2020smooth}
J.~Nestruev.
\newblock {\em Smooth Manifolds and Observables}, volume 220.
\newblock Springer Nature, 2020.

\bibitem[Pal65]{PalaisVectorBundles}
R.~S. Palais.
\newblock Differential operators on vector bundles.
\newblock In {\em Seminar on the Atiyah-Singer index theorem}, pages 51--93,
  1965.

\bibitem[Pom12]{Pommaret}
J.~F. Pommaret.
\newblock Spencer operator and applications: From continuum mechanics to
  mathematical physics".
\newblock In Yong~X. Gan, editor, {\em Continuum Mechanics}, chapter~1.
  IntechOpen, Rijeka, 2012.

\bibitem[Seg60]{SegalQuantization}
I.~E. Segal.
\newblock Quantization of nonlinear systems.
\newblock {\em Journal of Mathematical Physics}, 1(6):468--488, 1960.

\bibitem[Spe69]{Spencer}
D.~C. Spencer.
\newblock Overdetermined systems of linear partial differential equations.
\newblock {\em Bulletin of the American Mathematical Society}, 75(2):179--239,
  1969.

\bibitem[Ste98]{sternheimer1998deformation}
D.~Sternheimer.
\newblock Deformation quantization: Twenty years after.
\newblock {\em AIP Conference Proceedings}, 453(1):107--145, 1998.

\bibitem[Vir67]{virsik1967non}
J.~Virs{\'\i}k.
\newblock Non-holonomic connections on vector bundles, ii.
\newblock {\em Czechoslovak Mathematical Journal}, 17(2):200--224, 1967.

\bibitem[Yue71]{yuen1971higher}
P.~C. Yuen.
\newblock Higher order frames and linear connections.
\newblock {\em Cahiers de topologie et geometrie differentielle},
  12(3):333--371, 1971.

\bibitem[Zwo12]{semiclassical}
M.~Zworski.
\newblock {\em Semiclassical analysis}, volume 138 of {\em Graduate Studies in
  Mathematics}.
\newblock American Mathematical Society, Providence, RI, 2012.

\end{thebibliography}
\bibliographystyle{alpha}
\end{document}